\documentclass[11pt]{amsart}
\usepackage{amsmath, amsthm, amssymb}
\usepackage{amsfonts}
\usepackage{mathtools}

\usepackage[cal=boondox]{mathalfa}
\usepackage{bbm}
\usepackage{mathrsfs}
\usepackage{stmaryrd}
\usepackage{esint}
\usepackage[T1]{fontenc}
\usepackage{lmodern}

\usepackage{graphicx}
\usepackage{tikz}
\usepackage{tikz-cd}
\usetikzlibrary{arrows}
\usepackage[all,cmtip]{xy}
\usepackage{amscd}
\usepackage{picinpar}
\usepackage{float}
\usepackage{caption}

\usepackage{enumerate}
\usepackage{enumitem}
\usepackage{comment}
\usepackage{array}
\usepackage{extarrows}
\usepackage{verbatim}
\usepackage{calc}
\usepackage{xcolor}

\usepackage{hyperref}

\numberwithin{equation}{section}

\theoremstyle{plain}
\newtheorem{thm}{Theorem}
\numberwithin{thm}{subsection}
\newtheorem{cor}[thm]{Corollary}
\newtheorem{lem}[thm]{Lemma}
\newtheorem{prop}[thm]{Proposition}

\newtheorem{mainthm}{Theorem}

\theoremstyle{definition}
\newtheorem{defi}[thm]{Definition}

\theoremstyle{remark}
\newtheorem{rem}[thm]{Remark}

\DeclareMathOperator{\Id}{Id}
\DeclareMathOperator{\coh}{H}

\DeclareMathOperator{\pr}{pr}

\DeclareMathOperator{\Sing}{Sing}

\DeclareMathOperator{\Spec}{Spec}

\DeclareMathOperator{\Crit}{Crit}

\def\bC{{\mathbb{C}}}
\def\bD{{\mathbb{D}}}
\def\bR{{\mathbb{R}}}

\def\bF{{\mathbb{F}}}

\def\bZ{{\mathbb{Z}}}

\def\bP{{\mathbb{P}}}
\def\bN{{\mathbb{N}}}

\def\cH{{\mathscr{H}}}

\def\cK{{\mathscr{K}}}

\def\cM{{\mathscr{M}}}

\def\cO{{\mathscr{O}}}

\def\cR{{\mathscr{R}}}

\def\fm{{\mathfrak{m}}}

\title{Postcritically finite endomorphisms II: algebraic core and Fatou sets}

\author{Ruiran Sun}
 \address{School of Mathematical Sciences, Xiamen University, Xiamen 361005, China}
\email{ruiransun@xmu.edu.cn}
\author{Shengyuan Zhao}
\address{Universit\'e Paul Sabatier, Institut de Math\'ematiques de Toulouse, 118, route de Narbonne, F-31062 Toulouse, France}
\email{shengyuan.zhao@math.univ-toulouse.fr}

\begin{document}
\begin{abstract}
Let $f$ be a postcritically finite endomorphism of the complex projective space $\mathbb{P}^N$ of degree at least two. We prove that at every periodic point of $f$, the non-superattracting tangent directions span the tangent space of a periodic projective subvariety which is rationally connected. This phenomenon strongly constrains both the dynamics and the geometry of postcritically finite endomorphisms. In particular, we show that every Fatou component of $f$ is preperiodic and is eventually mapped into the basin of a superattracting cycle.
\end{abstract}
\maketitle

\tableofcontents

\section{Introduction}
Let $N\geq 1$. Let $f:\bP^N\to\bP^N$ be a holomorphic endomorphism of degree larger than one. Its critical locus and postcritical set are
\[
\Crit(f):=\{x\in\bP^N\mid\operatorname{rank}(d_xf)<N\}, \quad P_f:=\bigcup_{n\geq1}f^n(\Crit(f)).
\]
We say that $f$ is \emph{postcritically finite}, or \emph{PCF}, if $P_f$ is algebraic, i.e.\ every irreducible component of the divisor $\Crit(f)$ is preperiodic. Examples include the coordinate power maps
\[
[z_0:\cdots:z_N]\longmapsto[z_0^d:\cdots:z_N^d],
\qquad d\geq 2,
\]
whose postcritical divisors are the coordinate hyperplanes. Further higher dimensional examples include symmetric products or suspensions of lower dimensional PCF maps (see \cite{Ueda1993,GauthierHutzKaschner2023}), Lattès maps (see \cite{BertelootLoeb2001,Dupont2003,Rong2010Lattes}), certain maps preserving fibrations or webs (\cite{Rong2008,DabijaJonsson2010}), a few examples by Crass~\cite{Crass2007Klein,Crass2011Valentiner} and a construction by Koch~\cite{Koch2013} based on Teichmüller theory that englobes many other special examples (see \cite{FornaessSibony1992,Ueno2007Symmetric,BelkKoch2010,BuffEpsteinKoch2012,Ingram2020MinimallyCritical}). We present also two new examples in this paper\footnote{The authors used AI to find a few other new examples, see our personal pages}. Gauthier--Taflin--Vigny~\cite{GTV} proved that in the space of all endomorphisms of a fixed degree the PCF ones do not form a Zariski-dense subset. 

The \emph{Fatou set} is the largest open subset of $\bP^N$ on which the iterates $(f^n)$ form a normal family. Its connected components are called the \emph{Fatou components}. Its complement is the \emph{Julia set}. A Fatou component is \emph{preperiodic} if some forward image is periodic, and is \emph{wandering} otherwise. If a periodic point $p$ has minimal period $r$, then its \emph{multipliers} are the eigenvalues of $d_p(f^r)$. We call $p$ or the orbit of $p$ \emph{superattracting} if this differential is nilpotent, or equivalently if all its multipliers vanish. 

For a superattracting cycle $\mathcal O=\{p,f(p),\ldots,f^{r-1}(p)\}$, its \emph{basin} is
\[
\mathcal B(\mathcal O)
:=\{z\in\bP^N\mid\operatorname{dist}(f^n(z),\mathcal O)\to 0\},
\]
where the distance is taken with respect to the Fubini--Study metric. The \emph{immediate basin} of $p$ is the connected component of $\mathcal B(\mathcal O)$ containing $p$.

It is a classical fact in one-dimensional complex dynamics that \emph{the Fatou set of a PCF rational function consists entirely of basins of superattracting cycles} (see \cite[Cor.~16.5]{Milnor2006}). We prove that the same description of the Fatou set holds in every dimension:
\begin{mainthm}\label{mainthm:Fatou}
Let $f:\bP^N\to\bP^N$ be a PCF endomorphism of degree at least two. Then every Fatou component of $f$ is contained in the basin of a cycle of superattracting periodic points. Every Fatou component is preperiodic to the immediate basin of that cycle.
\end{mainthm}

In particular, \emph{a PCF endomorphism of projective space has no wandering Fatou components}. Its Fatou set is empty if and only if it has no superattracting periodic point.

The higher-dimensional problem has a long history. Forn\ae ss and Sibony studied PCF maps on $\bP^2$ and established the description of their Fatou components under additional hypotheses on the critical orbits and on hyperbolicity~\cite[Thms.~7.7 and~7.8]{FornaessSibony}. Jonsson~\cite{Jonsson1998} studied some recursive conditions of PCF maps. Ueda~\cite{Ueda1998} introduced a general method based on normality of inverse lifts. Among its consequences are the absence of rotation domains for PCF endomorphisms and the containment of every Fatou limit image in the eventual postcritical set $\bigcap_{n\geq0}f^n(P_f)$. 
Here by a \emph{Fatou limit map} we mean a locally uniform limit of a sequence $f^{n_j}$ on a Fatou component such that $n_j\to\infty$. A \emph{rotation domain} is a Fatou component on which the identity map is such a limit. Combining the work of Forn\ae ss--Sibony and Ueda, Rong~\cite[Thm.~1.1]{Rong2008} proved Theorem~\ref{mainthm:Fatou} for $N=2$. In dimension two Zhuchao Ji~\cite{Ji} obtained finer information on one-dimensional Fatou directions and on Julia sets.

In arbitrary dimension, Astorg~\cite[Thm.~C]{Astorg2020} obtained Theorem~\ref{mainthm:Fatou} when the divisor $P_f$ satisfies some transversality conditions and $\bP^N\setminus P_f$ is Kobayashi hyperbolic. In our previous work~\cite{SunZhao} we studied the Kobayashi hyperbolic condition in details. Under Astorg's conditions Van Tu Le~\cite{Le2020} proved that every multiplier of a PCF endomorphism of $\bP^2$ is either zero or has modulus greater than one, and subsequently generalized this to higher dimension~\cite{Le2023}. The superattracting-or-repelling dichotomy for periodic points arising in the work of Astorg and Le was a major inspiration for the present work (see Proposition~\ref{prop:replling-direction}), though our method is quite different. However some examples that we present in Section~\ref{sec:geometry} show that Astorg's conditions, which hold for Koch's constructions, do not hold for all PCF endomorphisms.

Our proof strategy consists of establishing several comparison results in Serre's GAGA-style between dynamically defined a priori transcendental objects and algebraic varieties. The first one is based on Ueda's pioneer work \cite{Ueda1998}. By using Ueda's results, we show that the image of a rank-$r$ limit map on a Fatou component has Zariski closure of dimension precisely $r$. This algebraic control allows us to produce a periodic constant limit map on every Fatou component. Then the difficulty is to show that such a constant limit is superattracting. The key is to show that,  at a periodic point, the invariant manfiold corresponding to non superattracting directions in the sense of differential dynamical systems, which exists in more general situations but is usually highly transcendental, is in fact algebraic in the PCF setting. We believe that this new structural phenomenon will play fundamental roles for further studies on PCF endomorphisms, especially for an eventual classification of such maps. 

\begin{mainthm}\label{mainthm:algebraic-cores}
Let $f:\bP^N_{\bC}\to\bP^N_{\bC}$, $N\geq 1$, be a PCF endomorphism
of degree at least two. Let $p\in \bP^N$ be a period point of minimal period $a$. Let $E_p\subset T_p\bP^N$ be the sum of the generalized eigenspaces of $d_p(f^{a})$ for nonzero eigenvalues.  Denote $r:=\dim E_p$. Then there is an irreducible projective subvariety $W\subset\bP^N$ of dimension $r$ containing $p$ and a positive multiple $m$ of $a$ such that the iterate $F=f^m$ satisfies the following:
\begin{enumerate}[label=\textup{(\roman*)},ref=\roman*,leftmargin=*]
\item\label{maincore:existence} (Theorem~\ref{thm:algebraic-core} and Theorem~\ref{thm:core-smoothness})
One has $F(W)=W$. The variety $W$ is smooth at $p$ and the tangent space satisfies $T_pW=E$. The map $F$ induces an automorphism of the germ $(W,p)$. 

\item\label{maincore:normalization}(Lemma~\ref{lem:normalization-boundary})
Let $\nu:X\to W$ be the normalization. The induced endomorphism $g:X\to X$ is postcritically finite in the following sense. There is a proper algebraic subset $B\subset X$ containing $\Sing X$ and
satisfying $g(B)\subset B$ such that, for every $n\geq1$, the map
$g^n:(g^n)^{-1}(X\setminus B)\to X\setminus B$ is a finite unramified covering.

\item\label{maincore:transverse}(Theorem~\ref{thm:ideal-square} and Theorem~\ref{thm:transverse})
The map $F_p$ is transversely superattracting along $W$ in the
following sense: if $\mathcal I_W\subset\cO_{\bP^N}$ denotes
the ideal sheaf of $W$, then $F^*\mathcal I_W\subset\mathcal I_W^2$. If $r<N$ then $W$ is contained in the postcritical locus $P_f$. 

\item\label{maincore:geometry}(Theorem~\ref{thm:rc-core} and Example \S~\ref{subsec:non-rational-core})
The variety $W$ is rationally connected, but not necessarily rational.
\end{enumerate}
\end{mainthm}

\begin{defi}
In the setup of Theorem~\ref{mainthm:algebraic-cores}, we call $W$ the \emph{algebraic core} of the periodic point $p$. 
\end{defi}

As we mentionned earlier the algebraic core $W$ is the invariant manifold at $p$ tangent to non-superattracting directions, which a priori is only holomorphic. The major point of Theorem~\ref{mainthm:algebraic-cores} consists of proving its algebraicity. This algebraization result is quite different from several classical algebraicity criterion, like GAGA or the algebraicity criteria for formal germs due to Hartshorne \cite{hart68} and Bost \cite{bost01}. Instead of using positivity to kill higher order jets of the germ, we construct many dynamically defined analytic arcs to force its defining equations to vanish to arbitrarily high order along these arcs. Then the same method allows us to prove the transverse superattracting property along $W$, if we apply it not to $p$ but to the generic point of $W$ in the sense of scheme theory. More precisely, we perform the same study of jets for some so called "twisted dynamics", i.e.\ local dynamics over the function field of $W$ with an induced action also on this coefficient function field. Then a classical theorem of Ein–Lazarsfeld–Smith on ideal sheaves allows us to obtain the concise conclusion in Theorem~\ref{mainthm:algebraic-cores}~(\ref{maincore:transverse}). Note that rationally connectedness of algebraic core is an analogue of a Theorem of Bogomolov--McQuillan for invariant subvarieties of foliations; the proof is based on Bogomolov--McQuillan's theorem as well.

The paper is organized as follows. In Section~\ref{sec:preparation} we review Ueda's work and deduce the algebraic control of Fatou limit maps. We construct the algebraic core in Section~\ref{sec:local-invariant} and prove a superattracting-or-repelling dichotomy. In Section~\ref{sec:transverse} we study transverse superattraction along the algebraic core by continuing the techniques started in Section~\ref{sec:local-invariant}. Then in Section~\ref{sec:ideal-sheaves} we use the formalism of ideal sheaves to formulate transverse superattraction, and prove the local smoothness of the algebraic core. The proof of the main theorem on Fatou components is completed in Section~\ref{sec:convergence}. In Section~\ref{sec:geometry} we present two new examples of PCF endomorphisms exhibing new phenomenons, and prove the rationally connectedness of algebraic cores. The first appendix collects some technical lemmas in complex analysis used in our paper. The second appendix is an application of our results to a conjectured desciption of Julia sets in dimension two; it is provided by Zhuchao Ji.

\section{Preparation}\label{sec:preparation}

\subsection{Notations and terminology}\label{subsec:notations}
For a reduced complex space $X$, we denote by $\Sing X$ its singular locus and define $X_{\rm reg}:=X\setminus\Sing X$.

For a finite morphism $G$ between smooth complex spaces of the same dimension, let $\Crit(G)$ be the locus where its differential is singular, and let $\operatorname{Br}(G):=G(\Crit(G))$ be its branch locus.

For a finite morphism $G:X\to Y$ between normal spaces, its branch locus $\operatorname{Br}(G)$ is by definition the set of points $y\in Y$ such that $G$ is not locally biholomorphic at some point of $G^{-1}(y)$.

We denote by $\deg$ the topological degree of a finite surjective map. Thus $\deg f=d_f^N$ for an endomorphism $f:\bP^N\to \bP^N$ where $d_f$ is the algebraic degree of $f$, i.e.\ the degree of the homogeneous polynomials defining $f$.
We call a point $x\in \bP^N$ a \emph{Fatou limit point} of $f$ if there exist a Fatou component $\Omega$ and a locally uniform limit $h$ of $f^{n_j}\vert_\Omega$ for some sequence $n_j\to\infty$ such that $x\in h(\Omega)$.

\subsection{Fatou limit maps}\label{subsec:limits-iterates}

From now on let $N\geq 1$ and let $f:\bP^N\to\bP^N$ be a PCF
endomorphism with $d_f\geq 2$. We denote $U_f:=\bP^N\setminus P_f$. For every $n\in \bN^*$, we have un unramified covering $f^n:f^{-n}(U_f)\to U_f$. For every $m\in \bN^*$, the map $f^m$ is PCF, $P_{f^m}=P_f$, and the Fatou sets of $f$ and $f^m$ coincide.

Throughout the paper we use $\Omega$ to denote a Fatou component of $f$. Let $\cH_\Omega$ be the set of locally uniform limits of $f^{n_j}\vert_\Omega$ with $n_j\to +\infty$. By definition of the Fatou set, the set $\cH_\Omega$ is nonempty and compact in the compact-open topology. It is forward invariant under the composition $H\mapsto f\circ H$.

For a positive iterate $F$ of $f$, we denote by $\cH_\Omega(F)$ the set of locally uniform limits of $F^{n_j}\vert_\Omega$ with $n_j\to+\infty$. We have $\cH_\Omega=\cH_\Omega(f)$, and $\cH_\Omega(F)\subset\cH_\Omega$. 

The following lemma is a slight modification of \cite[Thm.~2.1]{Ueda1998} where $(\psi_j)$ is a constant sequence. The same proof works.
\begin{lem}[Ueda]\label{lem:inverse-normality}
Let $Z$ be a complex manifold, let $\phi_j,\psi_j:Z\to\bP^N$ be holomorphic maps, and let $n_j\geq 0$ be integers such that $f^{n_j}\circ \phi_j=\psi_j$. If the sequence of maps $(\psi_j)$ is normal, then the sequence $(\phi_j)$ is normal.
\end{lem}

The following fact is the main starting point of our work.
\begin{prop}[Ueda~{\cite[Prop.~5.1]{Ueda1998}}]\label{prop:Fatou-limit-point}
Every Fatou limit point belongs to $\bigcap_{n\geq0}f^n(P_f)$.
\end{prop}

In the sequel we will fix a log resolution of the pair $(\bP^N,P_f)$, i.e.\ a smooth projective variety $Y$ together with a morphism $\pi:Y\to\bP^N$ such that
\begin{enumerate}
\item $\pi$ is an isomorphism over $U_f\cup(P_f)_{\rm reg}=\bP^N\setminus\operatorname{Sing}(P_f)$;
\item the reduced total transform $D:=\pi^{-1}(P_f)_{\rm red}$ is a simple normal crossing divisor on $Y$;
\item the strict transform $\widetilde P_f:=\overline{\pi^{-1}\bigl((P_f)_{\rm reg}\bigr)}$ is smooth.
\end{enumerate}
Such a log resolution exists by Hironaka's theorem~\cite{Hironaka}. 
Let $D_1,\ldots,D_l$ be the irreducible components of $D$. Let $Y_1,\ldots,Y_s$ be the nonempty closed subsets of $Y$ of the form
$D_{i_1}\cap\cdots\cap D_{i_k}$, where $1\leq k\leq l$ and
$1\leq i_1<\cdots<i_k\leq l$. Since $D$ has simple normal crossings, each $Y_j$ is smooth and projective, but need not be connected. We define
\begin{equation}\label{eq:resolution-exceptional-set}
E_\pi:=\bigcup_{i=1}^s
\pi\bigl(\{y\in Y_i\mid d_y(\pi\vert_{Y_i})=0\}\bigr).
\end{equation}
For $0\leq r\leq N-1$, we also define
\begin{equation}\label{eq:resolution-rank-loci}
E_{\pi,r}:=\bigcup_{i=1}^s
\pi\bigl(\{y\in Y_i\mid \operatorname{rank}d_y(\pi\vert_{Y_i})\leq r\}\bigr).
\end{equation}
Note that $E_{\pi,0}=E_\pi$.

Recall that the generic rank of a holomorphic map on a connected complex manifold is the maximum rank of its differential, which is realized on a dense subset.

\begin{prop}\label{prop:log-rank}
The set $E_\pi$ is finite. Every constant map in $\cH_\Omega$ has value in $E_\pi$. More generally, the set $E_{\pi,r}$ is algebraic and has dimension at most $r$. Every map in $\cH_\Omega$ of generic rank at most $r$ has image in $E_{\pi,r}$. If $H\in\cH_\Omega$ has generic rank $r$, then the Zariski closure of $H(\Omega)$ is an irreducible component of $E_{\pi,r}$ of dimension $r$, and has finite forward orbit under $f$.
\end{prop}
\begin{proof}
For each $r$, the locus 
\[
C_{i,r}:=\{y\in Y_i\mid \operatorname{rank}d_y(\pi\vert_{Y_i})\leq r\}
\]
is algebraic. On the smooth locus of every irreducible component of $C_{i,r}$, the restriction of $\pi$ has rank at most $r$. Its image has dimension at most $r$. Since $\pi$ is proper, the set $E_{\pi,r}$ is algebraic and has dimension at most $r$. In particular, $E_\pi$ is finite.

Let $H$ be a locally uniform limit of $(f^{n_j}\vert_\Omega)$, with $n_j\to\infty$. The points in $\Omega$ whose forward $f$-orbit avoids $P_f$ are dense, since every $f^{-n}(P_f)$ is a proper algebraic hypersurface. Let $z\in\Omega$ be such a point.
By compactness of $Y$, up to passing to a subsequence, we may
assume that $y_j=\pi^{-1}(f^{n_j}(z))$ converges to a point $y\in Y$.
Then $\pi(y)=H(z)\in P_f$, so $y\in D$.

Let $T$ be the intersection of all irreducible components of $D$ passing through $y$, and let $r:=\operatorname{rank}d_y(\pi\vert_T)$. Then $T$ is one of the varieties $Y_i$.
Suppose $r>0$. We can find holomorphic maps $\sigma_j:\bD^r\to Y\setminus D$ defined on a polydisk satisfying $\sigma_j(0)=y_j$ for all sufficiently large $j$, converging locally uniformly to a holomorphic map $\sigma:\bD^r\to T$ with $\sigma(0)=y$ and $\operatorname{rank}(\mathrm d_0(\pi\circ\sigma))=r$.
Recall that each $f^{n_j}:f^{-n_j}(U_f)\to U_f$ is a unramified covering. Since $\bD^r$ is simply connected, there is a holomorphic
lift $\tau_j$ with $\tau_j(0)=z$ making the following diagram commute:
\begin{equation}\label{eq:lifted-disc-identity}
\begin{tikzcd}
\bD^r \arrow[r, "\sigma_j"] \arrow[d,  "\tau_j"']
  & Y\setminus D \arrow[d, "\pi"] \\
f^{-n_j}(U_f) \arrow[r, "f^{n_j}"']
  & U_f.
\end{tikzcd}
\end{equation}
The maps $\pi\circ\sigma_j$ converge locally uniformly.
By Lemma~\ref{lem:inverse-normality}, up to passing to a subsequence, we may assume that $\tau_j$ converges locally uniformly to a holomorphic map $\tau$. Differentiating the commutativity relation
in \eqref{eq:lifted-disc-identity} at zero and passing to the limit,
we get
\[
\mathrm d_zH\circ\mathrm d_0\tau=\mathrm d_0(\pi\circ\sigma)
\]
because $\tau_j(0)=z$. Hence $\operatorname{rank}(\mathrm d_zH)\geq r$.

Suppose now that $H$ has generic rank at most $q$. Then $r\leq\operatorname{rank}(\mathrm d_zH)\leq q$ by the above argument. Hence $H(z)\in E_{\pi,q}$. Recall that $z$ is an arbitrary point in $\Omega$ whose forward orbit avoids $P_f$. Since these points are dense in $\Omega$ and $E_{\pi,q}$ is closed, we conclude that $H(\Omega)\subset E_{\pi,q}$.

If $H$ has generic rank $q$, then the Zariski closure of its image is an irreducible component $X$ of $E_{\pi,q}$ of dimension $q$ because $\Omega$ is connected. Since $f$ is finite, each $f^n\circ H$ still has generic rank $q$ and belongs to $\cH_\Omega$. Therefore every $f^n(X)$ is an irreducible component of $E_{\pi,q}$ of dimension $q$. There are only finitely many such components, so the forward orbit of $X$ is finite.
\end{proof}

\subsection{Inverse germs}\label{subsec:inverse-germs}

Let $X\subset\bP^N$ be an irreducible algebraic subvariety.
Recall that we fixed the log resolution $\pi:Y\to\bP^N$ and the divisor $D=\pi^{-1}(P_f)_{\mathrm{red}}$. For an integer $n\geq1$ and an open subset $B\subset X_{\rm reg}$, a holomorphic section of $f^n$ over $B$ will mean a holomorphic map $\beta:B\to\bP^N$ such that $f^n(\beta(x))=x$ for every $x\in B$.

\begin{lem}\label{lem:inverse-extension}
Let $X\subset\bP^N$ be an irreducible algebraic subvariety.
There exists a connected dense Zariski-open subset
$X^\circ\subset X_{\rm reg}$ with the following property.
Every point of $X^\circ$ has a neighborhood $B\subset X^\circ$
in the analytic topology such that, for every integer $n\geq1$,
the inverse image $f^{-n}(B)$ is a finite union of images of
holomorphic sections of $f^n$ over $B$.
\end{lem}

\begin{proof}
Let $C$ be an irreducible component of $\pi^{-1}(X)$ dominating $X$. 
There is a dense Zariski-open subset of $C_{\rm reg}$ on which $\pi\vert_C$ takes values in $X_{\rm reg}$ and has surjective differential. We remove from this subset its intersections with the
components of $D$ that do not contain $C$. 
The remaining subset, denoted $C^\circ$, is still dense and Zariski open in $C$. Its image under $\pi$ thus contains a dense Zariski-open subset $X^\circ\subset X_{\rm reg}$. 

Let $x\in X^\circ$, and let $y\in \pi^{-1}(x)\subset C^\circ$. 
There is a small polydisk $P\subset Y$ around $y$, with coordinates $u_1,\ldots,u_N$, such that $D\cap P=\{u_1\cdots u_k=0\}$ for some $0\leq k\leq N$ because $D$ has simple normal crossing. We may assume also that every component of $D$ meeting $P$ contains $C$. By the holomorphic implicit function theorem, there are a neiborhodd $B\subset X^\circ$ of $x$ and a holomorphic map $s:B\to C^\circ\cap P$ such that $\pi(s(z))=z$ for every $z\in B$. Since every component of $D$ meeting $P$ contains $C$, we have $s(B)\subset\{u_1=\cdots=u_k=0\}$. 

Let $n\geq 1$, and let $W_n=P\times_{\bP^N,f^n}\bP^N$ be the fiber product, where the maps to $\bP^N$ are respectively $\pi\vert_P$ and $f^n$. In other words the points of $W_n$ are the pairs $(u,v)\in P\times\bP^N$ satisfying $\pi(u)=f^n(v)$. Let $p_n:W_n\to P$ and $q_n:W_n\to\bP^N$ be the projections. These projections fit into the commutative diagram
\begin{equation}\label{eq:inverse-extension-projections}
\begin{tikzcd}
W_n \arrow[r, "q_n"] \arrow[d, "p_n"']
  & \bP^N \arrow[d, "f^n"] \\
P \arrow[r, "\pi\vert_P"']
  & \bP^N.
\end{tikzcd}
\end{equation}
Note that every irreducible component of $W_n$ dominates $P$. 
Let $\nu_n:Z_n\to W_n$ be the normalization of $W_n$. Define $\rho_n:=p_n\circ\nu_n:Z_n\to P$, which is a finite surjective map. Over $P\setminus D$, the map $p_n$ is an unramified covering because $f^n$ is unramified over $\bP^N\setminus P_f$. As the space $W_n$ is smooth over $P\setminus D$, and $\nu_n$ is an isomorphism there, $\rho_n$ is unramified over $P\setminus D$.

By Lemma~\ref{lem:local-power-coordinates}, the inverse image of $\{u_1=\cdots=u_k=0\}$ under $\rho_n$ is a finite disjoint union of copies of this coordinate subspace, each mapped isomorphically onto it by $\rho_n$. For each copy, let $\widetilde s:B\to Z_n$ be the
composition of $s$ with the inverse of the restriction of $\rho_n$ to that copy. Then $\rho_n\circ\widetilde s=s$. 
Define $\gamma=q_n\circ\nu_n\circ\widetilde s:B\to\bP^N$.
These maps fit into the commutative diagram
\[
\begin{tikzcd}
B \arrow[r, "\widetilde s"] \arrow[dr, "s"']
  \arrow[rrr, bend left=30, "\gamma"]
  & Z_n \arrow[r, "\nu_n"] \arrow[d, "\rho_n"]
  & W_n \arrow[r, "q_n"] \arrow[dl, "p_n"]
  & \bP^N \arrow[d, "f^n"] \\
  & P \arrow[rr, "\pi\vert_P"']
  & & \bP^N.
\end{tikzcd}
\]
By \eqref{eq:inverse-extension-projections}, we obtain $f^n\circ\gamma=\pi\circ\rho_n\circ\widetilde s =\pi\circ s$. Thus $f^n(\gamma(z))=z$ for every $z\in B$.

Let $z\in B$ and $v\in f^{-n}(z)$. Since $\pi(s(z))=z=f^n(v)$, the pair $(s(z),v)$ belongs to $W_n$. The map $\nu_n$ being surjective, there exists $\xi\in Z_n$ such that $\nu_n(\xi)=(s(z),v)$. We then have $\rho_n(\xi)=s(z)$. Hence $\xi$ lies in one of the copies of
$\{u_1=\cdots=u_k=0\}$, and the corresponding map $\widetilde s$ satisfies $\widetilde s(z)=\xi$. For its associated section $\gamma$, we therefore have $\gamma(z)=q_n(\nu_n(\xi))=v$. We deduce that the images of the sections cover $f^{-n}(B)$. The same neiborhood $B$ works for every $n\geq 1$.
\end{proof}

\begin{lem}\label{lem:normalization-boundary}
Let $Z$ be any irreducible algebraic subvariety of $\bP^N$ satisfying $f(Z)=Z$, and let $\nu:X\to Z$ be its normalization. 
Then $f$ induces a finite surjective holomorphic map $g:X\to X$ satisfying $\nu\circ g=f\circ\nu$. There is a proper forward-invariant algebraic subset $B\subset X$ containing $\Sing X$ and $\operatorname{Br}(g^n)$ for every $n\geq1$.
\end{lem}
\begin{proof}
The existence of $g$ follows from the universal property of normalization. 

We apply Lemma~\ref{lem:inverse-extension} to $Z$, and obtain a dense Zariski-open subset $U\subset Z_{\rm reg}$ with the following
property: every point of $U$ has a connected analytic neighborhood
$V\subset U$ such that, for every $n\geq1$, the set $f^{-n}(V)$
is the union of the images of finitely many holomorphic sections
of $f^n$ over $V$.

Fix such a neighborhood $V$ and an integer $n\geq 1$. Every irreducible component of $Z\cap f^{-n}(V)$ is contained in one of the section images and has the same dimension, so they are equal. Consequently, the normalization of $Z\cap f^{-n}(V)$ is simply a disjoint union of these section images. Note that this normalization is $(g^n)^{-1}(\nu^{-1}(V))$, and that $g^n$ maps each of its components biholomorphically onto $\nu^{-1}(V)$. Thus $\operatorname{Br}(g^n)\subset X\setminus\nu^{-1}(U)$ for every $n\geq 1$. We now define
\[
B:=\Sing X\cup \overline{\bigcup_{n\geq1}\operatorname{Br}(g^n)}^{\mathrm{Zar}}.
\]
Since $\nu$ is an isomorphism over $U\subset Z_{\rm reg}$, the algebraic subset $B$ is contained in $X\setminus\nu^{-1}(U)$. 
Note that $g(\operatorname{Br}(g^n))\subset\operatorname{Br}(g^{n+1})$. Moreover, $g(\Sing X)\subset\Sing X\cup\operatorname{Br}(g)$ because a local biholomorphism preserves smoothness. Hence  $g(B)\subset B$.
\end{proof}

The following lemma is a reformulation and recombination of some arguments of Ueda; it is a technical step used multiple times in proofs of later sections and can be skipped for a first reading of the present section. 
\begin{lem}\label{lem:fixed-inverses}
Let $Z\subset\bP^N$ be an irreducible algebraic subvariety with
$f(Z)=Z$, and let $p\in Z$ be a fixed point of $f$. Let $M$ be a smooth irreducible component of the germ of $Z$ at $p$. Assume that $f$ induces an automorphism of $M$. For every sufficiently small analytic representative $V$ of the germ $M$, there is a connected neighborhood $U\subset V$ of $p$ on which, for every $n\geq 0$, the inverse germ of $f^n\vert_M$ at $p$ extends to a holomorphic map $b_n:U\to V$ such that
\begin{equation}\label{eq:fixed-inverse-sections}
b_0=\Id_U,\qquad b_n(p)=p,\qquad f\circ b_{n+1}=b_n,
\end{equation}
and all the images $b_n(U)$ lie in one compact subset of $V$. In particular, $(b_n)$ is normal and equicontinuous at $p$.
\end{lem}

\begin{proof}
If $\dim Z=0$, then there is nothing to prove. Suppose henceforth that $\dim Z>0$. Let $\nu:X\to Z$ be the normalization and let $g:X\to X$ be induced by $f$. By basic properties of normalization~\cite[\S~8.1, Th.~2]{GrauertRemmert}, the smooth branch $M$ corresponds to a fixed point $q\in X$. The point $q$ is mooth and $g$ is locally biholomorphic there. For $V$ sufficiently small, $\nu$ maps a neighborhood $\widetilde V$ of $q$ biholomorphically onto $V$.

Let $B\subset X$ be as in Lemma~\ref{lem:normalization-boundary}. As $q$ is a smooth point, we can find a germ of analytic hypersurface containing $B$ at $q$. This allows us to apply \cite[Lem.~3.8]{Ueda1998}. Since $X$ is normal and $\operatorname{Br}(g^n)\subset B$ for every $n\geq 1$, \cite[Lem.~3.8]{Ueda1998} implies that there is a connected neighborhood $O\subset\widetilde V$ of $q$ such that every connected component of $(g^n)^{-1}(O)$ contains exactly one point of $(g^n)^{-1}(q)$, for every $n\geq 1$.

Let $O_n$ be the connected component of $(g^n)^{-1}(O)$ containing $q$. Then the finite map $g^n:O_n\to O$ is biholomorphic. Let $c_n$ be its inverse and let $c_0:=\Id_U$. As $f^n\circ\nu\circ c_n=\nu\vert_O$, the family $(\nu\circ c_n)$ is normal by \cite[Thm.~2.1]{Ueda1998}. Also $(c_n)$ is normal by Lemma~\ref{lem:finite-normality}. Since all $c_n$ fix $q$, equicontinuity allows us to shrink $O$ so that all $c_n(O)$ lie in one compact subset of $\widetilde V$. Then $U=\nu(O)$ and
$b_n=\nu\circ c_n\circ(\nu\vert_O)^{-1}$ satisfy the conclusions.
\end{proof}

\subsection{Rank reduction and constant limits}\label{subsec:rank-reduction}

Let $S\subset\bP^N$ be a positive-dimensional irreducible algebraic subvariety satisfying $f(S)=S$.

\begin{lem}\label{lem:no-identity-returns}
There is no nonempty analytic open subset $V\subset S_{\rm reg}$ and sequence of positive integers $a_j$ such that $f^{a_j}\vert_V$ converges to $\Id_V$ locally uniformly.
\end{lem}
\begin{proof}
Suppose for contradiction that such a sequence exists.
Let $X\subset S_{\rm reg}$ be the connected dense Zariski-open
subset supplied by Lemma~\ref{lem:inverse-extension}.
Let $\cR\subset X$ consist of the points having a neighborhood
on which $f^{a_j}$ converges locally uniformly to the identity.
Then $\cR$ is open and contains the nonempty set $V\cap X$.

We are going to prove that $\cR$ is closed in $X$. Let $x\in\overline{\cR}\cap X$. Consider a connected neighborhood $B\subset X$ of $x$ over which the conclusion of Lemma~\ref{lem:inverse-extension} holds for every iterate of $f$. Let $D\Subset B\cap\cR$ be a nonempty connected open subset. For sufficiently large $j$, we have $f^{a_j}(D)\subset B$. Hence there is a holomorphic section $\sigma_j:B\to S$ of $f^{a_j}$ satisfying 
\begin{equation}\label{eq:section-D-lemma}
\sigma_j\circ f^{a_j}\vert_D=\Id_D.
\end{equation}

By Lemma~\ref{lem:inverse-normality}, the sequence $(\sigma_j)$
is normal on $B$. By \eqref{eq:section-D-lemma} every
subsequential limit equals the identity map on $D$, hence on $B$.
Thus $\sigma_j\to\Id_B$ locally uniformly. This then implies that the inverses of $\sigma_j$ are defined on a fixed smaller neighborhood $U$ of $x$ for large $j$ and converge to $\Id_U$. These inverses are simply $f^{a_j}\vert_U$. Hence $x\in\cR$. Thus $\cR$ is closed in $X$. As a consequence $\cR=X$.

If $x,y\in X$ satisfy $f(x)=f(y)$, then $f^{a_j}(x)=f^{a_j}(y)$ for every $j$, thus $x=y$ by passing to the limit. Hence $f\vert_X$ is injective. Thus $\deg(f\vert_S)=1$, contradiction to Lemma~\ref{lem:restriction-degree} which says $\deg(f\vert_S)=d_f^{\dim S}>1$. 
\end{proof}

\begin{lem}\label{lem:rank-drop}
Let $\varphi\in\cH_\Omega$ have generic rank $r\geq1$. Every accumulation point of $f^n\circ\varphi$ as $n\to\infty$ has generic rank at most $r-1$. 
\end{lem}
\begin{proof}
Remark first that rank cannot increase under locally uniform convergence. Suppose for contradiction that a subsequence $(f^{m_j}\circ\varphi)$ converges to $\psi$ and $\psi$ has generic rank $r$. By Proposition~\ref{prop:log-rank}, up to a taking further a subsequence, all $f^{m_j}\circ\varphi$ take values in one irreducible variety $X$ of dimension $r$ such that $f^q(X)=X$ and $m_j\equiv m_1\pmod q$ for some $q\geq 1$. By the implicit function theorem, there are a nonempty polydisk $V\subset X_{\rm reg}$ and right inverses $\alpha_j:V\to\Omega$ of $f^{m_j}\circ\varphi$ converging locally uniformly to a right inverse $\alpha$ of $\psi$. Let $a_j=(m_{j+1}-m_j)/q\geq1$. Then, locally uniformly on $V$,
\begin{equation*}
(f^q)^{a_j}\vert_V =(f^{m_{j+1}}\circ\varphi)\circ\alpha_j
\to\psi\circ\alpha=\Id_V. \end{equation*}
This contradicts Lemma~\ref{lem:no-identity-returns}.
\end{proof}

By compactness, every nonempty compact subset of $\cH_\Omega$ which is forward invariant under $H\mapsto f\circ H$ contains a minimal nonempty forward invariant compact subset. Putting together various strenghthening and reformulation of Ueda's results we presented in this section, we finally get:
\begin{thm}\label{thm:periodic-constant}
Every minimal nonempty compact subset of $\cH_\Omega$ which is forward invariant under $H\mapsto f\circ H$ is a periodic orbit of constant maps. In particular, $\cH_\Omega$ contains a constant map whose value is periodic and belongs to $E_\pi$.
\end{thm}
\begin{proof}
Let $\cM$ be such a minimal subset, and let $\varphi\in\cM$. The accumulation set of $(f^n\circ \varphi)$ is a nonempty compact
forward-invariant subset of $\cM$, so it equals $\cM$. In particular, $f^{m_j}\circ \varphi\to \varphi$ for some $m_j\to\infty$. By Lemma~\ref{lem:rank-drop}, $\varphi$ is constant. By Proposition~\ref{prop:log-rank}, all forward images of this constant
map have values in the finite set $E_\pi$. By minimality, $\cM$ is a periodic orbit of constant maps.
\end{proof}

\section{Germs, formal series, invariant direction and algebraization}\label{sec:local-invariant}

In this section, we establish the existence of an algebraic invariant manfiold at a periodic point $p$. More precisely, the tangent directions corresponding to nonzero multipliers of $p$ determine a unique invariant formal germ. The existence of such invariant formal germs is classical. Using the PCF property we prove that it is algebraic and call it the \emph{algebraic core} (see Theorem~\ref{thm:algebraic-core}). From this we obtain the "superattracting or repelling" dichotomy statement (cf. Proposition~\ref{prop:replling-direction}). 

\subsection{Notations and terminology}\label{subsec:notations2}
For a complex analytic germ $(X,p)$, let $\cO_{X,p}$ be the local
ring of holomorphic functions and let $\mathfrak m_{X,p}$ be its
maximal ideal. For a Noetherian local ring $A$, we denote by 
$\widehat A=\varprojlim_{j\geq1}A/\mathfrak m_A^j$ its maximal-ideal completion. A local homomorphism extends uniquely to a continuous homomorphism of completions. In particular, a holomorphic map germ $h:(X,p)\to(Y,q)$ induces $\widehat h^*:\widehat{\cO}_{Y,q}\to\widehat{\cO}_{X,p}$.

We denote by $\bC\{z\}$ and $\bC[[z]]$ the ring of convergent and formal power series in $z=(z_1,\ldots,z_d)$, respectively. For a nonzero element $u$ of a Noetherian local ring, $\operatorname{ord}u$
is the largest integer $j$ such that $u\in\mathfrak m^j$. We use the convention $\operatorname{ord}0=\infty$. 

If $X$ is a complex algebraic variety, we distinguish its algebraic local ring $\cO^{\rm alg}_{X,p}$ from the analytic local ring $\cO_{X,p}$. Their completions are canonically isomorphic. A continuous local $\bC$-algebra homomorphism $\widehat{\cO}_{X,p}\to\bC[[s]]$ restricts to a local homomorphism $\cO^{\rm alg}_{X,p}\to\bC[[s]]$ and hence defines an algebraic morphism $\Spec\bC[[s]]\to X$ whose closed point maps to $p$. 

A holomorphic self-map germ $g:(X,p)\to(X,p)$ is called \emph{superattracting} if $(g^a)^*(\mathfrak m_{X,p})\subset\mathfrak m_{X,p}^2$ for some $a\geq1$. When $X$ is smooth, this is equivalent
to nilpotence of $d_pg$.

\subsection{Formal invariant direction}\label{subsec:formal}

Let $G:(\bC^n,0)\to(\bC^n,0)$ be a finite holomorphic map germ whose differential has a nonzero eigenvalue. Let $r$ be the dimension of the sum of the generalized eigenspaces of $d_0G$ for nonzero eigenvalues. Then $1\leq r\leq n$. Let $b\geq 1$ be an integer such that $(d_0G)^b$ vanishes on the generalized eigenspace for zero, and let $(x,y)\in\bC^r\times\bC^{n-r}$ be linear coordinates adapted to these two subspaces. Let $A$ be the invertible matrix of $(d_0G)^b$ on the first subspace. We also denote by $x$ and $y$ the coordinate projection maps, and define $G_1:=x\circ G^b,G_2:=y\circ G^b$, so 
\[
G^b(x,y)=(G_1(x,y),G_2(x,y)).
\]
Note that by construction the coordinate functions of $G_1(x,y)-Ax$ and $G_2(x,y)$ vanish to order at least two.

Let $(x)\subset\bC[[x]]$ be the ideal generated by $x_1,\ldots,x_r$.
We seek a tuple of formal series
$\widehat\varphi=(\widehat\varphi_1,\ldots,\widehat\varphi_{n-r})$,
with $\widehat\varphi_i\in(x)^2$, satisfying
\begin{equation}\label{eq:series-identities}
G_2(x,\widehat\varphi(x))
=\widehat\varphi\bigl(G_1(x,\widehat\varphi(x))\bigr).
\end{equation}
Geometrically such $\widehat \varphi$ defines a $r$-dimensional $G$-invariant formal subspace passing through the origin. For each choice of such series $\widehat \varphi$, we define
\begin{equation}\label{eq:formal-h-definition}
\widehat h(x):=G_1(x,\widehat\varphi(x)), \quad
\widehat h^*:\begin{cases}
\bC[[x]]\to\bC[[x]]\\
u\mapsto u\circ\widehat h
\end{cases}
\hspace{-1em},\quad
\theta:\begin{cases}
\bC[[x,y]]\to\bC[[x]]\\
u\mapsto u(x,\widehat\varphi(x))
\end{cases}\hspace{-1em}.
\end{equation} 
These homomorphisms are continuous in the maximal-ideal topologies. We denote by $\widehat{(G^b)^*}$ the continuous extension of $(G^b)^*$ from $\bC\{x,y\}$ to $\bC[[x,y]]$.
We set $I:=\ker\theta\subset\bC[[x,y]]$, and denote by $\fm=(x,y)$ the maximal ideal of $\bC[[x,y]]$.

\begin{lem}\label{lem:invariant-series}
Equation~\eqref{eq:series-identities} has a unique solution $\widehat\varphi$.
For this solution, $\widehat h^*$ is an automorphism, $\theta$ is surjective, and 
\begin{equation}\label{eq:substitution-homomorphism}
\ker\theta
=\bigl(y_1-\widehat\varphi_1(x),\ldots,y_{n-r}-\widehat\varphi_{n-r}(x)\bigr), \quad
\theta\circ\widehat{(G^b)^*}=\widehat h^*\circ\theta.
\end{equation}
Moreover, for every $k\geq 1$, we have 
\begin{equation}\label{eq:core-contraction}
\widehat{(G^{bk})^*}I\subset\fm^k I.
\end{equation}
\end{lem}

\begin{proof}
When $r=n$ this follows from the classical Poincaré--Dulac normal form, see e.g.\ \cite[\S~25]{Arnold1988}. When $r<n$, the same elementary computations of formal series work, see \cite[Ch.~I, \S1, Cor.~2]{Belitskii1978} and \cite[p.~1774]{AbateRaissy2013}. Note that here we are only stating a basic feature of Poincaré--Dulac form, instead of its full stength. 
\end{proof}

\subsection{Algebraicity of the invariant manifold}\label{subsec:algebraic-core}

We continue to use $f$ to denote a PCF endomorphism as in Section~\ref{sec:preparation}. 
Let $p$ be a periodic point of $f$ of minimal period $a\geq 1$, and let $E\subset T_p\bP^N$ be the sum of the generalized eigenspaces of $d(f^a)_p$ corresponding to its nonzero eigenvalues. Assume that $r:=\dim E>0$. 

We recall a few facts about normalization. The normalization morphism of a reduced complex space is finite. Its fiber over a point corresponds to the irreducible components of the analytic germ at that point \cite[\S~8.1, Th.~2]{GrauertRemmert}. In particular, the normalization of an irreducible reduced analytic germ has a unique point over its center.
\begin{thm}\label{thm:algebraic-core}
There are a positive integer $m$ divisible by $a$ and an irreducible $f^m$-invariant algebraic subvariety $W\subset\bP^N$ of dimension $r$. The analytic germ $(W,p)$ has a smooth irreducible component $M$ with $T_pM=E$, and $f^m$ induces an automorphism of $M$. On the normalization of $W$, the component $M$ corresponds to a smooth fixed point of the map induced by $f^m$, and the normalization map has injective differential there.
\end{thm}

Let us explain the idea of Theorem~\ref{thm:algebraic-core}. If a formal arc through the fixed point admits inverse lifts through every iterate of $f$, all centered at that point, then \eqref{eq:core-contraction} forces every defining equation of the formal invariant subspace to vanish to arbitrarily high order along the arc. Thus the arc lies on the formal invariant subspace. Conversely, every formal arc on the formal invariant subspace admits such lifts because the induced formal map is invertible. If there is a holomorphic family of arcs with this lifting property, then the formal invariant equations will imply that its generic rank is at most $r$. Postcritical finiteness supplies such a family whose generic rank equals the dimension of the Zariski closure of the formal invariant subspace, so this closure has dimension $r$.

\subsubsection{A rank bound via tangent formal arcs}\label{subsubsec:rank-upper-bound}

When $r<N$, fix an affine chart of $\bP^N$ containing $p$. By translating $p$ to the origin and applying an invertible linear change of coordinates, we obtain affine coordinates $(x,y)\in\bC^r\times\bC^{N-r}$ adapted to $E$ and the generalized eigenspace of $d(f^a)_p$ for zero, as in Subsection~\ref{subsec:formal}. We emphasize that the coordinates $x,y$ here are also globally algebraic coordinates, not just locally analytic. Let $G$ be the germ of $f^a$ in these coordinates, and let $b,\widehat\varphi,\widehat h$ be as in Lemma~\ref{lem:invariant-series}, with $n=N$. Let $m=ab$ and $F=f^m$. We use the homomorphism $\theta(u)=u(x,\widehat\varphi(x))$, the ideal $I=\ker\theta$, and the maximal ideal $\fm=(x,y)$ from Subsection~\ref{subsec:formal}.

The ring $\bC[x,y]$ is the coordinate ring of the above chosen affine chart. Let $W\subset\bP^N$ be the Zariski closure of the common zero set, in the affine chart $(x,y)$, of all polynomials in $I\cap\bC[x,y]$. Note that $I\cap\bC[x,y]\neq \bC[x,y]$ because $\theta (1)=1$ implies $1\notin I$. Let $s:=\dim W$. Let $\sigma:\Spec\bC[[x]]\to W$ be the morphism of $\bC$-schemes induced by the restriction $\theta\vert_{\bC[x,y]}$. Let $\nu:X\to W$ be the normalization. 

\begin{lem}\label{lem:core-closure}
Assume that $r<N$.
The variety $W$ is irreducible, the morphism $\sigma$ is dominant, and $F(W)=W$. Moreover, its dimension satisfies $s\geq r$.
\end{lem}
\begin{proof}
Recall that $I$ denotes the kernel of $\theta$, whose target is $\bC[[x]]$. Since $\bC[[x]]$ is an integral domain, $I\cap\bC[x,y]$ is prime. Hence $W$ is irreducible, and its coordinate ring in the chosen affine chart embeds into $\bC[[x]]$. The generic point $(0)$ of $\Spec\bC[[x]]$ therefore maps to the generic point of $W$, so $\sigma$ is dominant. Since $\theta$ fixes $\bC[x_1,\ldots,x_r]$, we have $I\cap\bC[x_1,\ldots,x_r]={0}$. Hence the images of $x_1,\ldots,x_r$ in this coordinate ring are algebraically independent, which implies $s\geq r$. 

Recall that $p$ is a fixed point of $F$. The morphisms of schemes $F\circ\sigma$ and $\sigma\circ\Spec(\widehat h^{*})$ send the closed point of $\Spec\bC[[x]]$ to $p$. Since $\bC[[x]]$ is a local ring, these morphisms are determined by their pullbacks on $\cO^{\rm alg}_{\bP^N,p}$. For $u\in\cO^{\rm alg}_{\bP^N,p}$, viewed as its own Taylor series in $\bC[[x,y]]$, the two pullbacks are $\theta(u\circ F)$ and $\widehat h^{*}(\theta(u))$, respectively. These are equal by \eqref{eq:substitution-homomorphism}. Hence 
\begin{equation}\label{eq:Fsigmasigmah}
F\circ\sigma=\sigma\circ\Spec(\widehat h^{*})
\end{equation} 
as morphims of $\bC$-schemes. Since $\Spec(\widehat h^{*})$ is an automorphism, $\sigma$ and $F\circ\sigma$ have the same image. The Zariski closure of the image of $\sigma$ being $W$, the Zariski closure of the image of $F\circ\sigma$ is $F(W)$. Hence $F(W)=W$.
\end{proof}

Note that $\Spec\bC[[x]]$ is normal. The morphism $\sigma$ being dominant, it has a unique lift through $\nu$ by the universal property of normalization. Denote by $\kappa:\Spec\bC[[x]]\to X$ this lift, and by $q$ the image of the closed point. Let $g:X\to X$ be the finite surjective morphism induced by $F$, so $\nu\circ g=F\circ\nu$. We have the commutative diagram:
\[
\begin{tikzcd}
  & X \arrow[d, "\nu"] \arrow[loop right, "g"] \\
  \Spec\bC[[x]]
    \arrow[ur, "\kappa"]
    \arrow[r, "\sigma"']
    \arrow[loop left, "\Spec(\widehat h^{*})"]
  & W \arrow[loop right, "F"]
\end{tikzcd}
\]
Since $\kappa$ sends the closed point to $q$, the induced homomorphism of local rings $\kappa^{*}:\cO^{\rm alg}_{X,q}\to\bC[[x]]$ is local. As $\bC[[x]]$ is complete, it extends uniquely to a continuous local homomorphism
\begin{equation}\label{eq:core-completed-local-ring}
\widehat\kappa^{*}:\widehat{\cO}_{X,q}\longrightarrow\bC[[x]].
\end{equation}

\begin{lem}\label{lem:core-normalization}
Assume that $r<N$. The morphism $\kappa$ is dominant, $\nu(q)=p$, and $g(q)=q$. The homomorphism $\widehat\kappa^{*}$ is surjective.
\end{lem}
\begin{proof}
As $\nu\circ\kappa=\sigma$, $\sigma$ dominant, and $\nu$ birational, we deduce that $\kappa$ is dominant and $\nu(q)=p$. By \eqref{eq:Fsigmasigmah} and uniqueness of lifting through normalization, we have $g\circ\kappa=\kappa\circ\Spec(\widehat h^*)$. In particular $g(q)=q$.

Let $\widehat g^{*}:\widehat{\cO}_{X,q}\to\widehat{\cO}_{X,q}$ and
$\widehat\nu^{*}:\bC[[x,y]]\to\widehat{\cO}_{X,q}$ be the completions
of the local pullbacks induced at $q$ respectively by $g$ and by the composition $X\xrightarrow{\nu}W\hookrightarrow\bP^N$. 
We regard $\widehat{\cO}_{\bP^N,p}$ as $\bC[[x,y]]$ using the
coordinates $(x,y)$. The homomorphism $\cO^{\rm alg}_{\bP^N,p}\to\bC[[x]]$ induced by $\Spec\bC[[x]]\xrightarrow{\sigma}W\hookrightarrow\bP^N$ extends continuously to $\theta:\bC[[x,y]]\to\bC[[x]]$ by the definition of $\sigma$. From $g\circ\kappa=\kappa\circ\Spec(\widehat h^{*})$ and
$\nu\circ\kappa=\sigma$, we deduce that
\begin{equation}\label{eq:normalized-core-equivariance}
\widehat\kappa^{*}\circ\widehat g^{*}
=\widehat h^{*}\circ\widehat\kappa^{*},\qquad
\widehat\kappa^{*}\circ\widehat\nu^{*}=\theta.
\end{equation}
As $\theta$ is surjective by Lemma~\ref{lem:invariant-series}, \eqref{eq:normalized-core-equivariance} implies that $\widehat\kappa^{*}$ is surjective.
\end{proof}

Let $\Delta\subset\bC$ be a disk centered at zero, with coordinate $t$, and let $U\subset\bC^s$ be a nonempty polydisk with coordinates $c=(c_1,\ldots,c_s)$. Let $J:\Delta\times U\to\bP^N$ denote a holomorphic map with image in the affine chart $(x,y)$ such that $J(0,c)=p$ for every $c\in U$. 

For each fixed $c\in U$, consider the Taylor series at $t=0$ of
the functions $t\mapsto x_i(J(t,c))$, for $i=1,\ldots,r$, and those of $t\mapsto y_j(J(t,c))$, for $j=1,\ldots,N-r$. Let $\widehat J(t,c)$ be the ordered tuple of these $N$ formal power series, i.e.\ a vector valued formal series. Since $J(0,c)=p$ and the coordinates $(x,y)$ vanish at $p$, every series has zero constant term. Thus for fixed $c$ we have $\widehat J(t,c)\in(t)\bC[[t]]^N$, while $\widehat J$ can be thought of as an element of $(t)\cO(U)[[t]]^N$ where $\cO(U)$ is the ring of holomorphic functions on $U$.

\begin{lem}\label{lem:inverse-arc-rank}
Assume that $r<N$.
If, for every $n\geq1$ and $c\in U$, the equation \footnote{For $n\geq1$ and $\widehat\beta\in(t)\bC[[t]]^N$, we define the composition $F^n\circ\widehat\beta$ by substituting $\widehat\beta$ into the Taylor series of $F^n$ at $p$ in the coordinates $(x,y)$.} $F^n\circ\widehat\beta=\widehat J(t,c)$ has a solution $\widehat\beta$ in $(t)\bC[[t]]^N$, then $\operatorname{rank}(dJ)\leq r$ on $\Delta\times U$.
\end{lem}
\begin{proof}
By Lemma~\ref{lem:invariant-series}, we have $e\circ F^n\in\fm^nI\subset\fm^{n+1}$ for every $e\in I$ and $n\geq1$. For $n\geq1$ and $c\in U$, if $\widehat\beta_{n,c}\in(t)\bC[[t]]^N$ satisfies $F^n\circ\widehat\beta_{n,c}=\widehat J(t,c)$, then for every $e\in I$, we have
\[
e(\widehat J(t,c))=(e\circ F^n)(\widehat\beta_{n,c}(t))\in(t)^{n+1}.
\]
This holds for every $n\geq 1$, so $e(\widehat J(t,c))=0$.

Since $c$ was arbitrary, $e(\widehat J)=0$ in
$\cO(U)[[t]]$ for every $e\in I$, hence $y\circ\widehat J=\widehat\varphi(x\circ\widehat J)$.
By doing differentiation to the last equation and by applying the chain rule in $(t,c)$, we deduce that every $(r+1)$-minor of $dJ$ has zero Taylor series in $t$. These minors are holomorphic, so they vanish identically.
\end{proof}

\subsubsection{Constructing tangent formal arcs by postcritical finiteness}\label{subsubsec:construct-arcs}

Recall that we want to prove $r=s=\dim W$. We always have $r\leq s$. To prove the reverse inequality, it suffices by Lemma~\ref{lem:inverse-arc-rank} to construct a holomorphic
family $J$ of generic rank $s$ satisfying its lifting hypothesis. The following lemma constructs such a family using the postcritical finiteness of $f$. 

\begin{lem}\label{lem:core-inverse-family}
Assume that $r<N$.
There exists a holomorphic map $J:\Delta\times U\to\bP^N$ of generic rank $s$, with image in the affine chart $(x,y)$ and such that $J(0,c)=p$ for every $c\in U$, satisfying the following property. For every $n\geq1$ and $c\in U$, there exists a holomorphic map $\beta_{n,c}:\Delta\to\bP^N$ such that $\beta_{n,c}(0)=p$ and $F^n(\beta_{n,c}(t))=J(t,c)$ for every $t\in\Delta$.
\end{lem}

\begin{rem}\label{rem:core-inverse-family}
The constant maps $J\equiv p$ and $\beta_{n,c}\equiv p$ satisfy the lift property trivially, but they have rank zero. The hard requirement of Lemma~\ref{lem:core-inverse-family} is one family $J$ of generic rank $s$. 

The required lifting property is not trivial for a non constant map. For example, let $T:\bC^2\to\bC^2$ be defined by $T(z,w)=(2z,w^2)$, so $T^n(z,w)=(2^nz,w^{2^n})$.
For each integer $L\geq0$, the arc $t\mapsto(0,t^{2^L})$
has the holomorphic lift $t\mapsto(0,t^{2^{L-n}})$ through
$T^n$ whenever $1\leq n\leq L$.
A lift through $T^{L+1}$ would require its second coordinate
to have vanishing order $1/2$ at zero, which is impossible.
In particular, the arc $t\mapsto(0,t)$ has no holomorphic lift
through $T$.

\end{rem}

\begin{rem}\label{rem:core-inverse-family-proof}
As the proof of Lemma~\ref{lem:core-inverse-family} is notationally heavy, we give a sketch here. 
In the proof of Lemma~\ref{lem:core-inverse-family}, postcritical finiteness is used through Lemma~\ref{lem:normalization-boundary}: the branch loci of all $g^n$ lie in one proper algebraic subset of $X$. 
The dominance of $\kappa$ allows a formal arc centered at $q$ which does not go through this subset, while the invertibility of $\widehat h$ supplies formal lifts through every $g^n$, all centered at $q$. 
By Lemma~\ref{lem:arc-transfer}, holomorphic arcs with the same
center and contact orders as the formal arc on the resolution,
and avoiding the divisor away from zero, lift to certain covers. We vary coefficients independently in the $s$ local coordinates on the resolution, which will lead to the desired rank property.
\end{rem}

\begin{proof}[{Proof of Lemma~\ref{lem:core-inverse-family}}]
We apply first Lemma~\ref{lem:normalization-boundary} to $F$ and $W$,
and obtain a proper forward-invariant algebraic subset $B\subset X$ containing $\Sing X$ and the branch loci of all $g^n$, $n\geq 1$. 
Let $\pi:Y\to X$ be a projective log resolution morphism which is an
isomorphism outside $B$, such that $D=\pi^{-1}(B)$ a simple normal crossing divisor. Let $\mu:=\nu\circ\pi$. 
\[
\begin{tikzcd}
  D \arrow[r, hook] \arrow[d]
  & Y \arrow[d, "\pi"] \arrow[dr, "\mu"] & \\
  B \arrow[r, hook]
  & X \arrow[r, "\nu"'] & W
\end{tikzcd}
\]
Let $\eta\in\cO^{\rm alg}_{X,q}$ be a nonzero function which vanishes on $B$ near $q$. Since $\kappa$ is dominant, $\kappa^*(\eta)\neq 0$.
The first nonzero homogeneous term of $\kappa^*(\eta)$ is a
nonzero polynomial on $\bC^r$. Let $v\in\bC^r\setminus\{0\}$ be a point where this polynomial does not vanish, and let $\chi_v:\bC[[x]]\to\bC[[t]]$, where $t$ is a single variable, to be the continuous homomorphism defined by $x_i\mapsto tv_i$. 

For $n\geq 0$, the restriction of $\chi_v\circ(\widehat h^{*})^{-n}\circ\widehat\kappa^*$ to $\cO^{\rm alg}_{X,q}$ is a local homomorphism. Let $\gamma_n:\Spec\bC[[t]]\to X$ be the induced morphism. Then each $\gamma_n$ sends the closed point of $\Spec\bC[[t]]$ to $q$. We have $g^n\circ\gamma_n=\gamma_0$ by \eqref{eq:normalized-core-equivariance}.
\[
\begin{tikzcd}[column sep=large]
  \Spec\bC[[t]]
    \arrow[d, "\Spec(\chi_v)"']
    \arrow[ddr, bend left=20, "\gamma_n"] & \\
  \Spec\bC[[x]]
    \arrow[d, "{\Spec((\widehat h^{*})^{-n})}"'] & \\
  \Spec\bC[[x]] \arrow[r, "\kappa"']
  & X \arrow[loop right, "g^n"]
\end{tikzcd}
\]
Since $\chi_v(\kappa^*(\eta))\neq 0$ by construction, the morphism $\gamma_0$ sends the generic point $\Spec\bC((t))$ of $\Spec\bC[[t]]$ to somewhere outside $B$. The restriction of $\gamma_0$ to $\Spec\bC((t))$ therefore lifts uniquely through $\pi$. By the valuative criterion for properness~\cite[Th.~II.4.3]{Hartshorne},  this lift extends to a morphism $\delta:\Spec\bC[[t]]\to Y$ such that $\pi\circ\delta=\gamma_0$. Let $y_0\in Y$ be the image under $\delta$ of the closed point of $\Spec\bC[[t]]$, and let $\widehat\delta^*:\widehat{\cO}_{Y,y_0}\to\bC[[t]]$ be the continuous extension of
$\delta^*:\cO^{\rm alg}_{Y,y_0}\to\bC[[t]]$.
\[
\begin{tikzcd}
  & Y \arrow[d, "\pi"] \\
  \Spec\bC[[t]] \arrow[ur, "\delta"]
    \arrow[r, "\gamma_0"'] & X
\end{tikzcd}
\]
Let $P$ be an analytic polydisk neighborhood around $y_0$, with coordinates $u_1,\ldots,u_s$ centered at $y_0$, such that $D\cap P=\{u_1\cdots u_k=0\}$ for some $0\leq k\leq s$ and $\mu(P)$ lies in the affine chart $(x,y)$ of $\bP^N$. For $1\leq i\leq k$, let $a_i:=\operatorname{ord}_t\widehat\delta^*(u_i)$. These integers are positive and finite because a local algebraic equation of $D$, which is a unit multiple of $u_1\cdots u_k$ in $\cO_{Y,y_0}$, has nonzero image under $\delta^*$.

For $n\geq 1$, Define the fiber product $Q_n:=Y\mathbin{\times}_{\pi,X,g^n}X$ by using $\pi:Y\to X,g^n:X\to X$. Let
$Y_n$ be the normalization of the union of the irreducible
components of $Q_n$ dominating $Y$. Let $\rho_n:Y_n\to Y$ and $\tau_n:Y_n\to X$ be induced respectively by the first and second projections of $Q_n$, so that $\pi\circ\rho_n=g^n\circ\tau_n$.
The fiber product $Q_n$ is finite étale over the smooth variety $Y\setminus D$, thus unchanged by normalization. In other words $\rho_n$ is unramified outside $D$. Note that $Y_n$ has pure dimension $s$, so any irreducible component $Z$ of $\rho_n^{-1}(P)$ has
dimension $s$ and satisfies $\rho_n(Z)=P$.
\[
\begin{tikzcd}[column sep=large, row sep=small]
  Y_n \arrow[dr]
    \arrow[drr, bend left=15, "\rho_n"]
    \arrow[ddr, bend right=15, "\tau_n"'] & & & \\
  & Q_n \arrow[r] \arrow[d]
  & Y \arrow[d, "\pi"] & P \arrow[l, hook'] \\
  & X \arrow[r, "g^n"'] & X &
\end{tikzcd}
\]
The pair $(\delta,\gamma_n)$ defines a morphism $\Spec\bC[[t]]\to Q_n$ whose restriction to $\Spec\bC((t))$ lies over $Y\setminus D$ and therefore lifts in a unique way to $Y_n$. Then the lift extends to
$\lambda_n:\Spec\bC[[t]]\to Y_n$ by the valuative criterion for properness. We have $\rho_n\circ\lambda_n=\delta$ and
$\tau_n\circ\lambda_n=\gamma_n$.
Let $\xi_n$ be the image under $\lambda_n$ of the closed point of
$\Spec\bC[[t]]$. Then $\rho_n(\xi_n)=y_0$ and $\tau_n(\xi_n)=q$.
Let $\widehat\lambda_n^*:\widehat{\cO}_{Y_n,\xi_n}\to\bC[[t]]$
be the continuous extension of the local homomorphism induced
by $\lambda_n$ on $\cO^{\rm alg}_{Y_n,\xi_n}$. For $1\leq i\leq k$, we have $\widehat\lambda_n^*(u_i\circ\rho_n)=\widehat\delta^*(u_i)$,
so $\operatorname{ord}_t\widehat\lambda_n^*(u_i\circ\rho_n)=\operatorname{ord}_t\widehat\delta^*(u_i)=a_i$.
\[
\begin{tikzcd}[column sep=large, row sep=small]
  \Spec\bC[[t]] \arrow[dr, "\lambda_n"]
    \arrow[drr, bend left=15, "\delta"]
    \arrow[ddr, bend right=15, "\gamma_n"'] & & \\
  & Y_n \arrow[r, "\rho_n"] \arrow[d, "\tau_n"']
  & Y \arrow[d, "\pi"] \\
  & X \arrow[r, "g^n"'] & X
\end{tikzcd}
\]
For $1\leq i\leq k$, let $\zeta_i\in\bC$ satisfy $\lvert\zeta_i\rvert>\sup_{c\in U}\lvert c_i\rvert$. We fix a sufficiently small $\epsilon>0$ so that we can define $\Psi:\Delta\times U\to P$ by 
\begin{equation}\label{eq:weighted-core-arc}
u_i(\Psi(t,c)):= 
\begin{cases}
\epsilon(c_i-\zeta_i)t^{a_i},&1\leq i\leq k,\\
\epsilon c_i t,&k<i\leq s.
\end{cases}
\end{equation}
Let $J:=\mu\circ\Psi$. Then $J(0,c)=p$ for every $c\in U$.
\[
\begin{tikzcd}[column sep=large, row sep=small]
  \Delta\times U \arrow[r, "\Psi"]
    \arrow[ddrr, bend right=20, "J"']
  & P \arrow[r, hook] \arrow[ddr, "\mu|_P"]
  & Y \arrow[d, "\pi"] \\
  & & X \arrow[d, "\nu"] \\
  & & W
\end{tikzcd}
\]
Fix $n\geq 1$ and $c\in U$. By \eqref{eq:weighted-core-arc}, we have $\Psi(t,c)\notin D$ for $t\ne0$ and $\operatorname{ord}_{t=0}u_i(\Psi(t,c))=a_i$ for $1\leq i\leq k$.
By Lemma~\ref{lem:arc-transfer}, applied to
$\rho_n^{-1}(P)\to P$ and $\widehat\lambda_n^*$, there exists a
holomorphic map $\widetilde\Psi_{n,c}:\Delta\to Y_n$ such that
$\rho_n\circ\widetilde\Psi_{n,c}=\Psi(\,\cdot\,,c)$ and $\widetilde\Psi_{n,c}(0)=\xi_n$.
If we take $\beta_{n,c}:=\nu\circ\tau_n\circ\widetilde\Psi_{n,c}$, then $\beta_{n,c}(0)=\nu(q)=p$ and $F^n\circ\beta_{n,c} =\mu\circ\Psi(\,\cdot\,,c)=J(\,\cdot\,,c)$, as desired.
\[
\begin{tikzcd}[column sep=60pt, row sep=large]
  \Delta \arrow[dr, "\widetilde\Psi_{n,c}"]
    \arrow[ddr, "{\Psi(\,\cdot\,,c)}"', pos=0.7]
    \arrow[drrr, bend left=18, "\beta_{n,c}"]
    \arrow[ddrrr, out=270, in=200, looseness=1.3,
      "{J(\,\cdot\,,c)}", pos=0.6] & & & \\
  & Y_n \arrow[r, "\tau_n"] \arrow[d, "\rho_n"]
  & X \arrow[r, "\nu"] \arrow[d, "g^n"]
  & W \arrow[d, "F^n"] \\
  & Y \arrow[r, "\pi"']
  & X \arrow[r, "\nu"'] & W
\end{tikzcd}
\]
Finally we remark that, for each fixed $t\in\Delta\setminus\{0\}$, the map $c\mapsto\Psi(t,c)$ is biholomorphic onto an open subset of $P$. As $\mu$ has generic rank $s$ on $P$, the map $c\mapsto J(t,c)$ has generic rank $s$. The generic rank cannot be larger than $s$ because $J$ factors through $P$, thus is $s$.
\end{proof}

\begin{proof}[Proof of Theorem~\ref{thm:algebraic-core}]
If $r=N$, then let $m=a$, $W=\bP^N$, $M=(\bP^N,p)$ and we obtain the conclusions immediately. 

Assume now that $r<N$, and let $J$ and $\beta_{n,c}$ be as in
Lemma~\ref{lem:core-inverse-family}. For every $n\geq1$ and $c\in U$, the tuple of Taylor series, i.e.\ the vector-valued series, at $t=0$ of $(x,y)\circ\beta_{n,c}$ belongs to $(t)\bC[[t]]^N$ and solves the equation $F^n\circ\widehat\beta=\widehat J(t,c)$. 
By Lemma~\ref{lem:inverse-arc-rank}, the generic rank $s$ of $J$
is at most $r$. By Lemma~\ref{lem:core-closure}, the variety $W$ is irreducible, $F(W)=W$, and $s\geq r$. Hence $\dim W=s=r$.

The variety $X$ being normal, the completed local ring $\widehat{\cO}_{X,q}$ is normal by \cite{Zariski_normal}. By Lemma~\ref{lem:core-normalization}, the homomorphism $\widehat\kappa^*:\widehat{\cO}_{X,q}\to\bC[[x]]$ is surjective. A non-trivial integral quotient of $\widehat{\cO}_{X,q}$ would have Krull dimension less than $r$. Since $\dim\bC[[x]]=r$, the kernel is zero. Thus $\widehat\kappa^*$ is an isomorphism, so $\widehat{\cO}_{X,q}$ is regular and $X$ is smooth at $q$.

By \eqref{eq:normalized-core-equivariance}, we have $\widehat\kappa^*(x_i\circ\nu)=x_i$ for $1\leq i\leq r$. Therefore $x_1\circ\nu,\ldots,x_r\circ\nu$ form local analytic coordinates. This implies that $\nu$ has injective differential at $q$. In particular, the germ $(X,q)$ is mapped by $\nu$ biholomorphically to a smooth graph germ $M=\{y=\varphi(x)\}$ for some holomorphic germ $\varphi:(\bC^r,0)\to(\bC^{N-r},0)$. As $M\subset(W,p)$ and $\dim M=\dim W=r$, the germ $M$ is an irreducible component of $(W,p)$. The equality $\widehat\kappa^*\circ\widehat\nu^*=\theta$ in \eqref{eq:normalized-core-equivariance} implies that the Taylor series of $\varphi$ is $\widehat\varphi$. As $\widehat\varphi\in(x)^2\bC[[x]]^{N-r}$, we get $T_pM=E$.

In the coordinates $x_1\circ\nu,\ldots,x_r\circ\nu$, the Taylor
series of $g$ at $q$ is $\widehat h$, by the first identity in
\eqref{eq:normalized-core-equivariance}. Since the linear part of $\widehat h$ is invertible, $g$ is locally biholomorphic at $q$. 
As $g(q)=q$ and $\nu\circ g=F\circ\nu$, we obtain an induced automorphism of $M$.
\end{proof}

\begin{cor}\label{cor:core-independence}
The variety $W$ is the Zariski closure of the germ $M$. The germ $M$ and the variety $W$ in Theorem~\ref{thm:algebraic-core} are uniquely determined by $f$ and $p$ and are unchanged when $f$ is replaced by a
positive iterate.
\end{cor}
\begin{proof}
Every proper algebraic subset of $W$ has dimension less than $r$, whereas $M$ has dimension $r$. Thus $W$ is the Zariski closure of $M$. If $r=N$, then $M=(\bP^N,p)$ and $W=\bP^N$, so there is nothing to prove. Assume that $r<N$.

Let $M_1,M_2$ be germs obtained from Theorem~\ref{thm:algebraic-core} at $p$ for positive iterates of $f$. Positive powers of $d(f^a)_p$ have the same sum $E$ of generalized eigenspaces for nonzero eigenvalues, so both germs are tangent to $E$. Let $m'\geq1$ be such that $f^{m'}$ preserves both germs and $d(f^{m'})_p$ vanishes on the generalized eigenspace of $d(f^a)_p$ for zero. In the fixed adapted coordinates, each $M_i$ is a graph $y=\varphi_i(x)$ for a holomorphic germ $\varphi_i$ vanishing to order at least two. By uniqueness in Lemma~\ref{lem:invariant-series}, applied to the germ of $f^{m'}$ with $b=1$, the Taylor series of $\varphi_1$ and $\varphi_2$ coincide.
Thus $M_1=M_2$.
\end{proof}

\subsection{Repelling direction}\label{subsec:repelling-direction}
Astorg~\cite{Astorg2020} and Van Tu Le~\cite{Le2020,Le2023} proved some "superattracting or repelling" dichotomy statements for periodic points of a PCF endomorphism under the so called "all-the-way-down" hypothesis and some additional hypotheses on regularity of the PCF divisor. Their results inspired the following consequence of Theorem~\ref{thm:algebraic-core}, where the conclusion is weaker than theirs but holds unconditionally. 

\begin{prop}\label{prop:replling-direction}
Let $p\in \bP^N$ be a periodic point of $f$ of minimal period $a\geq 1$. If $p$ is not superattracting, then $d(f^a)_p$ has an
eigenvalue of modulus greater than one.
\end{prop}
\begin{proof}
Suppose for contradiction that $p$ is not superattracting and that every eigenvalue of $d(f^a)_p$ has modulus at most one. Let $E$ be the sum of its generalized eigenspaces for nonzero eigenvalues. Let $m,W,M$ be as in Theorem~\ref{thm:algebraic-core}, and let $F=f^m$.
Since $T_pM=E$ and $m$ is divisible by $a$, every eigenvalue of
$d_pF\vert_E$ has modulus at most one.

We apply Lemma~\ref{lem:fixed-inverses} to $F$, $W$, and $M$. The inverse maps $b_n$ fixing $q$ obtained by Lemma~\ref{lem:fixed-inverses} form a normal family on a common neighborhood $U\subset M$ of $p$ in $M$. Therefore the differentials $d_pb_n=(d_pg)^{-n}$ are uniformly bounded, and $\lvert\det(d_pF\vert_E)\rvert\geq 1$. As every eigenvalue of $d_pF\vert_E$ has modulus at most one, we have $\lvert\det(d_pF\vert_E)\rvert=1$.

By normality, there is a subsequence $(b_{n_j})$ which converges locally uniformly on $U$ to a holomorphic map $b$. We have $b(p)=p$ and $\lvert\det d_pb\rvert=1$ by passing to the limit, so $b$ is locally biholomorphic at $p$.
Thus, there is a connected neighborhood $V$ of $p$ in $M$ on which, for all large $j$, the inverses $b_{n_j}^{-1}$ are defined, take values in $U$, and converge locally uniformly to $b^{-1}$. 
Let $l_j:=n_{j+1}-n_j\geq 1$. Iterating $F\circ b_{n+1}=b_n$, we obtain $F^{l_j}\circ b_{n_{j+1}}=b_{n_j}$ on $U$. We thus get locally uniform convergence
\begin{equation}\label{eq:core-identity-returns}
F^{l_j}\vert_V
=b_{n_j}\circ b_{n_{j+1}}^{-1}
\longrightarrow b\circ b^{-1}=\Id_V.
\end{equation}
The set $V\cap W_{\rm reg}$ is a nonempty open subset of
$W_{\rm reg}$. The convergence in \eqref{eq:core-identity-returns} on this subset contradicts Lemma~\ref{lem:no-identity-returns}. 
\end{proof}

\section{Transverse superattraction along an algebraic core}\label{sec:transverse}

As before, $f$ is a PCF endomorphism. Let $p$ be a
non-superattracting periodic point. Theorem~\ref{thm:algebraic-core}
provides a periodic irreducible projective subvariety $W$ and a
smooth irreducible component $M$ of the analytic germ $(W,p)$.
We call $M$ the \emph{smooth core germ} of $p$, and $W$ its
\emph{algebraic core}. By Corollary~\ref{cor:core-independence},
$W$ is the Zariski closure of $M$. We denote $r=\dim W=\dim M$.
The goal of this section is to prove the following.

\begin{thm}\label{thm:transverse}
There is an integer $m\geq 1$ such that $f^m(p)=p$, $f^m(W)=W$, and, for every sufficiently large $n$, the differential $d(f^{mn})_x:T_x\bP^N\to T_{f^{mn}(x)}\bP^N$ has rank $r$ for $x\in W_n$ where $W_n\subset W$ is some non-empty Zariski-open subset.
\end{thm}

\begin{rem}
We will adapt the existence of formal invariant manifold of
Lemma~\ref{lem:invariant-series} from a fixed point to a general point of $W$. As the base point under consideration needs not be fixed, it is a version of Lemma~\ref{lem:invariant-series} for a dynamical cocycle. In algebraic language, we study dynamics over a function field where the transformation is not really a morphism defined over the function field, but a morphism of $\bC$-schemes which induces also an endomorphism of the function field. See \cite{ZhaoCentralizer}, \cite{NieZhao}, \cite{DFR} and \cite{AbboudXie} for other uses of such methods of twisted maps in complex dynamics. 

More concretely, in the fixed-point-case of Section~\ref{sec:local-invariant}, geometrically we considered a fixed point and jets at that point, which analytically amounts to a study of formal series; in this section the fixed point is replaced by an invariant subvariety and formal series will have coefficients in the function field of that subvariety. The formal series ring $L[[w]]$ that we will be manipulating in this section should be understood as an algebraic formulation of the jet bundle along the algebraic core we constructed in Section~\ref{sec:local-invariant}. 
\end{rem}

Throughout the section we fix an iterate $F=f^a$ satisfying $F(p)=p$, $F(W)=W$, and $F(M)=M$, such that the germ $h:=F\vert_M$ is locally biholomorphic at $p$. Such an iterate exists by Theorem~\ref{thm:algebraic-core}. Let $\varrho$ be the minimum, over $n\geq1$, of the generic rank along $W$ of the differential $d(F^n)$. 

\subsection{The coefficient field}\label{subsubsec:coefficient-field}
The objects that we are going to define are for the following purpose: we want to move simultaneously a point $x\in W$ and a nearby point $z\in\bP^N$. Consider the commutative diagram:
\[
\begin{tikzcd}[column sep=large]
W\times\bP^N \arrow[r,"F\vert_W\times F"]
  \arrow[d,"\mathrm{pr}_1"']
& W\times\bP^N \arrow[d,"\mathrm{pr}_1"] \\
W \arrow[r,"F\vert_W"'] & W.
\end{tikzcd}
\]
On the generic fiber, the pullback acts on the function field $K:=\bC(W)$ by $F^*:b\mapsto b\circ F$. Note that this embedding of fields need not be surjective. However, we will enlarge $K$ so that it becomes an automorphism. We denote by $\mathcal M$ the field of meromorphic germs on $M$ at $p$. As $M$ is Zariski-dense in $W$, we can think of $K$ as a subfield of $\mathcal M$. Define
\begin{equation}\label{eq:sigma-definition}
\sigma(b):=b\circ h,\qquad
\sigma^{-1}(b):=b\circ h^{-1},\qquad
L:=\bigcup_{n\geq0}\sigma^{-n}K.
\end{equation}
The first two maps are inverse automorphisms of
$\mathcal M$. Since $F(W)=W$ and $F^n\vert_W$
is finite and surjective, we have finite extensions
\begin{equation}\label{eq:K-nested}
K\subset\sigma^{-1}K\subset\sigma^{-2}K\subset\cdots.
\end{equation}
Thus $L$ is a field, $\sigma(L)=L$ and $L/K$ is algebraic. 

\subsection{Twisted pullback}\label{subsubsec:twisted-pullback}
Choose affine coordinates $z_1,\ldots,z_N$ of $\bP^N$ at $p$, and let
$\xi_i\in K$ be the restriction of $z_i$ to $W$. We will write
$\xi=(\xi_1,\ldots,\xi_N)$ and $F_i=z_i\circ F$. Let $w=(w_1,\ldots,w_N)$ be independent variables which will be used as variables of power series; they are supposed to represent displacement from the center $\xi$. 

Since $F(p)=p$, we can write $F_i=N_i/D_i$ with $N_i,D_i\in\bC[z_1,\ldots,z_N]$ and $D_i(p)\neq 0$. 
Then $D_i(\xi)$ as an element of $K$ is nonzero, so $D_i(\xi+w)$ is a unit in $L[[w]]$. Therefore $F_i(\xi+w)$ is a well-defined formal
series in $L[[w]]$ and
\begin{equation}\label{eq:difference_F_i}
F_i(\xi+w)-F_i(\xi)\in(w)\subset L[[w]].
\end{equation}
For a multi-index $\alpha\in\bZ_{\geq0}^N$, we write $w^\alpha=\prod_iw_i^{\alpha_i}$ and $\lvert\alpha\rvert=\sum_i\alpha_i$. We define
\begin{equation}\label{eq:twisted-Phi}
\Phi\left(\sum_\alpha b_\alpha w^\alpha\right)
:=\sum_\alpha\sigma(b_\alpha)
\prod_{i=1}^N\bigl(F_i(\xi+w)-F_i(\xi)\bigr)^{\alpha_i},
\quad b_\alpha\in L,
\end{equation}
where each $\xi_i$ is viewed as element of $L$. 
By \eqref{eq:difference_F_i}, only finitely many terms contribute modulo $(w)^k$, for each $k$. Thus $\Phi$ is a well defined continuous endomorphism of the $\bC$-algebra $L[[w]]$ such that $\Phi((w)^k)\subset(w)^k$ and $\Phi\vert_L=\sigma$. 

\begin{rem}
The geometric intepretation of \eqref{eq:twisted-Phi} is as follows. For $x\in W$, we write a nearby point in $\bP^N$ as $x+w$, where $w$ represents a displacement. The pair transforms as $(x,w)\mapsto \bigl(F(x),F(x+w)-F(x)\bigr)$, i.e.\ it approximates the tangent map. 
\end{rem}

\begin{lem}\label{lem:twisted-iterates}
For $n\in \bN^*$ we have $\Phi^n(w)=F^n(\xi+w)-F^n(\xi)$.
\end{lem}
\begin{proof}
Since $\sigma(\xi)=F(\xi)$, we have $\Phi(\xi)=F(\xi)$ and $\Phi(\xi+w)=F(\xi+w)$. As $\Phi$ is a morphism of $\bC$-algebras, for polynomials $P,Q\in\bC[z_1,\ldots,z_N]$ such that $Q(p)\neq 0$, we have
\[
\Phi\left(\frac{P(\xi+w)}{Q(\xi+w)}\right)
=\frac{P(\Phi(\xi+w))}{Q(\Phi(\xi+w))}
=\frac{P(F(\xi+w))}{Q(F(\xi+w))}.
\]
Then for $n\in \bN^*$ we have
\[
\Phi\bigl(F^n(\xi+w)-F^n(\xi)\bigr)
=F^{n+1}(\xi+w)-F^{n+1}(\xi).
\]
The desired formula follows by induction.
\end{proof}

The following lemma illustrates the idea that our twisted formalism is a way to encode the differential cocyle (and higher jets).
\begin{lem}\label{lem:twisted-jacobian}
For every $n\geq1$, the image of the map induced by $\Phi^n$ on $(w)/(w)^2$ has dimension over $L$ equal to the generic rank of $F^n$ along $W$.
\end{lem}

\begin{proof}
Fix $n\geq 1$. Denote by $a_{ij}\in K$ the restriction to $W$ of $\partial(z_i\circ F^n)/\partial z_j$. 
For $c_1,\ldots,c_N\in L$, by Lemma~\ref{lem:twisted-iterates} we have
\begin{equation}\label{eq:twisted-jacobian}
\Phi^n\left(\sum_{i=1}^N c_iw_i\right)
\equiv
\sum_{i,j=1}^N\sigma^n(c_i)a_{ij}w_j
\pmod{(w)^2}.
\end{equation}
Since $\sigma^n:L\to L$ is bijective, by \eqref{eq:twisted-jacobian} the image on $(w)/(w)^2$ is the $L$-span of the rows of $(a_{ij})$.

Let $d$ be the generic rank of $d(F^n)$ along $W$. Then for a general point $x\in W$, both $x$ and $F^n(x)$ belong to the chosen affine chart, and the matrix $(a_{ij}(x))$ represents $d(F^n)_x$ which has rank $d$. Hence every minor of $(a_{ij})$ of size greater
than $d$ vanishes identically as a rational function on $W$, whereas at least one minor of size $d$ does not vanish identically.
These determinants are elements of $K\subset L$. Therefore $(a_{ij})$ as a matrix with entries in $L$, has rank exactly $d$.
\end{proof}

\begin{lem}\label{lem:twisted-generic-rank}
For $n$ sufficiently large, the generic ranks of $F^n$ along $W$ are equal to a constant $\varrho$ such that $\varrho\geq r$.
\end{lem}

\begin{proof}
The images of the maps induced by $\Phi$ on $(w)/(w)^2$ form a decreasing sequence of $L$-vector subspaces. Their dimensions stabilize because $(w)/(w)^2$ has finite dimension over $L$. By Lemma~\ref{lem:twisted-jacobian}, the generic ranks of $d(F^n)$ along $W$ therefore decrease to some minimum $\varrho$.

Note that the finite morphism $F^n\vert_W$ is étale at a general point $x\in W$, where both $x$ and $F^n(x)$ are smooth points of $W$. Thus for such a point $x$ the restriction of $d(F^n)_x$ to $T_xW$ has already rank $r$. The ambient differential consequently has rank at least $r$, and we obtain $\varrho\geq r$.
\end{proof}

\begin{rem}\label{rem:twisted-eventual-rank}
The generic rank of $F^n$ along $W$ need not equal to $\varrho$ for every $n\geq 1$. For example, consider the PCF endomorphism
\[
F:\bP^3\to\bP^3,\quad
[x:y:z:t]\mapsto[x^2:y^2:z^2-2xt:t^2].
\]
The point $p=[1:1:0:0]$ is fixed and non-superattracting, and its algebraic core is the line $W=\{z=t=0\}$. Along $W$, $F$ has generic rank $2$, whereas $F^n$ has generic rank $1$ for every $n\geq 2$. 
\end{rem}

\subsection{Invariant formal jet}

\begin{lem}\label{lem:twisted-formal-quotient}
Assume that $r<N$. 
There exists a continuous surjective local homomorphism of $L$-algebras $\theta:L[[w]]\to L[[u]]$, where $u$ denotes $(u_1,\ldots,u_{\varrho})$, and a continuous automorphism of $\bC$-algebras
$\Upsilon:L[[u]]\to L[[u]]$ such that
\begin{equation}\label{eq:twisted-quotient-equivariance}
\Upsilon\vert_L=\sigma,\quad \theta\circ\Phi=\Upsilon\circ\theta.
\end{equation}
\end{lem}

\begin{proof}
For $k\geq 0$, we denote $R_k:=L[[w]]/(w)^{k+1}$, and let $\Phi_k$ be the endomorphism induced by $\Phi$ on $R_k$. 
Note that the images of the powers of $\Phi_k$ are $L$-vector subspaces of $R_k$. They form a decreasing sequence, which stabilizes because $R_k$ is finite-dimensional over $L$. Let $S_k$ denote this eventual image. Note also that $S_k$ is an $L$-subalgebra of $R_k$ because $\Phi_k$ is a ring homomorphism and $\sigma$ is surjective. 
For sufficiently large $n$, we have $\Phi_k(S_k)=S_k$. Let $e_1,\ldots,e_l$ be an $L$-basis of $S_k$. We have $\Phi_k(S_k)=S_k$ and $\Phi_k(\sum_i a_i e_i)=\sum_i\sigma(a_i)\Phi_k(e_i)$ for $a_i\in L$. If $\Phi_k(\sum_i a_i e_i)=0$, then $\sigma(a_i)=0$, hence $a_i=0$, for every $i$. In other words $\Phi_k\vert_{S_k}$ is also injective, thus is an automorphism.

For any $n\geq 1$ such that $\operatorname{im}\Phi_k^n=S_k$, define
\begin{equation}\label{eq:twisted-retractions}
\varepsilon_k
=(\Phi_k^n\vert_{S_k})^{-1}\circ\Phi_k^n:
R_k\to S_k\subset R_k.
\end{equation}
Both factors in \eqref{eq:twisted-retractions} are ring homomorphisms, and their actions $\sigma^n$ and $\sigma^{-n}$
on $L$ cancel. Thus $\varepsilon_k$ is a projection of $L$-algebras onto $S_k$. Since $\Phi_k^{n+1}=(\Phi_k\vert_{S_k})\circ\Phi_k^n$,
the expression \eqref{eq:twisted-retractions} does not depend on $n$. Note also that $\varepsilon_k\circ\Phi_k=\Phi_k\circ\varepsilon_k$.

The truncation map $\tau_k:R_{k+1}\to R_k$ commutes with the endomorphisms induced by $\Phi$ and sends $S_{k+1}$ into $S_k$. The restrictions $\Phi_{k+1}\vert_{S_{k+1}}$ and $\Phi_k\vert_{S_k}$ being bijective, the truncation map is also equivariant with respect to their inverses. Hence $\tau_k\circ\varepsilon_{k+1}=\varepsilon_k\circ \tau_k$. As a consequence the family of maps $(\varepsilon_k)_k$ induces a continuous endomorphism of $L$-algebras $\varepsilon:L[[w]]\to L[[w]]$ satisfying $\varepsilon^2=\varepsilon$ and $\varepsilon\circ\Phi=\Phi\circ\varepsilon$. We have
\begin{equation}
\varepsilon(L[[w]])=\varprojlim S_k.
\end{equation}
As $(\Phi_k\vert_{S_k})^{-1}$ are compatible with truncation, their limit is the inverse of $\Phi\vert_{\varepsilon(L[[w]])}$.
Remark also that $\varepsilon((w))\subset(w)$.

By Lemma~\ref{lem:twisted-generic-rank} we have $\dim_L\operatorname{im}\Phi_1^n=\dim_L S_1=1+\varrho$ for every sufficiently large $n$. 
In particular $\varepsilon$ induces a projection of rank $\varrho$
on $(w)/(w)^2$. Let $a_1,\ldots,a_{\varrho}$ and $b_1,\ldots,b_{N-\varrho}$ be $L$-linear forms in $w$ whose classes form bases of the
image and kernel of this projection, respectively. Let $v_1,\ldots,v_{N-\varrho}$ be independent variables, and write $v=(v_1,\ldots,v_{N-\varrho})$. Define a continuous homomorphism of $L$-algebras by
\[
\psi:L[[u,v]]\to L[[w]],\qquad
u_i\mapsto\varepsilon(a_i),\quad
v_j\mapsto b_j-\varepsilon(b_j),
\] 
where $u$ denotes $(u_1,\ldots,u_{\varrho})$. 
We assert that $\psi$ is an isomorphism, because linear terms of $\psi(u_i)$ and $\psi(v_j)$ are respectively $a_i$ and $b_j$. 
Since $\varepsilon^2=\varepsilon$, we have $\varepsilon(\psi(u_i))=\psi(u_i)$ and $\varepsilon(\psi(v_j))=0$.
Since $\varepsilon$ and $\psi$ fix $L$ and are continuous, we obtain
$\varepsilon(\psi(P(u,v)))=\psi(P(u,0))$ for every
$P\in L[[u,v]]$.
Consequently, the restriction $\iota:=\psi\vert_{L[[u]]}:L[[u]]\to\varepsilon(L[[w]])$ is an isomorphism of complete local $L$-algebras. Regard $\varepsilon$ as a surjection onto $\varepsilon(L[[w]])$, and define
\[
\theta:=\iota^{-1}\circ\varepsilon:L[[w]]\to L[[u]],
\quad
\Upsilon:=\iota^{-1}\circ \bigl(\Phi\vert_{\varepsilon(L[[w]])}\bigr)\circ\iota.
\]
Then $\theta$ is continuous, local, and surjective, and
$\Upsilon$ is a continuous $\bC$-algebra automorphism. Since $\iota$ fixes $L$ and $\Phi\vert_L=\sigma$, we obtain $\Upsilon\vert_L=\sigma$. Since $\varepsilon\circ\Phi=\Phi\circ\varepsilon$, we also obtain $\theta\circ\Phi=\Upsilon\circ\theta$.
\end{proof}

Throughout the remainder of this section, we will assume $r<N$ and fix $\theta$ and $\Upsilon$ as in Lemma~\ref{lem:twisted-formal-quotient}.

We define a morphism $\kappa:\Spec L[[u]]\to\bP^N$ in our affine coordinates by  
\begin{equation}\label{eq:generic-chart-morphism}
\bC[z_1,\ldots,z_N]\to L[[u]], \quad R\mapsto\theta\bigl(R(\xi+w)\bigr),
\end{equation}
where $R(\xi+w)$ is considered as an element of $L[[w]]$ (see the second paragraph of \S~\ref{subsubsec:twisted-pullback}). Let $Z$ be the Zariski closure of $\kappa(\Spec L[[u]])$ in $\bP^N$.

\begin{lem}\label{lem:twisted-algebraic-image}
The variety $Z$ is irreducible, contains $W$, satisfies $F(Z)=Z$, and has dimension at least $\varrho$. The morphism $\kappa$ is dominant onto $Z$ and satisfies $F\circ\kappa=\kappa\circ\Spec\Upsilon$.
\end{lem}

\begin{proof}
First remark that the pullbacks of $z=(z_1,\ldots,z_N)$ by $F\circ\kappa$ and $\kappa\circ\Spec\Upsilon$ are,
respectively, $\theta(F(\xi+w))$ and $\Upsilon(\theta(\xi+w))$. By
\eqref{eq:twisted-quotient-equivariance} and \eqref{eq:twisted-Phi}, we have
\[
\Upsilon(\theta(\xi+w))=\theta(\Phi(\xi+w))
=\theta(F(\xi+w)).
\]
Thus we get the equality of morphisms of $\bC$-schemes $F\circ\kappa=\kappa\circ\Spec\Upsilon$.

The kernel of the homomorphism in
\eqref{eq:generic-chart-morphism} is prime because $L[[u]]$ is an integral domain.
Thus $Z$ is irreducible and $\kappa$ is dominant onto $Z$.
Since $\Upsilon$ is an automorphism, we obtain $\kappa(\Spec L[[u]])=F(\kappa(\Spec L[[u]]))$, so $F(Z)=Z$. 

If $R\in\bC[z_1,\ldots,z_N]$ satisfies $\theta(R(\xi+w))=0$, then as a series in $L[[w]]$ its constant term $R(\xi)$ is zero because $\theta$ is local and fixes $L$. Every polynomial equation of $Z$ therefore vanishes on $W$, so $W\subset Z$.

Finally, since $\theta$ is a surjective local homomorphism of $L$-algebras, the induced map $(w)/(w)^2\to(u)/(u)^2$ is surjective.
Let $i_1,\ldots,i_{\varrho}$ be indices such that the classes of $\theta(w_{i_1}),\ldots,\theta(w_{i_{\varrho}})$ form an $L$-basis of $(u)/(u)^2$. Then $\theta(w_{i_1}),\ldots, \theta(w_{i_{\varrho}})$ are algebraically independent over $L$. For each coordinate polynomial $z_i$, by
\eqref{eq:generic-chart-morphism} we have
\[
\kappa^*z_i=\theta(\xi_i+w_i)=\xi_i+\theta(w_i).
\]
Therefore the pullback by $\kappa$ is injective on $\bC[z_{i_1},\ldots,z_{i_{\varrho}}]$. This implies $\dim Z\geq\varrho$.
\end{proof}

\subsection{Family of holomorphic arcs}\label{subsec:family-of-arcs}

Let $\nu:X\to Z$ be the normalization. As $F(Z)=Z$ by Lemma~\ref{lem:twisted-algebraic-image}, we obtain an endomorphism $g:X\to X$ induced by $F$. Define $B:=\Sing X\cup\overline{\bigcup_{n\geq1}\operatorname{Br}(g^n)}^{\mathrm{Zar}}$. Recall from Lemma~\ref{lem:normalization-boundary} that $B$ is proper and forward invariant under $g$.

\begin{lem}\label{lem:transverse-specialization}
Let $s$ be an indeterminate.
For every analytic neighborhood $V$ of $p$ in $\bP^N$,
there exist a sequence $(x_n)_{n\geq0}$ of points of $W$
and morphisms of $\bC$-schemes
$\gamma_n:\Spec\bC[[s]]\to X$, for $n\geq0$, with the
following properties.
The points $x_n$ lie in one compact subset of $V$,
the morphism $\gamma_0$ maps $\Spec\bC((s))$ into
$X\setminus B$, and for every $n\geq0$ we have
\[
F(x_{n+1})=x_n,\qquad
g\circ\gamma_{n+1}=\gamma_n,\qquad
\nu(\gamma_n(0))=x_n.
\]
Here $\gamma_n(0)$ denotes the image under $\gamma_n$ of
the point of $\Spec\bC[[s]]$ corresponding to the maximal
ideal $(s)$.
\end{lem}

\begin{rem}\label{rem:transverse-specialization}
In Lemma~\ref{lem:transverse-specialization} the subtlety is about lifted formal arcs parametrized by $\bC[[s]]$: the existence of a
backward orbit in a small neighborhood does not ensure the required formal lifts over $\bC[[s]]$. This can be illustrated by the following (counter) example. 
Consider the map $G(t,y)=(2t,y^2)$ on $\bC^2$, whose branch locus
is $\{y=0\}$. Let $\alpha(s)=(t(s),y(s))$, with $t(s),y(s)\in\bC[[s]]$ and $\alpha(0)=(t_0,0)$.
Suppose that, for every $n\geq1$, a pair $\alpha_n(s)$ of formal power series satisfies $G^n\circ\alpha_n=\alpha$.
Then $\alpha_n(0)=(2^{-n}t_0,0)$, so they remain near the origin when $t_0$ is small. Writing $y_n(s)$ for the second coordinate of $\alpha_n$, we have $y_n(s)^{2^n}=y(s)$. 
If $y\neq 0$, then the positive integer $\operatorname{ord}_0y$
would be divisible by $2^n$ for every $n$, which is impossible.
Thus $\alpha$ must be contained in the branch locus. 
\end{rem}

\begin{proof}
Fix an analytic neighborhood $V$ of $p$ in $\bP^N$. Recall that $h=F\vert_M$ is the biholomorphic germ fixing $p$ on the smooth irreducible component $M$ of the germ $(W,p)$. We apply Lemma~\ref{lem:fixed-inverses} to $F$, $W$, and $M$, and obtain holomorphic inverse maps $h^{-n}:V_M\to W$, for $n\geq 0$, defined on a common connected neighborhood $V_M$ of $p$ in $M$. These maps form a normal family, fix $p$, and satisfy $F\circ h^{-(n+1)}=h^{-n}$. In particular, all the images $h^{-n}(V_M)$ lie in one compact subset of $V$.

By Theorem~\ref{thm:algebraic-core}, the point of the normalization of $W$ corresponding to $M$ is smooth and fixed, and the endomorphism induced by $F$ is locally biholomorphic there.
We apply Lemma~\ref{lem:fixed-inverses} at this point and compose those inverse maps with the normalization map. We then obtain holomorphic maps $h^{-n}:V_M\to W$ defined on a common connected neighborhood $V_M$ of $p$ in $M$, which form a normal family, fix $p$, and satisfy $F\circ h^{-(n+1)}=h^{-n}$. The family being normal, we can take $V_M$ sufficiently small so that all the images $h^{-n}(V_M)$ lie in one compact subset of $V$.

For $a\in K$ and $n\geq0$, the element $\sigma^{-n}(a)$ corresponds by definition to the meromorphic function $a\circ h^{-n}$ defined on $V_M$. We will think of $L=\bigcup_{n\geq0}\sigma^{-n}K$
as a field of meromorphic functions on $V_M$.

By Lemma~\ref{lem:twisted-algebraic-image}, $\kappa$ is
dominant onto $Z$ and $F\circ\kappa=\kappa\circ\Spec\Upsilon$.
Since $\Spec L[[u]]$ is normal, $\kappa$ has a unique
lift $\widetilde\kappa:\Spec L[[u]]\to X$, which is dominant.
We have the commutative diagram
\[
\begin{tikzcd}[column sep=large]
\Spec L[[u]] \arrow[r, "\widetilde\kappa"]
  \arrow[d, "\Spec\Upsilon"']
& X \arrow[r, "\nu"] \arrow[d, "g"]
& Z \arrow[d, "F"] \\
\Spec L[[u]] \arrow[r, "\widetilde\kappa"']
& X \arrow[r, "\nu"'] & Z.
\end{tikzcd}
\]
For $n\geq 0$, define
$\widetilde\kappa_n:=\widetilde\kappa\circ\Spec(\Upsilon^{-n})$.
Then
\begin{equation}\label{eq:generic-inverse-morphisms}
g\circ\widetilde\kappa_{n+1}=\widetilde\kappa_n,
\quad n\geq 0.
\end{equation}

Let $Z_0$ be the intersection of $Z$ with our fixed affine
chart of $\bP^N$, and let $X_0:=\nu^{-1}(Z_0)$. Since $\nu$ is finite, $X_0$ is affine. 
Note that all the morphisms $\widetilde\kappa_n$ take values in $X_0$. 
Let $b\in \bC[X_0]$ be nonzero and vanish on $B\cap X_0$. By dominance of $\widetilde\kappa$, we have $\widetilde\kappa^*(b)\in L[[u]]$ and $\widetilde\kappa^*(b)\neq 0$. Let $c_*\in L$ be one of the nonzero coefficients of this formal series.

Let $a_1,\ldots,a_l$ generate $\bC[X_0]$ as a $\bC$-algebra.
The coefficients of the series $\widetilde\kappa_n^*(a_j)$, for all $n\geq0$ and $1\leq j\leq l$, form a countable family of meromorphic functions on $V_M$. 
Their nonholomorphic loci, together with the zero and pole sets of $c_*$, form a countable union of proper analytic subsets. By Baire's theorem, there is a point $x\in V_M$ outside this union.
Then every coefficient of $\widetilde\kappa_n^*(a)\in L[[u]]$, for $a\in \bC[X_0]$, is an element of $L$ which is holomorphic at $x$.
Therefore $\widetilde\kappa_n^*$ induces a homomorphism $\bC[X_0]\to\cO_{V_M,x}[[u]]$, where $\cO_{V_M,x}$ is the ring of holomorphic germs at $x$.

Let $\operatorname{ev}_x:\cO_{V_M,x}[[u]]\to\bC[[u]]$
be coefficientwise evaluation at $x$.
Since $c_*(x)\ne0$, the series
$\operatorname{ev}_x(\widetilde\kappa^*(b))$ is nonzero.
Let $v\in\bC^{\varrho}$ be a point where its first nonzero
homogeneous term does not vanish.
For each $n\geq0$, define a morphism
$\gamma_n:\Spec\bC[[s]]\to X_0\subset X$ by the homomorphism
\begin{equation}\label{eq:specialization-series}
\gamma_n^*:\bC[X_0]\to\bC[[s]],\qquad
\gamma_n^*(a):=
\bigl(\operatorname{ev}_x(\widetilde\kappa_n^*(a))\bigr)(sv).
\end{equation}
By the choice of $v$, we have $\gamma_0^*(b)\neq 0$, so $\gamma_0$ maps $\Spec\bC((s))$ outside $B$. 

The equation $g(y')=y$ defines a closed subvariety of $X_0\times X_0$. By \eqref{eq:generic-inverse-morphisms}, the series
$\widetilde\kappa_{n+1}^*(a_j)$ and $\widetilde\kappa_n^*(a_j)$ satisfy these equations. Hence $g\circ\gamma_{n+1}=\gamma_n$
for every $n\geq 0$.

Recall that $\theta$ is local and fixes $L$, while $\Upsilon$
preserves the ideal $(u)$ and acts on $L[[u]]/(u)=L$ by $\sigma$.
By \eqref{eq:generic-chart-morphism}, the constant term of $\widetilde\kappa_n^*(\nu^*z_i)$ is $\sigma^{-n}(\xi_i)\in L$.
By \eqref{eq:specialization-series}, the constant term of
$\gamma_n^*(\nu^*z_i)$ is therefore $\sigma^{-n}(\xi_i)(x)=z_i(h^{-n}(x))\in \bC$. Define $x_n:=h^{-n}(x)$ for every $n\geq 0$. Then $\nu(\gamma_n(0))=x_n$ and $F(x_{n+1})=x_n$. Since $x\in V_M$, all the points $x_n$ lie in the common compact subset of $V$ containing the images $h^{-n}(V_M)$.
\end{proof}

\begin{lem}\label{lem:formal-inverse-family}
Fix $(x_n)$ and $(\gamma_n)$ as in Lemma~\ref{lem:transverse-specialization}. 
For every analytic neighborhood $V$ of $\nu(\gamma_0(0))$
in $\bP^N$, there exists a holomorphic map $J:\bD\times\bD^{\dim Z}\to Z\cap V$ of generic rank $\dim Z$ such that $J(0,c)=\nu(\gamma_0(0))$ for all $c\in\bD^{\dim Z}$. For every $n\geq1$ and $c\in\bD^{\dim Z}$, there exists a holomorphic map $\beta_{n,c}:\bD\to\bP^N$ such that $\beta_{n,c}(0)=\nu(\gamma_n(0))$ and $F^n(\beta_{n,c}(s))=J(s,c)$ for all $s\in\bD$.
\end{lem}

\begin{rem}
Lemma~\ref{lem:formal-inverse-family} is a upgraded version of Lemma~\ref{lem:core-inverse-family} where the holomorphic family of arcs were centered at a fixed point $p$.  
\end{rem}

\begin{proof}
Let $V$ be an analytic neighborhood of $\nu(\gamma_0(0))$ in
$\bP^N$. Let $\pi:Y\to X$ be an embedded resolution of $B$ which is an isomorphism outside $B$, such that $D=\pi^{-1}(B)$ is a simple normal crossing divisor. 
Since $\gamma_0$ maps $\Spec\bC((s))$ into $X\setminus B$,
its restriction to $\Spec\bC((s))$ lifts uniquely through $\pi$ and extends to $\delta:\Spec\bC[[s]]\to Y$ satisfying $\pi\circ\delta=\gamma_0$. Let $y_0:=\delta(0)$.

Let $P\subset Y$ be a small polydisk neighborhood of $y_0$, with coordinates $u_1,\ldots,u_{\dim Z}$ centered at $y_0$, such that
$\nu(\pi(P))\subset V$ and $D\cap P=\{u_1\cdots u_k=0\}$ for some $k$. The morphism $\delta$ induces a homomorphism of local rings
$\delta^*:\cO^{\rm alg}_{Y,y_0}\to\bC[[s]]$ which extends uniquely to a continuous homomorphism $\widehat\delta^*:\widehat{\cO}_{Y,y_0}\to\bC[[s]]$, where we identify the completions of the algebraic and analytic local rings of $Y$ at $y_0$. For $1\leq i\leq k$, we denote $a_i:=\operatorname{ord}_s\widehat\delta^*(u_i)$, which is positive and finite because $\delta(0)=y_0$ and $\delta$ maps $\Spec\bC((s))$ into $Y\setminus D$. For $k<i\leq\dim Z$, we use the convention $a_i=1$.

For sufficiently small $\epsilon>0$, define a family of holomorphic arcs $\Psi:\bD\times\bD^{\dim Z}\to P$, with disk coordinate $s$ and parameter coordinate $c=(c_1,\ldots,c_{\dim Z})$, by
\begin{equation}\label{eq:moving-family-arcs}
u_i(\Psi(s,c))=\epsilon(1+c_i)s^{a_i},\qquad 1\leq i\leq\dim Z.
\end{equation}
Since $1+c_i\neq 0$, each map $\Psi(\,\cdot\,,c)$ is centered
at $y_0$, avoids $D$ away from zero, and has contact orders
$a_1,\ldots,a_k$ with $D\cap P$.
Define 
\begin{equation}\label{eq:definition-J}
J:=\nu\circ\pi\circ\Psi.
\end{equation}
Then $J$ takes values in $Z\cap V$ and $J(0,c)=\nu(\gamma_0(0))$.
For fixed $s\ne0$, the map $c\mapsto\Psi(s,c)$ is biholomorphic
onto an open subset of $P$ by \eqref{eq:moving-family-arcs}.
Since $\nu\circ\pi$ is birational, $J$ has generic rank $\dim Z$.

Fix $n\geq1$. Let $Q_n$ be the fiber product of $\pi:Y\to X$ and $g^n:X\to X$. Since $\pi$ is birational, $Q_n$ has a unique component dominating $Y$. Denote its normalization by $Y_n$. The projections of $Q_n$ induce maps $\rho_n:Y_n\to Y$ and
$\tau_n:Y_n\to X$. We have $\pi\circ\rho_n=g^n\circ\tau_n$, and $\rho_n$ is finite, surjective and unramified outside $D$. 
 
The pair $(\delta,\gamma_n)$ lifts over $\Spec\bC((s))$ to $Y_n$ and extends to a morphism $\lambda_n:\Spec\bC[[s]]\to Y_n$ such that the following commutative diagram holds:
\begin{equation}\label{eq:formal-family-projections}
\begin{tikzcd}[column sep=large, row sep=small]
  \Spec\bC[[s]]
    \arrow[dr, "\lambda_n"]
    \arrow[drr, bend left=18, "\delta"]
    \arrow[ddr, bend right=18, "\gamma_n"'] & & \\
  & Y_n \arrow[r, "\rho_n"] \arrow[d, "\tau_n"']
  & Y \arrow[d, "\pi"] \\
  & X \arrow[r, "g^n"'] & X
\end{tikzcd}.
\end{equation}
The morphism $\lambda_n$ corresponds to a homomorphism $\lambda_n^*:\cO^{\rm alg}_{Y_n,\lambda_n(0)}\to\bC[[s]]$. It extends uniquely to a continuous homomorphism $\widehat\lambda_n^*:\widehat{\cO}_{Y_n,\lambda_n(0)}\to\bC[[s]]$.

Remark that every irreducible component of $\rho_n^{-1}(P)$ has dimension $\dim Z$ and dominates $P$ because $\rho_n$ is finite. Fix $c\in\bD^{\dim Z}$. By \eqref{eq:moving-family-arcs}, we have
$\operatorname{ord}_0(u_i\circ\Psi(\,\cdot\,,c))
=a_i=\operatorname{ord}_s\widehat\lambda_n^*(u_i\circ\rho_n)$
for $1\leq i\leq k$.
Moreover, $\Psi(0,c)=y_0=\rho_n(\lambda_n(0))$, and
$\Psi(s,c)\notin D$ for $s\ne0$.
We can therefore apply Lemma~\ref{lem:arc-transfer} to
$\rho_n:\rho_n^{-1}(P)\to P$, with the homomorphism
$\widehat\lambda_n^*$ and the map $\Psi(\,\cdot\,,c)$.
We obtain a holomorphic map
$\widetilde\Psi_{n,c}:\bD\to\rho_n^{-1}(P)$ satisfying
$\rho_n\circ\widetilde\Psi_{n,c}=\Psi(\,\cdot\,,c)$ and
$\widetilde\Psi_{n,c}(0)=\lambda_n(0)$. 
Define $\beta_{n,c}:=\nu\circ\tau_n\circ\widetilde\Psi_{n,c}$.
Then
\begin{align*}
F^n\circ\beta_{n,c}
& =\nu\circ g^n\circ\tau_n\circ\widetilde\Psi_{n,c}=\nu\circ\pi\circ\rho_n\circ\widetilde\Psi_{n,c}
\\
& =\nu\circ\pi\circ\Psi(\,\cdot\,,c)
=J(\,\cdot\,,c).
\end{align*}
By \eqref{eq:formal-family-projections}, we also have
$\beta_{n,c}(0)=\nu(\tau_n(\lambda_n(0)))=\nu(\gamma_n(0))$.
\end{proof}

\subsection{Proof of transverse superattraction}

\begin{proof}[Proof of Theorem~\ref{thm:transverse}]
Recall that $F=f^a$ was fixed at the beginning of the section.
If $r=N$, then $W=\bP^N$ and we are done. Assume henceforth that $r<N$. By Lemmas~\ref{lem:twisted-generic-rank} and \ref{lem:twisted-algebraic-image}, we have $W\subset Z$ and $r\leq\varrho\leq\dim Z$.

Let $e=(e_1,\ldots,e_{N-r})$ be holomorphic local equations for $M$, with linearly independent differentials at $p$. The functions $e_1,\ldots,e_{N-r}$ generate the ideal of holomorphic germs vanishing on $M$. Since $F(M)=M$, each $e_i\circ F$ belongs to this ideal. Thus there is a holomorphic matrix $C$ near $p$ such that
\begin{equation}\label{eq:analytic-contraction}
e(F(z))=C(z)e(z).
\end{equation}
By Theorem~\ref{thm:algebraic-core}, $T_pM$ is the sum of the generalized eigenspaces of $d_pF$ for nonzero eigenvalues. Hence $d_pF$ induces a nilpotent endomorphism of $T_p\bP^N/T_pM$.
Differentiating \eqref{eq:analytic-contraction} at $p$, we obtain
$d_pe\circ d_pF=C(p)\circ d_pe$. 
Since $d_pe$ induces an isomorphism between $T_p\bP^N/T_pM$ and $\bC^{N-r}$, the scalar matrix $C(p)$ is nilpotent.
There is therefore a norm on $\bC^{N-r}$ whose induced operator norm satisfies $\lVert C(p)\rVert<1$.
Let $c_0$ satisfy $\lVert C(p)\rVert<c_0<1$.
Let $V$ be a relatively compact neighborhood of $p$ in the
fixed affine chart of $\bP^N$ on which \eqref{eq:analytic-contraction} holds, $e$ is bounded, and $\lVert C(z)\rVert\leq c_0$.
We can also assume that $M$ is defined by $e=0$ on $V$.

Let $(x_n)_{n\geq0}$ and $(\gamma_n)_{n\geq0}$ be the
sequences supplied by Lemma~\ref{lem:transverse-specialization}
for this $V$. 
By Lemma~\ref{lem:formal-inverse-family}, there is a
holomorphic map $J:\bD\times\bD^{\dim Z}\to Z\cap V$
of generic rank $\dim Z$ such that $J(0,c)=x_0$ for every
$c\in\bD^{\dim Z}$.
For each $n\geq1$ and $c\in\bD^{\dim Z}$, we also obtain
a holomorphic map $\beta_{n,c}:\bD\to\bP^N$ satisfying
$F^n\circ\beta_{n,c}=J(\,\cdot\,,c)$ and
$\beta_{n,c}(0)=x_n$.

Fix $c\in\bD^{\dim Z}$. Consider the holomorphic maps $F^j\circ\beta_{n,c}:\bD\to\bP^N$ as the integers $n\geq1$ and $0\leq j\leq n$ vary. For every such pair $(n,j)$, we have $F^{n-j}\circ(F^j\circ\beta_{n,c})=J(\,\cdot\,,c)$. By Lemma~\ref{lem:inverse-normality}, this family is normal
on $\bD$ and hence equicontinuous at zero.
Since $(F^j\circ\beta_{n,c})(0)=x_{n-j}$ and the points $x_k$, for $k\geq0$, lie in one compact subset of $V$, there is a disk $\Delta_c\subset\bD$, centered at zero, such that $(F^j\circ\beta_{n,c})(\Delta_c)\subset V$ for every $n\geq1$ and $0\leq j\leq n$.

For $n\geq1$ and $s\in\Delta_c$, we apply \eqref{eq:analytic-contraction} successively at the points $F^j(\beta_{n,c}(s))$, for $j=0,\ldots,n-1$. Using $F^n\circ\beta_{n,c}=J(\,\cdot\,,c)$, we obtain
\begin{equation}\label{eq:transverse-arc-estimate}
\lVert e(J(s,c))\rVert
\leq c_0^n\sup_{z\in V}\lVert e(z)\rVert,
\qquad s\in\Delta_c,\quad n\geq1.
\end{equation}
Letting $n\to\infty$, we obtain $e(J(s,c))=0$ on $\Delta_c$.
By analytic continuition, we get $e(J(s,c))=0$ for every $s\in\bD$.
As this holds for every $c\in\bD^{\dim Z}$, the map $J$ takes values
in $M$ and has rank at most $r$. Since $J$ has generic rank $\dim Z$ by Lemma~\ref{lem:formal-inverse-family}, we obtain $\dim Z\leq r$. Recall also that $r\leq\varrho$ by Lemma~\ref{lem:twisted-generic-rank} and that $\varrho\leq\dim Z$ by Lemma~\ref{lem:twisted-algebraic-image}. Thus $\varrho=r$. Then Lemma~\ref{lem:twisted-generic-rank} asserts that for every sufficiently large $n$, the differential $d(F^n)_x$ has rank $r$ on a nonempty Zariski-open subset $W_n\subset W$, as required. 
\end{proof}

\subsection{Finiteness of algebraic cores}\label{subsec:finiteness}

Recall the resolution $\pi:Y\to\bP^N$ of $(\bP^N,P_f)$ and the reduced divisor $D=\pi^{-1}(P_f)_{\rm red}$ from Subsection~\ref{subsec:limits-iterates}. 
Recall also that, for $0\leq r<N$, the set $E_{\pi,r}$ in \eqref{eq:resolution-rank-loci} is the union, over all nonempty intersections $T$ of components of $D$, of the sets $\pi(\{y\in T\mid\operatorname{rank}d_y(\pi\vert_T)\leq r\})$.

\begin{prop}\label{prop:finite-cores}
Every $r$-dimensional algebraic core of $f$ with $r<N$ is an irreducible component of $E_{\pi,r}$. In particular, the periodic points of $f$ have only finitely many distinct algebraic cores, and the union of the cores of dimension less than $N$ is contained in the postcritical set $P_f$.
\end{prop}
\begin{proof}
Let $W$ be the algebraic core of a periodic point $p$, with $r=\dim W<N$. By Theorem~\ref{thm:transverse} there is an iterate $F$ of $f$ such that $F(W)=W$ and $\operatorname{rank}dF\leq r$ everywhere on $W$. As $r<N$ we have $W\subset\Crit(F)$. Hence $W=F(W)\subset P_f$. 

Let $z\in W$, let $x=F(z)$, and fix $y\in\pi^{-1}(x)$.
Let $T$ be the intersection of the components of $D$ through $y$.
In a sufficiently small analytic polydisk $P$ centered at $y$, by construction $D\cap P$ is a union of coordinate hyperplanes and $T\cap P$ is their common intersection.

As in the proof of Lemma~\ref{lem:inverse-extension}, let $Z$ be the
normalization of the fiber product of $\pi\vert_P$ and $F$. The induced projections $\rho:Z\to P$ and $q:Z\to\bP^N$ satisfy $F\circ q=\pi\circ\rho$. Remark that every component of $Z$ dominates $P$, and that over $P\setminus D$ the fiber product is already an
unramified cover. Thus the finite map $\rho$ is unramified there as well.

The pair $(y,z)$ belongs to the fiber product since $\pi(y)=F(z)$.
Let $\xi\in Z$ lie over this pair. Lemma~\ref{lem:local-power-coordinates} gives a holomorphic
section $s:T\cap P\to Z$ of $\rho$ with $s(y)=\xi$.
\[
\begin{tikzcd}
  & Z \ar[r, "q"] \ar[d, "\rho"] & \bP^N \ar[d, "F"] \\
T \cap P \ar[r, hook] \ar[ru, bend left, "s"] & P \ar[r, "\pi|_P"] & \bP^N
\end{tikzcd}
\]
Composing the inverse with $q$, we obtain a holomorphic map $\beta:T\cap P\to\bP^N$
with $\beta(y)=z$ and $F\circ\beta=\pi\vert_{T\cap P}$. Thus
\[
\operatorname{rank}d_y(\pi\vert_T)
=\operatorname{rank}(d_zF\circ d_y\beta)\leq r.
\]
It follows from \eqref{eq:resolution-rank-loci} that $x\in E_{\pi,r}$.

We have proved $W\subset E_{\pi,r}$. Since $\dim E_{\pi,r}\leq r$ by
Proposition~\ref{prop:log-rank}, $W$ is an irreducible component. 
\end{proof}

\section{Transverse superattraction, ideal sheaves and smoothness}\label{sec:ideal-sheaves}

\subsection{Ideal sheaves and Fatou limit maps}
For an irreducible algebraic subvariety $W\subset\bP^N$, let
$\mathcal I_W$ denote its algebraic ideal sheaf and let $\eta_W$ be its generic point. Recall that, for $k\geq 1$, the symbolic power
$\mathcal I_W^{(k)}$ consists, on each Zariski-open subset of $\bP^N$ intersecting $W$, of the regular functions whose germs at $\eta_W$ belong to $(\mathcal I_W\cO^{\rm alg}_{\bP^N,\eta})^k$. Equivalently, its sections vanish to order at least $k$ along $W$ at general smooth points. We refer to \cite[\S~11.3]{Lazarsfeld2} for background on symbolic powers of ideal sheaves.

Let $W\subset\bP^N$ be the algebraic core associated with a periodic point of $f$, constructed in Theorem~\ref{thm:algebraic-core}. The following is a slightly stronger and more compact version of Theorem~\ref{thm:transverse}:

\begin{thm}\label{thm:ideal-square}
There is an integer $a\geq1$ such that $f^a(W)=W$ and $(f^a)^*\mathcal I_W\subset\mathcal I_W^2$.
\end{thm}
\begin{proof}
If $W=\bP^N$, then $\mathcal I_W=0$ and we may take $a=1$. Suppose that $W\neq \bP^N$. By Theorem~\ref{thm:transverse}, there is an iterate $F=f^m$ preserving $W$ whose differential has generic rank $\dim W$ along $W$. In particular, the differential of $F\vert_W$ has rank $\dim W$ at a general smooth point $x\in W$ with smooth image. Thus
\[
\operatorname{im}d_xF=T_{F(x)}W,
\]
and pullback by $F$ vanishes on the conormal space of $W$.

Denote $\eta$ the generic point of $W$. Since $F(\eta)=\eta$, pullback induces a local endomorphism of the regular local ring $\cO^{\rm alg}_{\bP^N,\eta}$. As $\mathfrak m_\eta/\mathfrak m_\eta^2$ is the generic fiber of the conormal sheaf $\mathcal I_W/\mathcal I_W^2$, the above conormal vanishing argument implies that this endomorphism sends the maximal ideal $\mathfrak m_\eta$ into $\mathfrak m_\eta^2$. Its $n$-th iterate therefore sends $\mathfrak m_\eta$ into $\mathfrak m_\eta^{2^n}$. By the definition of symbolic powers, we get
\begin{equation}\label{eq:transverse-symbolic} 
F^*\mathcal I_W\subset\mathcal I_W^{(2)},\quad
(F^n)^*\mathcal I_W\subset\mathcal I_W^{(2^n)},
\quad n\geq1.
\end{equation}
By Ein--Lazarsfeld--Smith's theorem~\cite[Theorem A]{ELS2001} (see also \cite[Th.~11.3.4]{Lazarsfeld2}), we have $\mathcal I_W^{(Nk)}\subset\mathcal I_W^k$ for every $k\geq1$. Choose $n$ such that $2^n\geq 2N$. Then
\[
(F^n)^*\mathcal I_W \subset\mathcal I_W^{(2^n)}
\subset\mathcal I_W^{(2N)} \subset\mathcal I_W^2.
\]
\end{proof}

Let $b\geq 1$ satisfy $f^b(W)=W$, and let $\Omega$ be a Fatou
component of $f$. Recall that $\cH_\Omega$ is the set of Fatou limits maps introduced in Section~\ref{sec:preparation}.

\begin{prop}\label{prop:capture}
If a constant map in $\cH_\Omega(f^b)$ has value in $W$, then
every map in $\cH_\Omega(f^b)$ has image in $W$. 
\end{prop}
\begin{proof}
If $W=\bP^N$, then there is nothing to prove. Suppose that
$W\neq\bP^N$. Let $a$ be as in Theorem~\ref{thm:ideal-square}. Set $F:=f^b$ and $G:=f^{ab}=F^a$. By Theorem~\ref{thm:ideal-square} we have $G^*\mathcal I_W\subset\mathcal I_W^2$.

Let $P_1,\ldots,P_s$ be homogeneous generators of the ideal
of $W$, with $\deg P_i=d_i$. Define, for $z\in\bC^{N+1}$ with $[z]\in \bP^N$,
\[
\delta_W([z]) :=\max_{1\leq i\leq s}\frac{\lvert P_i(z)\rvert}{\lVert z\rVert^{d_i}},
\]
where $\lVert z\rVert$ is the Euclidean norm on $\bC^{N+1}$. This expression is independent of the representative $z$, and defines a continuous function $\bP^N\to \bR_+$ whose zero set is $W$. Remark that, on the affine chart $z_j\neq 0$, the functions $P_i(z)/z_j^{d_i}$ generate $\mathcal I_W$. The corresponding terms in $\delta_W$ are their moduli multiplied by the factors $(|z_j|/\lVert z\rVert)^{d_i}$, which are locally bounded above and away from zero.
Thus $\delta_W$ is locally comparable to the maximum of these moduli.
 
By $G^*\mathcal I_W\subset\mathcal I_W^2$ and by compactness of $W$ we obtain a constant $A>0$ such that
\begin{equation}\label{eq:ideal-norm-contraction}
\delta_W(Gz)\leq A\delta_W(z)^2
\end{equation}
on a neighborhood of $W$. We pick $\varepsilon>0$ so that $A\varepsilon\leq1/2$ and $\{\delta_W<\varepsilon\}$ lies in
this neighborhood.

Let $K\subset\Omega$ be a compact subset. As a sequence of iterates of $F$ converges on $K$ to a constant with value in $W$, there is
$n_0\geq 0$ such that $\delta_W(F^{n_0}(z))<\varepsilon$ for all $z\in K$. By \eqref{eq:ideal-norm-contraction} and induction, we get
\begin{equation}\label{eq:delta-estimate-iteration}
\delta_W(F^{n_0+an}(z)) =\delta_W(G^n(F^{n_0}(z))) \leq2^{-n}\varepsilon 
\quad  \text{for every } z\in K,\ n\in \bN.
\end{equation}
Hence for each $0\leq j <a$, the sequence of functions $(\delta_W\circ F^{n_0+an+j})_n$ converges to $0$ uniformly on $K$.
Thus we proved that $\delta_W\circ F^n$ converges to $0$ locally uniformly on $\Omega$. In other words, every map in $\cH_\Omega(F)$ has image in $W$. 
\end{proof}

\begin{prop}\label{prop:capture2}
The algebraic core at every non-superattracting periodic point of $W$ is also contained in $W$.
\end{prop}
\begin{proof}
Let $q\in W$ be a non-superattracting periodic point, and let $M_q$ be its smooth local core, i.e. a smooth local branch of the core of $q$. We fix a multiple $k$ of $a$ such that $f^k$ fixes $q$ and preserves $M_q$. Then $h:=f^k\vert_{M_q}$ is a local automorphism by Theorem~\ref{thm:algebraic-core}. Recall that $(f^k)^*\mathcal I_W\subset\mathcal I_W^2$ by Theorem~\ref{thm:ideal-square}. Let $J:=(\mathcal I_W)_q\cO_{M_q,q}$ be the ideal generated by the restrictions of local equations of $W$ to $M_q$. Since $q\in W$,
we have $J\subset\mathfrak m_{M_q,q}$, and $h^*J\subset J^2$. 

If $J\neq 0$, let $d\geq1$ be the least vanishing order of its
nonzero elements. As a local automorphism preserves this vanishing order, the least order in $h^*J$ is also $d$. Every element of $J^2$
has order at least $2d$, contradicting $h^*J\subset J^2$.
Thus $J=0$ and $M_q\subset(W,q)$. By Corollary~\ref{cor:core-independence}, the algebraic core at $q$ is the Zariski closure of $M_q$, so it is contained in $W$.
\end{proof}

\begin{cor}\label{cor:core-singular-points}
If $q$ is a periodic point of $f$ in $\Sing W$, then its algebraic core is strictly contained in $W$.
\end{cor}
\begin{proof}
Let $W_q$ be the algebraic core of $q$. We have $W_q\subseteq W$ by
Proposition~\ref{prop:capture2}. By Theorem~\ref{thm:core-smoothness}, $W_q$ is smooth at $q$. Since $W$ is singular at $q$, the inclusion is strict.
\end{proof}

\subsection{Local smoothness of algebraic cores}\label{subsec:core-smoothness}

Let $p$ be a periodic point of $f$, and let $W$ be its algebraic core, of dimension $r$. Let $M$ be the smooth component of the analytic germ $(W,p)$ specified in Theorem~\ref{thm:algebraic-core}. We fix an iterate $F$ of $f$ which fixes $p$, preserves $W$, and induces an automorphism of $M$.

\begin{lem}\label{lem:branches}
The germ $M$ is the only irreducible branch of $(W,p)$ that is periodic under $F$. Every irreducible branch of $(W,p)$ eventually maps onto $M$ under $F$.
\end{lem}
\begin{proof}
Let $L$ be an irreducible branch of $(W,p)$ that is periodic under $F$. By Theorem~\ref{thm:ideal-square} up to replacing $F$ by a further iterate we can assume that $F^*\mathcal I_W\subset\mathcal I_W^2$ and $F(L)=L$.

Suppose that $L\neq M$. Then there is a holomorphic function $u$ defined near $p$ which vanishes on a representative of $L$ and whose
restriction to $M$ is nonzero. Let $d$ be the vanishing order of the restriction $u\vert_M$ at $p$. Let $n$ be an interger such that $2^n>d$. We have $(F^n)^*\mathcal I_W\subset\mathcal I_W^{2^n}$.

Since $F^n$ preserves $L$ and is dominant there, for a general
$x\in L$ which is close to $p$, the point $F^n(x)$ is in the zero locus of $u$ and lies on no other local branch of $W$ at $p$. Near $F^n(x)$, we therefore have $W=L$.
Thus the germ of $u$ at $F^n(x)$ belongs to the analytic stalk
$(\mathcal I_W)_{F^n(x)}$, and
\[
u\circ F^n\in(\mathcal I_W)_x^{2^n}\subset\mathfrak m_{\bP^N,x}^{2^n}.
\]
All partial derivatives of $u\circ F^n$ of total order less than $2^n$
vanish at these points $x$ and, by continuity, at $p$. Since
$F^n\vert_M$ is a local automorphism, we obtain
\[
d=\operatorname{ord}_p\bigl((u\circ F^n)\vert_M\bigr)
\geq2^n>d,
\]
a contradiction. Hence $L=M$.

The map $F$ sends each irreducible branch of $(W,p)$ onto another branch. Every orbit in this finite set of branches eventually reaches a periodic branch, hence reaches $M$.
\end{proof}

\begin{thm}\label{thm:core-smoothness}
The variety $W$ is smooth at $p$. 
\end{thm}

\begin{proof}
If $r=0$, then $W=\{p\}$, and if $r=N$, then $W=\bP^N$. In both cases $W$ is smooth at $p$. We assume thus $0<r<N$. 
By Lemma~\ref{lem:branches}, it suffices to prove $F_p^{-1}(M)=M$, where $F_p$ denotes the germ of $F$ at $p$. As $F_p(M)=M$,
$M\subset F_p^{-1}(M)\subset(F_p^k)^{-1}(M)$ for every $k\geq1$. By construction (see Theorem~\ref{thm:algebraic-core}) $T_pM$ is the sum of the generalized eigenspaces of $d_pF$ for nonzero eigenvalues. Up to replacing $F$ by a further iterate, we will assume that $d_pF$ vanishes on its generalized eigenspace for zero. 

We choose coordinates $(x,y)$ centered at $p$ so that $M$ is defined by $y=0$. We have $d_pF(\xi,\eta)=(A\xi,0)$ for some invertible matrix $A$. We will use the maximum norm on the two coordinates and fix a constant $0<c<1$ such that $2c\lVert A^{-1}\rVert<1$. As the map $(x,y)\mapsto F(x,y)-(Ax,0)$ has zero differential at $p$, its Lipschitz constant on some ball $B(p,3R)$ is less than $c$.

Suppose for contradiction that $F_p^{-1}(M)$ has a component $L\neq M$. We have $F_p(L)=M$. Let $V$ be a neighborhood of $p$ in $M$, small enough so that $V\subset F(L\cap B(p,R))$. By Lemma~\ref{lem:fixed-inverses}, the inverse branches $b_j$ of the restrictions of iterates $F^j\vert_M$ send a common neighborhood $U$ of $p$ into $V$. As the sets $b_j^{-1}(L\cap M)$ are proper analytic subsets, some $z\in U$ avoids them all. Then we consider
\[
z_j:=b_j(z)\notin L, \quad z_0=z, \quad F(z_{j+1})=z_j.
\]

We apply \cite[Lemma~3.8(ii)]{Ueda1998} at $z$ to the postcritical locus $P_f\subset\bP^N$. There is a neighborhood $O$ of $z$ such that, for every $n\geq 1$, distinct preimages of $z$ under $F^n$ lie in distinct components of $(F^n)^{-1}(O)$. 
Fix $n$ sufficiently large so that $B(z,2Rc^n)\subset O$. Since $z_{n-1}\in V$, there is $w\in L\cap B(p,R)$ such that $F(w)=F(z_n)=z_{n-1}$. Thus $F^n(w)=F^n(z_n)=z$. However $w\neq z_n$ because $z_n\notin L$.

By the choice of $O$, the two preimages $z_n$ and $w$ lie in different components of $(F^n)^{-1}(O)$. Thus a path from $z_n$ to $w$ whose $n$-th image lies in $O$ would lead to a contradiction. In the sequel of this proof, we will construct such a path. For $0\leq j\leq n$, we set $s_j:=2Rc^j$ and $B_j:=B(z_{n-j},s_j)$. Recall that $B_n\subset O$.

\begin{figure}[H]
\centering
\includegraphics[width=\linewidth]{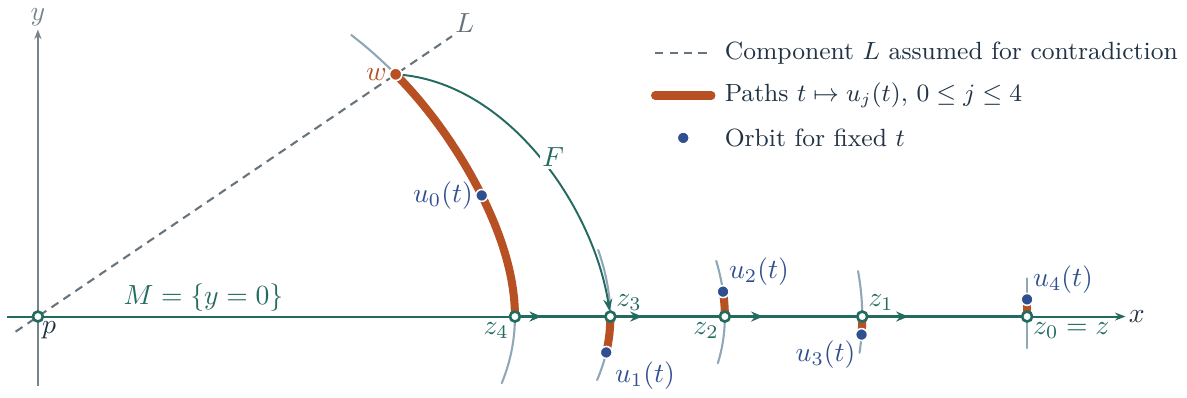}
\caption*{Illustration of the contradiction scenario}
\label{fig:core-smoothness-path}
\end{figure}

If the map $F$ was linear of the form $(x,y)\mapsto(Ax,0)$, then for a finite orbit of points, the initial $y$-coordinate and the final $x$-coordinate would determine the entire orbit because $A$ is invertible. For $F$ we have roughly the same description. More precesily, for each $t\in[0,1]$ we seek points $u_j=(x_j,y_j)\in B_j$, $0\leq j\leq n$, satisfying $F(u_j)=u_{j+1}$ for $0\leq j<n$, with $y_0=(1-t)y(z_n)+ty(w)$ and $x_n=x(z)$.
At $t=0$ the orbit of $z_n$ satisfies these conditions. At $t=1$ they hold for the orbit of $w$, since $w\in B_0$ and $F^j(w)=z_{n-j}$ for $1\leq j\leq n$. We will prove that for each $t$ there is exactly one such tuple and that it depends continuously on $t$. This will give us a path from $z_n$ to $w$ whose $n$-th image lies in $B_n\subset O$.

We define two functions $g$ and $h$ on $B(p,3R)$ by writing $F(x,y)=(Ax+g(x,y),h(x,y))$. In our coordinates, the equation $F(u_j)=u_{j+1}$ means that
\begin{equation}\label{eq:smooth-gh-formulas}
Ax_j+g(u_j)=x_{j+1}, \quad h(u_j)=y_{j+1}.
\end{equation} 
Fix $t\in[0,1]$. For any finite sequence of points $u=(u_0,\ldots,u_n)$ such that $u_j=(x_j,y_j)\in\overline B_j$, we define
$S_t(u):=((X_0,Y_0),\ldots,(X_n,Y_n))$ by
\begin{equation}\label{eq:core-path-coordinates}
\begin{aligned}
 Y_0&=(1-t)y(z_n)+ty(w), & X_n&=x(z),\\
 X_j&=A^{-1}\bigl(x_{j+1}-g(u_j)\bigr), &
 Y_{j+1}&=h(u_j), \quad 0\leq j<n.
\end{aligned}
\end{equation}
Thus $S_t(u)=u$ if and only if $F(u_j)=u_{j+1}$ for $0\leq j<n$,
$y_0=(1-t)y(z_n)+ty(w)$, and $x_n=x(z)$.

Let $a<c$ be a Lipschitz constant for $(g,h)$ on $B(p,3R)$. Note that $B(p,3R)$ contains each $\overline B_j$ because $z_{n-j}\in B(p,R)$ and $s_j\leq 2R$. Fix $t\in[0,1]$ and $u,v\in\prod_{j=0}^n\overline B_j$. Denote the $j$-th entries of $S_t(u)$ and $S_t(v)$ by $(X_j,Y_j)$ and $(X'_j,Y'_j)$, respectively. By \eqref{eq:core-path-coordinates}, we have $Y_0=Y'_0=(1-t)y(z_n)+ty(w)$ and $X_n=X'_n=x(z)$. Let $\lambda:=\max\{\lVert A^{-1}\rVert(c+a),a/c\}$. We have $\lambda<1$ and $2c\lVert A^{-1}\rVert<1$.
For $0\leq j<n$, applying the Lipschitz bounds for $g$ and $h$ to the differences $g(u_j)-g(v_j)$ and $h(u_j)-h(v_j)$, and using $\lVert u_i-v_i\rVert\leq s_i d(u,v)$, we obtain from \eqref{eq:core-path-coordinates}
\begin{align}\label{eq:core-path-bounds}
 \lVert X_j-X'_j\rVert
 &\leq \lVert A^{-1}\rVert(s_{j+1}+a s_j)d(u,v)
 \leq \lambda s_j d(u,v),\\
 \lVert Y_{j+1}-Y'_{j+1}\rVert
 &\leq a s_j d(u,v)\leq \lambda s_{j+1}d(u,v). \notag
\end{align}

We now verify that $(X_j,Y_j)\in B_j$ for every $j$. We equip $\prod_{j=0}^n\overline B_j$ with the metric $d(u,v):=\max_{0\leq j\leq n}s_j^{-1}\lVert u_j-v_j\rVert$, where $u=(u_0,\ldots,u_n)$ and $v=(v_0,\ldots,v_n)$. Take $v=(z_n,\ldots,z_0)$, so that $d(u,v)\leq 1$. We have $X'_j=x(z_{n-j})$ and $Y'_{j+1}=y(z_{n-j-1})$ by \eqref{eq:core-path-coordinates}. It follows from \eqref{eq:core-path-bounds} that $\lVert X_j-x(z_{n-j})\rVert<s_j$ for $0\leq j<n$ and $\lVert Y_j-y(z_{n-j})\rVert<s_j$ for $1\leq j\leq n$. At $j=n$ we have $X_n=x(z_0)$, and at $j=0$ we have
$\lVert Y_0-y(z_n)\rVert=t\lVert y(w)-y(z_n)\rVert<2R=s_0$
because $w,z_n\in B(p,R)$. By the choice of the maximum norm, we obtain $(X_j,Y_j)\in B(z_{n-j},s_j)=B_j$ for $0\leq j\leq n$. 

The above arguments showed that $S_t$ maps $\prod_{j=0}^n\overline B_j$ into itself. For any $u,v$ in $\prod_{j=0}^n\overline B_j$, \eqref{eq:core-path-bounds} and the equalities $X_n=X'_n$ and $Y_0=Y'_0$ imply 
\begin{equation}\label{eq:smooth-St-contraction}
d(S_t(u),S_t(v))\leq\lambda d(u,v).
\end{equation}
Hence $S_t$ has a unique fixed point $u(t)=(u_0(t),\ldots,u_n(t))$, which depends continuously on $t$.

The orbit sequences $(z_n,\ldots,z_0)$ and $(w,z_{n-1},\ldots,z_0)$
both end at $z$. By \eqref{eq:core-path-coordinates}, the first is
fixed by $S_0$ and the second by $S_1$. By uniqueness, $u_0(0)=z_n$ and $u_0(1)=w$. For every $t\in[0,1]$, we have $F^n(u_0(t))=u_n(t)\in B_n\subset O$. The path $t\mapsto u_0(t)$ connects $z_n$ and $w$, and force them to lie in the same component of $(F^n)^{-1}(O)$, contrary to the choice of $O$. This contradiction finishes our proof. 
\end{proof}

\section{Convergence to superattracting cycles}\label{sec:convergence}
We use the notations from previous sections. As before $f$ is a PCF endomorphism of $\bP^N$, and $\Omega$ is a Fatou component. Recall that $\cH_\Omega$ is defined in Subsection~\ref{subsec:limits-iterates}. 
\begin{prop}\label{prop:superattracting-limit}
For every Fatou component $\Omega$, the family $\cH_\Omega$ contains a constant map whose value is a superattracting periodic point.
\end{prop}

\begin{proof}
Let $C$ be the set of constant values in $\cH_\Omega$.
By Proposition~\ref{prop:log-rank} and
Theorem~\ref{thm:periodic-constant}, $C$ is finite and nonempty.
If $f^{n_j}\vert_\Omega\to q\in C$, a limit $H$ of
$f^{n_j-1}\vert_\Omega$ satisfies $f\circ H=q$.
Since $f$ is finite and $\Omega$ is connected, $H$ is constant.
Thus $f(C)=C$, and every point of $C$ is periodic.

Suppose for contradiction that no point of $C$ is superattracting. For each $q\in C$, let $W_q$ be its algebraic core. By
Theorems~\ref{thm:algebraic-core} and~\ref{thm:core-smoothness},
$W_q$ is smooth at $q$, and some iterate $F$ of $f$ preserves $W_q$
and induces an automorphism of the germ $(W_q,q)$. Recall that the family $\cH_\Omega(F)$ contains a constant map by Theorem~\ref{thm:periodic-constant}. For any two such constant values $p,q$, Propositions~\ref{prop:capture} and ~\ref{prop:capture2} say that $W_p=W_q$ and that every map in $\cH_\Omega(F)$ has image there. We will denote this common algebraic core by $W$. 

Suppose that $\cH_\Omega(F)$ contains a nonconstant map, and let $H$
be one of least positive generic rank. By Lemma~\ref{lem:rank-drop}, every accumulation point of $(F^n\circ H)$ is a constant map. Such a constant map has value in $C$ whose points are all fixed by $F$. Therefore $(F^n\circ H)$ converges locally uniformly to some $q\in C$. Hence $\cH_\Omega(F)$ always contains a constant limit map.

As $\cH_\Omega(F)$ contains the constant map with value $q$, by
Proposition~\ref{prop:capture} every $F^n\circ H$ has image in $W_q$. Let $V$ be a sufficiently small neighborhood of $q$ in $W_q$ on which $F$ is injective. By Lemma~\ref{lem:fixed-inverses}, there are a neighborhood $U\subset V$ of $q$ in $W_q$ and holomorphic maps
$b_n:U\to V$ for $n\geq 0$, equicontinuous at $q$, such that
\begin{equation}\label{eq:limit-core-inverses}
b_0=\Id_U,\qquad b_n(q)=q,\qquad F\circ b_{n+1}=b_n.
\end{equation}

Since $F^n\circ H\to q$, for any nonempty connected open subset $\Omega_0\Subset\Omega$ there is $n_0\geq 0$ such that $H_n:=F^{n_0+n}\circ H$ maps $\Omega_0$ into $U$ for every $n\geq 0$. By injectivity of $F$ on $V$ and \eqref{eq:limit-core-inverses}, we obtain by induction $b_j\circ H_n=H_{n-j}$ on $\Omega_0$ for $0\leq j\leq n$. Remark that $b_n\circ H_n\to q$ because of the equicontinuity of $(b_n)$ at $q=b_n(q)$. Thus $H_0$ is the constant map $q$. Since $F^{n_0}\circ H=H_0$ and $F^{n_0}$ is finite, $H$ is constant on $\Omega_0$, and hence on $\Omega$. Every map in $\cH_\Omega(F)$ is therefore constant.

We claim now that $F^n\vert_\Omega\to p$ locally uniformly for some $p\in C$. Indeed, the accumulation points of $(F^n\vert_\Omega)$ are in $C$ and each point of $C$ is fixed by $F$. The claim then follows from normality of $(F^n)$.

As $p$ is not superattracting by our absurd hypothesis, Proposition~\ref{prop:replling-direction}, applied to $F$, implies that $d_pF$ has an eigenvalue of modulus greater than one. 
By the local center-stable manifold theorem for non locally invertible maps \cite[Theorem~III.7(2) and Exercise~III.2]{Shub1987},
\cite[Theorem~1.2.1 and \S3.3]{Chaperon2004}, there are a neighborhood $O$ of $p$ and a real $C^1$ submanifold $S$ of positive codimension such that $z\in S$ whenever $F^n(z)\in O$ for every $n\geq 0$. Let $\Omega_1\Subset\Omega$ be a nonempty connected relatively compact subset. Since $F^n\vert_\Omega\to p$ locally uniformly, there is $n_1\geq 0$ such that $F^{n_1+n}(\Omega_1)\subset O$ for every $n\geq0$. Thus $F^{n_1}(\Omega_1)\subset S$. Since $F^{n_1}$ is open and $S$ has positive codimension, we obtain a contradiction.
\end{proof}

\begin{proof}[Proof of Theorem~\ref{mainthm:Fatou}]
Let $\Omega$ be a Fatou component.
By Proposition~\ref{prop:superattracting-limit}, there is a sequence $m_j\to\infty$ such that $f^{m_j}\vert_\Omega\to p$ for a superattracting periodic point $p$, so $\Omega$ is contained in the basin of the cycle of $p$. Since immediate basins are Fatou components and $f$ maps Fatou components onto Fatou components, some forward image of $\Omega$ equals the immediate basin of $p$.
\end{proof}

\section{Geometry of algebraic cores of PCF maps}\label{sec:geometry}

In this section we study the global geometry of algebraic cores. We prove that every algebraic core is rationally connected and that a PCF endomorphism has only finitely many distinct algebraic cores.

We give first two examples of explicit PCF endomorphisms for illustration. On the one hand these examples provide tests of the generality of Theorem~\ref{mainthm:Fatou}, because their PCF divisors are too singular to satisfy the "weak transversality condition" required by \cite[Thm.~C]{Astorg2020}. This also shows that these two examples are not covered by the construction of Koch~\cite{Koch2013}. On the other hand the example in \S\ref{subsec:non-rational-core} shows that an algebraic core is not necessarily a rational variety, which suggests that Theorem~\ref{thm:rc-core} is nearly optimal. Note that Theorem~\ref{thm:rc-core} holds beyond the algebraic core setting or even the PCF setting. It is interesting to look for finer restrictions on the geometry of the algebraic cores, which we think would be a crucial step towards a classification of PCF maps.


\subsection{Example: the Steiner surface}\label{subsec:Steiner}
We present a degree-seven PCF endomorphism of $\bP^3$ for which the Steiner surface is the algebraic core of a fixed point. 

We will use homogeneous coordinates $[z_0:z_1:z_2:z_3]$ on $\bP^3$. Let $L$ be the hyperplane $z_0+z_1+z_2+z_3=0$. Let $q:\bP^3\to\bP^3$ be the endomorphism $[z_0:z_1:z_2:z_3]\mapsto[z_0^2:z_1^2:z_2^2:z_3^2]$.
The image $\mathcal R:=q(L)$ is the Steiner surface, an irreducible quartic. Its equation is $D(z)=0$, where $Q(z)=(\sum_j z_j)^2-2\sum_j z_j^2$ and
$D(z)=Q(z)^2-64z_0z_1z_2z_3$ (see \cite[Example~3.5]{DeyGorlachKaihnsa2020}).

Consider the endomorphism $R=[R_0:R_1:R_2:R_3]$ of $\bP^3$ defined by
\begin{equation}\label{geom:roman-map}
R_i(z)=z_i^3 A_i(z)^2,\quad  A_i(z)=120z_i^2-84z_i\sum_j z_j+35Q(z)
,\quad 0\leq i\leq 3.
\end{equation}

\begin{prop}\label{geom:roman-postcritical}
The map $R$ is a postcritically finite endomorphism of $\bP^3$ of degree $7$. It fixes $p=[1:1:1:9]$, and its algebraic core at $p$ is $\mathcal R$.
\end{prop}

\begin{proof}
Modulo $7$, each $R_i(z)$ reduces to $z_i^7$. The resultant of these four polynomials is therefore $1$ modulo $7$, so they have no common zero. Hence $R$ defines an endomorphism of $\bP^3$.

We deduce from \eqref{geom:roman-map} that each coordinate plane $z_i=0$ is $R$-invariant and is the image of the quadric $A_i=0$. 
A direct computation shows that the critical locus of $R$ is the union of the Steiner quartic surface $\mathcal R$, the four coordinate planes $z_i=0$, and the four quadrics $A_i=0$, while its PCF set is the union of the quartic and the four coordinate planes.

Now let us check that $R(\mathcal R)=\mathcal R$. This will imply that $R$ is PCF. As $\mathcal R$ is irreducible, it suffices to prove the inclusion $R(\mathcal R)\subseteq\mathcal R$. Take $[z]=[z_0:z_1:z_2:z_3]\in L$ and write $z^2=(z_0^2,z_1^2,z_2^2,z_3^2)$. By \eqref{geom:roman-map}, the coordinates of $R(q([z]))$ are the squares of $z_i^3A_i(z^2)$ for $0\le i\le 3$. Hence $R(q([z]))$ belongs to $q(L)$ if the four expressions $z_i^3A_i(z^2)$ sum to zero. Indeed,
\begin{align*}
& \sum_{i=0}^3 z_i^3 A_i(z^2)
=120\sum_i z_i^7
       -84\Bigl(\sum_i z_i^2\Bigr)\Bigl(\sum_i z_i^5\Bigr)
+35\Bigl[\Bigl(\sum_i z_i^2\Bigr)^2-2\sum_i z_i^4\Bigr]\sum_i z_i^3
\\
&\quad =\Bigl(\sum_i z_i\Bigr)^3
       \Bigl[\Bigl(\sum_i z_i\Bigr)^4
       -7\Bigl(\sum_i z_i\Bigr)^2\sum_{i<j}z_i z_j
       +35\Bigl(\sum_i z_i\Bigr)\sum_{i<j<k}z_i z_j z_k\\
&\quad -140z_0z_1z_2z_3\Bigr]=0 \qquad \text{because }\sum_i z_i=0.
\end{align*}

A direct computation shows that $p=q([1:1:1:-3])$ and $R(p)=p$. In the affine chart $z_3=9$, the defining equation $D$ of $\mathcal R$ of the quartic satisfies $dD_p=-1536(dz_0+dz_1+dz_2)$. Hence $\mathcal R$ is smooth at $p$, with
$T_p\mathcal R=\ker dD_p$. 
An explicit matrix computation shows that $dR_p$ acts by $35/9$ on $T_p\mathcal R$ and vanishes on the complementary one dimensional direction. We have thus $W_p(R)=\mathcal R$ by definition of the algebraic core.
\end{proof}

The singular locus of $\mathcal R$ is described in \cite[Example~6.4]{HartshornePolini2015}. It consists of the three double lines $\{z_i=z_j,\ z_k=z_l\}, \{i,j,k,l\}=\{0,1,2,3\}$. They meet at the ordinary triple point $r=[1:1:1:1]$. Each of these lines contains two pinch points of $\mathcal R$. Away from these points and $r$, two smooth branches of $\mathcal R$ meet transversely along each line. Note that $\mathcal R$ is the only component of the PCF divisor $P_R$ through $r$. Thus Astorg's weak transverse condition~\cite[\S~2]{Astorg2020} fails at $r$. The picture below shows the PCF divisor of $R$ which is the union of $\mathcal R$ and the four coordinate plans. The square map $q$ sends $L$ to $\mathcal R$, and is actually the normalization of $\mathcal R$. 

\begin{figure}[H]
 \centering
 \includegraphics[width=\textwidth]{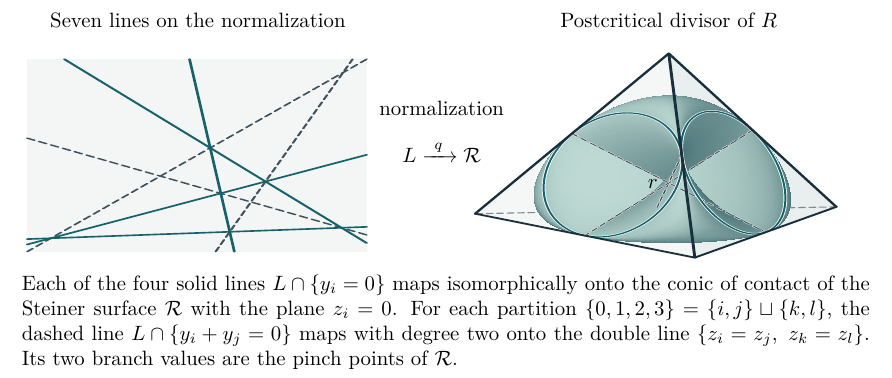}
\end{figure}

The quartic $\mathcal R$ contains also fixed points with smaller algebraic cores. Consider the line $\Delta=\{z_0=z_1,\ z_2=z_3\}$ in its singular locus. On this line $R$ fixes the point $t=[3:3:1:1]$. Near $t$ the quartic has two smooth branches which meet along $\Delta$. The map $R$ interchanges the branches and preserves their intersection. The image of $dR_t$ is $T_t\Delta$, on which $dR_t$ acts by $35/3$. Thus the algebraic core of $t$ is $\Delta$. The line also contains the triple point $r=[1:1:1:1]$ which is a superattracting fixed point.
\begin{figure}[H]
 \centering
 \includegraphics[width=\textwidth]{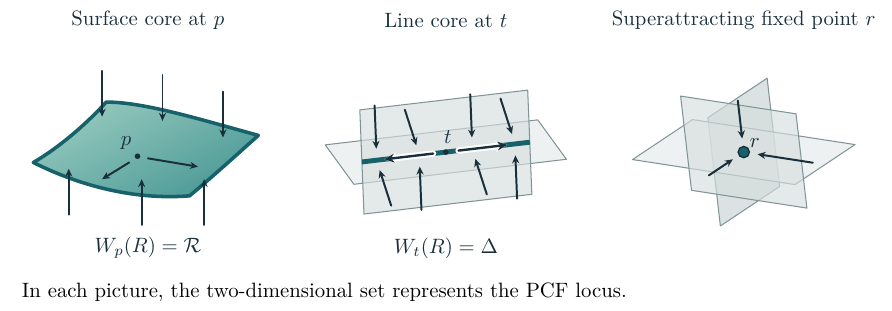}
\end{figure}

The superattracting basin of $r$ can be described by a theorem of Buff--Epstein--Koch as follows:
\begin{prop}\label{prop:Steiner-Fatou}
The action of $R$ on the immediate basin of $r$ is biholomorphically conjugate to the cubic power map $(u_1,u_2,u_3)\mapsto(u_1^3,u_2^3,u_3^3)$ on the three dimensional polydisk $\bD^3$.
\end{prop}

\begin{proof}
Since $\mathcal R=q(L)$, its inverse image under $q$ is the union of the planes $\sum_i\epsilon_i z_i=0$ with
$\epsilon_i=\pm1$. In the chart $z_3=1$, the local inverse of $q$ fixing $r$ is given by the square roots $\sqrt{z_i}$ with value $1$ at $r$.
The three planes through $r$ have two plus signs and two minus signs, so their images have local equations $u_j=0$, where
\[
u_j=\sqrt{z_0}+\sqrt{z_j}-\sqrt{z_k}-\sqrt{z_l}, \quad \{j,k,l\}=\{1,2,3\}.
\] 
The differentials $du_1,du_2,du_3$ are independent at $r$, so $u_1,u_2,u_3$ are local coordinates in which $\mathcal R$ has equation $u_1u_2u_3=0$. The quotient $(u_j\circ R)/u_j^3$
is holomorphic and non vanishing at $r$. Each plane $u_j=0$ lies in the critical locus of $R$, and $R$ maps it onto itself as a germ at $r$. The radial vector field $\sum_j u_j\,\partial/\partial u_j$ is tangent to all three planes. Note that $r$ is an adapted superattracting fixed point with one summand in the sense of \cite[Definition~1.1]{BuffEpsteinKoch2012}. Hence, by applying \cite[Theorem~0.3]{BuffEpsteinKoch2012} to the above vector field, we deduce that $R$ is locally conjugate to  $(u_1,u_2,u_3)\mapsto(u_1^3,u_2^3,u_3^3)$. Then by \cite[Theorem~0.4]{BuffEpsteinKoch2012} the conjugacy extends to a biholomorphism from the immediate basin of $r$ onto $\bD^3$.
\end{proof}

\begin{rem}\label{rem:Steiner-other-Fatou}
The six pinch points of $\mathcal R$ are also superattracting fixed points of $R$. On each of their immediate basins, the critical locus of $R$ has a singular irreducible branch at the fixed point, namely the germ of $\mathcal R$. As every branch of the critical locus of a proper holomorphic self-map of a polydisk is a coordinate hyperplane by \cite[Theorem~II.3]{RudinStout1969}, none of these basins is biholomorphic to the tridisk. 
\end{rem}

\subsection{Example: non stably rational variety}\label{subsec:non-rational-core}

We present a PCF endomorphism $F:\bP^{55}\to\bP^{55}$ of degree $16$ with a fixed point $p$ whose algebraic core $W$ is a singular rationally connected subvariety of dimension $25$ which is not stably rational, i.e. $W\times\bP^m$ is not rational for every $m$. Many assertions below require lenghthy computational verifications. We give an overview here and refer to our notes~\cite{PCF_II1/2} for details.

\subsubsection{The variety $W$}
Let $\mathcal{T}$ be the set of ten faces of the triangulation $\Delta$ of $\mathbb{RP}^2$ shown in Figure~\ref{fig:nonrational-triangulation}. We consider a system of homogeneous coordinates on $\bP^{55}$ as follows: each vertex $l$ corresponds to one \emph{vertex coordinates} $t_l$, to each face $\tau=\{i<j<k\}\in\mathcal{T}$ we associate two \emph{face coordinates} $x_\tau,y_\tau$ and three \emph{root coordinates} $a_\tau,b_\tau,c_\tau$. 
Later we will use the projection map:
\[
\pr:\bP^{55}\dashrightarrow\bP^{25}, \quad 
 [t_0:\ldots:t_5:(x_\tau,y_\tau,a_\tau,b_\tau,c_\tau)_\tau]
    \longmapsto[t_0:\ldots:t_5:(x_\tau,y_\tau)_\tau].
\]
For every face $\tau=\{i<j<k\}$ we consider equations
\begin{equation}\label{geom:nonrational-roots}
 a_\tau^2=2t_i x_\tau,\quad b_\tau^2=t_j(x_\tau+y_\tau),\quad
 c_\tau^2=t_k(x_\tau-y_\tau).
\end{equation}
Let $W$ be the intersection locus of all these equations, which is an integral subvariety of dimension $25$. The restriction
$\rho=\pr\vert_W:W\to\bP^{25}$ is finite.

\begin{figure}[H]
\centering
\includegraphics[width=0.2\textwidth]{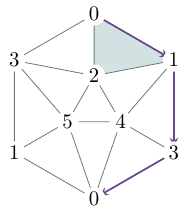}
\caption{$\mathbb{RP}^2$ as the antipodal quotient of the icosahedron}
\label{fig:nonrational-triangulation}
\end{figure}

\subsubsection{Dynamics on $W$ and $\bP^{55}$}
We define three functions on $\bC^2$:
\[
S(x,y):=x^2+y^2,\quad D(x,y):=x^2-y^2, \quad M(x,y):=2xy.
\]
For each face $\tau$ we denote by $S_\tau,D_\tau,M_\tau$ the evaluations of $S,D,M$ at $(x_\tau, y_\tau)$. Define two endmorphisms $h:\bC^2\to\bC^2$ and $H:\bP^{25}\to\bP^{25}$ by
\begin{align*}
& h(x,y):=\bigl(\tfrac12 S^2,\ D^2-\tfrac12 S^2\bigr),\\
& H\bigl([t_0:\ldots:t_5:(x_\tau,y_\tau)_\tau]\bigr)
:=\bigl[t_0^4:\ldots:t_5^4:(h(x_\tau,y_\tau))_\tau\bigr].
\end{align*}
The morphism $\pi:\bP^{25}\to \bP^{55}$
\begin{align*}
[t_0:\cdots:t_5:(x_\tau,y_\tau)_\tau]  &\longmapsto
[t_0^4:\cdots:t_5^4: (h(x_\tau,y_\tau),t_i^2S_\tau,t_j^2D_\tau,t_k^2M_\tau)_\tau]
\end{align*}
is surjective onto $W$. The projection $\rho:W\to\bP^{25}$ that
forgets the root coordinates satisfies $\rho\circ\pi=H$.

For $m\in\{0,1\}$, define the endomorphism $\Phi_m:\bP^{55}\to\bP^{55}$ by
\begin{align*}
&\Phi_m\bigl([t_0:\cdots:t_5:(x_\tau,y_\tau,a_\tau,b_\tau,c_\tau)_\tau]\bigr):=[t_0^4:\cdots:t_5^4:(X_\tau,Y_\tau,A_\tau,B_\tau,C_\tau)_\tau], 
 \\
& X_\tau:= \tfrac12 S_\tau^2, \quad Y_\tau:=D_\tau^2-\tfrac12 S_\tau^2, 
\\
& A_\tau:= t_i^2S_\tau-\bigl(a_\tau^2-t_i(x_\tau-y_\tau)\bigr)^2, \quad C_\tau:=t_k^2S_\tau-\bigl(c_\tau^2-t_k(x_\tau-y_\tau)\bigr)^2,
\\
& B_\tau:=\begin{cases}
     \bigl(b_\tau^2-t_j(x_\tau+y_\tau)\bigr)^2-t_j^2S_\tau, & m=1,\\
     2\bigl(b_\tau^2-t_jx_\tau\bigr)^2-t_j^2S_\tau, & m=0.
   \end{cases}
\end{align*}
where $\tau=\{i<j<k\}$ ranges over all faces. The two endomorphisms $\Phi_m$ are \emph{skew products} over $H$. The composition
$F:=\Phi_0\circ\Phi_1$ has degree $16$ and satisfies
\[
 F(W)=W,\qquad \rho\circ(F\vert_W)=H^2\circ\rho.
\]

In the following figure, each box lists three quadrics for a face $\tau$, while in each row, $\Phi_1$ sends the left quadric into the middle quadric, and $\Phi_0$ sends the middle quadric back into the original one. A label $2$ means that the pullback of the target equation vanishes to order two along the domain quadric. The vertical arrows forget the root coordinates. 
\begin{figure}[H]
\centering
\includegraphics[width=\textwidth]{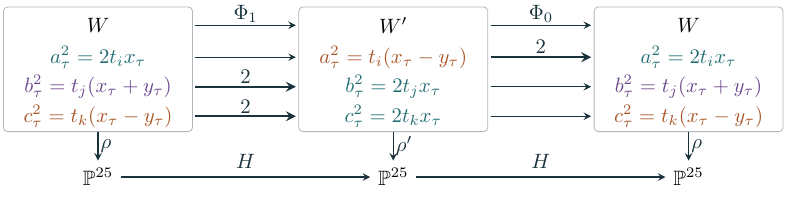}
\caption{}
\label{fig:nonrational-return}
\end{figure}

For each face $\tau$, both $\Phi_0$ and $\Phi_1$ act on the ratio $s=y_\tau/x_\tau$ by the PCF rational function
\[
 \varphi:\bP^1\to\bP^1,\qquad
 \varphi(s)=2\left(\frac{1-s^2}{1+s^2}\right)^2-1.
\]
Its critical values are $\infty,1,-1$, and fit into the cycle $\infty\mapsto1\mapsto-1\mapsto-1$. These are also the postcritical
points of $\varphi^2$, which are induced by $F$. In $\bP^{55}$ they correspond to the face hyperplanes $x_\tau=0$, $y_\tau=x_\tau$ and $y_\tau=-x_\tau$, respectively. 
This implies that $F$ is PCF, and its postcritical set $P_F$ is the union of the six vertex hyperplanes $t_0=0,\ldots,t_5=0$, the thirty
face hyperplanes above, and the sixty root quadrics in the left and
middle boxes of the preceding figure (six per face).

\subsubsection{The fixed point $p$ and its algebraic core}
The rational function $\varphi$ has a unique fixed point of the form $[1:s_*]$ where $s_*$ is in the real interval $(1/4,1/3)$. This fixed point is repelling with real multiplier $<-1$. Consider
\[
 x_*:=\left(\frac{2}{(1+s_*^2)^2}\right)^{1/3}>0,\quad
 y_*:=s_*x_*,\quad \bar p:=[1:\cdots:1:(x_*,y_*)_\tau]\in \bP^{55}.
\]
Since $h(1,s)=\tfrac12(1+s^2)^2(1,s)$ and $h$ is homogeneous of degree four, $(x_*,y_*)$ is fixed by $h$. Then $\bar p$ is a fixed point of $H$. At $\bar p$ the multipliers of $H$ are either $4$ or the multiplier of $\varphi$ at $[1:s_*]$. In particular $\bar p$ is repelling.

Consider $p:=F(\pi(\bar p))$. Since $\rho\circ\pi=H$, the point
$\pi(\bar p)$ belongs to $\rho^{-1}(\bar p)$. On $W$, the root coordinates of $\Phi_1$ are, by \eqref{geom:nonrational-roots},
\[
 (A_\tau,B_\tau,C_\tau)=(-t_i^2M_\tau,-t_j^2S_\tau,t_k^2S_\tau),
 \qquad \tau=\{i<j<k\}.
\]
Thus $F=\Phi_0\circ\Phi_1$ is constant on each fiber of $\rho$.
Since $\rho(p)=H^2(\bar p)=\bar p$, we have $F(p)=F(\pi(\bar p))=p$.

Since $0<y_*<x_*$, by \eqref{geom:nonrational-roots} we have $a_\tau b_\tau c_\tau\neq 0$ at $p$ for every $\tau$. Thus $\rho$ is
locally biholomorphic at $p$ and conjugates $F\vert_W$ to $H^2$, so
$dF_p$ is invertible on $T_pW$. Moreover $T_pW$ is exactly the sum of the generalized eigenspaces for nonzero eigenvalues. In other words $W$ is the algebraic core of $p$. 

\subsubsection{A nonvanishing unramified Brauer class}

Let $E=\bC(W)$ be the function field. The Brauer group $\operatorname{Br}(E)$ consists of classes of finite-dimensional central simple $E$-algebras, and is equipped with an addition induced by tensor product. Two such algebras $A$ and $B$ represent the same class precisely when $M_r(A)\simeq M_s(B)$ as $E$-algebras for some $r,s\geq 1$. We refer to \cite{ColliotTheleneSkorobogatovBook} for the theory of Brauer groups.

For $a,b\in E^*$, we denote by $(a,b)$ the Brauer class of the
quaternion $E$-algebra generated by $I,J$ with relations
$I^2=a$, $J^2=b$, and $IJ=-JI$. This class lies in the $2$-torsion
subgroup $\operatorname{Br}(E)[2]$ and vanishes precisely when the
conic $X^2=aY^2+bZ^2$ in $\bP^2_E$ has an $E$-point (see  \cite[\S1.1]{ColliotTheleneSkorobogatovBook}). In particular $(a,a)=(a,1)=0$ for any $a$. Note that there is also a bilinearity formula: $(a,b)+(a,c)=(a,bc)$.

We will work on $W$, and in the affine chart $t_0=1$ of $\bP^{55}$. We associate to each edge $ij$ the class $(t_i,t_j)$ in $\operatorname{Br}(E)[2]$. For each face $\tau=\{i<j<k\}$, equations~\eqref{geom:nonrational-roots} imply
\[
  \frac{a_\tau^2}{t_i}=2x_\tau =\frac{b_\tau^2}{t_j}+\frac{c_\tau^2}{t_k}.
\]
Thus $[a_\tau:b_\tau:c_\tau]$ is an $E$-point of the conic $X^2=(t_i/t_j)Y^2+(t_i/t_k)Z^2$, so $(t_i/t_j,t_i/t_k)=0$ in $\operatorname{Br}(E)[2]$. Using the formulas mentionned in the previous paragraph, we get
\begin{equation}\label{eq:face-relation}
 (t_i,t_j)+(t_j,t_k)+(t_k,t_i)=0\in \operatorname{Br}(E)[2].
\end{equation}

The pairs $ij$ for which $\{0,i,j\}$ is a face are $12,23,34,45,15$ (see Figure~\ref{fig:nonrational-triangulation}). Since every
edge $0i$ corresponds to a vanishing class $(1,t_i)=0$, the relation~\eqref{eq:face-relation} on each of these faces forces the class of the edges $12,23,34,45,15$ to vanish. Applying \eqref{eq:face-relation} again to the face $134$ we get $(t_1,t_3)=(t_1,t_4)$. Cyclically permuting $1,2,3,4,5$ we get 
\[
\beta := (t_1,t_3)=(t_1,t_4)=(t_2,t_4)=(t_2,t_5)=(t_3,t_5).
\]

We assign $1\in\bF_2$ to the edges $13,14,24,25,35$ and $0$ to all
other edges. Each face contains zero or two of the five edges, so
the sum of these values vanishes on the boundary of every face. However on the $1$-cycle $0\to1\to3\to0$ the sum is $1$. Thus this
cycle represents a nonzero class in $\coh_1(\mathbb{RP}^2,\bF_2)$. Note that the Brauer classes of its three edges sum to $(1,t_1)+(t_1,t_3)+(t_3,1)=\beta$. 
By \eqref{eq:face-relation}, the sum of the edge classes along a cycle
depends only on its homology class. This gives a homomorphism
\[
 \coh_1(\mathbb{RP}^2,\bF_2)\to\operatorname{Br}(E)[2],
 \qquad [0\to1\to3\to0]\mapsto\beta.
\]
A little bit more analysis would show the injectivity of this homomorphism. Hence $\beta\neq 0$.

A class $(a,b)$ over $E$ is \emph{unramified over $\bC$}
if, for every discrete valuation $v$ of $E$ trivial on $\bC^*$, the unit
\[
 (-1)^{v(a)v(b)}\frac{b^{v(a)}}{a^{v(b)}}
\]
reduces to a square in the residue field of $v$. 
Fix such a valuation $v$. If all $v(t_i)$ are even, then $\beta=(t_1,t_3)$ is unramified. Otherwise, after a cyclic relabeling of $1,\ldots,5$, we may assume that $v(t_2)$ is odd. Since
$(t_1,t_2)=(t_2,t_3)=(t_2,t_2)=0$, we have
\[
 \beta=(t_1t_2^\epsilon,t_3t_2^\delta)
 \qquad\text{for all }\epsilon,\delta\in\{0,1\}.
\]
Choose $\epsilon$ and $\delta$ to be the parities of $v(t_1)$ and
$v(t_3)$, respectively. Then $v(t_1)+\epsilon v(t_2)$ and
$v(t_3)+\delta v(t_2)$ are both even. Thus $\beta$ is again unramified.
As the unramified Brauer group of a stably rational complex variety is zero (see \cite{ColliotTheleneOjanguren} or \cite[Cor.~6.26]{ColliotTheleneSkorobogatovBook}), the non vanishing of $\beta$ implies that $W$ is not stably rational.

\subsection{Rational connectedness}\label{geom:rc}

The following theorem is a counterpart for endomorphisms of a theorem of Bogomolov--McQuillan for foliations (see \cite{KebekusSolaCondeToma2007}). In dimension two it was found by Bonifant--Dabija~\cite{BonifantDabija2002}.

\begin{thm}\label{thm:rc-core}
Let $F:\bP^N\to\bP^N$ be an endomorphism of algebraic degree
greater than one, preserving an irreducible subvariety $W$ of
dimension $r>0$. If $\operatorname{rank}(dF)=r$ at general points of $W$, then $W$ is rationally connected. In particular, this holds for every algebraic core of a PCF endomorphism.
\end{thm}
\begin{proof}
The idea is to transfer the positivity of $T_{\bP^N}$ to the tangent bundle on a resolution. If $W$ is smooth, then the rank
hypothesis makes $dF$ factor through a generically surjective
map $T_{\bP^N}|_W\to(F|_W)^*T_W$. Positivity of $T_{\bP^N}$ then forces the positivity of $T_W$, and the classical deformation theory of rational curves will establish the rational connectedness. For singular $W$, we will construct the corresponding map on its normalization after passing to a suitable iterate.

Let $\nu:X\to W\hookrightarrow\bP^N$ be the normalization, and let $g:X\to X$ be the induced morphism satisfing $\nu g=F\nu$.
The morphism $g$ is finite and $L:=\nu^*\cO_{\bP^N}(1)$ is ample. 
By assumption we have $\operatorname{rank}(dF^n)\leq r$ everywhere on $W$, for every $n\geq1$. 

To simplify the notations, for a morphism $h:\,X \to Y$ we shall denote by $\delta_h:\, h^*\Omega^1_Y \to \Omega^1_X$ the induced map on differentials. We use $\cM$ to denote the image sheaf $\operatorname{im}(\delta_\nu:\nu^*\Omega^1_{\bP^N}\to\Omega_X^1)$.
Then on $X_{\mathrm{reg}}$ the quotient $\Omega_X^1/\cM$ measures the failure of $\nu$ to be immersive. Note that $\delta_g(g^*\cM)\subset\cM$ by the chain rule.\\[.3mm]

\emph{Step 1: factorization and positivity of $T_X$.}
The main step is to construct a generically surjective morphism of sheaves
\begin{equation}\label{gen_surj_map}
\nu^*\mathcal O_{\mathbb P^N}(1)^{\oplus(N+1)}
\longrightarrow(g^{k+1})^*T_X
\end{equation}
over some big open subset (i.e. the complement has codimension $\ge 2$). To construct \eqref{gen_surj_map}, we need to prove the following \emph{factorization}: for some $k\geq1$, on a big open subset, one has
\begin{equation}\label{geom:compact-absorption}
\delta_{g^k}((g^k)^*\Omega_X^1)\subset\cM \subset \Omega_X^1.
\end{equation}
 
Before proving \eqref{geom:compact-absorption}, we first explain how to use this factorization to construct the desired map \eqref{gen_surj_map} in this step. The key point is that the induced map $\nu^*\delta_F:g^*\nu^*\Omega^1_{\bP^N}\to\nu^*\Omega^1_{\bP^N}$ actually factors through the surjection $g^*\delta_{\nu}:g^*\nu^*\Omega^1_{\bP^N}\to g^*\cM$, which gives us the map $g^*\cM \to \nu^*\Omega^1_{\bP^N}$. To see this, we consider the kernel $\cK:= \operatorname{ker}(g^*\delta_{\nu})$, and claim that $\nu^*\delta_F:g^*\nu^*\Omega^1_{\bP^N}\to\nu^*\Omega^1_{\bP^N}$ annihilates $\cK$, which gives us the desired factorization:
\[
\begin{tikzcd}
0 \ar[r] 
  & \cK \ar[r] \ar[rd, "0"] 
  & g^*\nu^*\Omega^1_{\bP^N} \ar[r, "g^*\delta_{\nu}"] \ar[d, "\nu^*\delta_F"] 
  & g^*\cM \ar[r] \ar[ld, dashed] 
  & 0 \\
  & & \nu^*\Omega^1_{\bP^N} & &
\end{tikzcd}
\]
Indeed, at a general point $x\in X$, the normalization $\nu$ is
immersive at $x$ and $g(x)$, the map $g$ is \'etale at $x$,
and $\operatorname{rank}(d_{\nu(x)}F)=r$. By differentiating
$\nu g=F\nu$ we get 
\[
d_{\nu(x)}F\bigl(d_x\nu(T_xX)\bigr) =d_{g(x)}\nu(T_{g(x)}X).
\]
The right-hand side has dimension $r$, so the rank hypothesis implies $\operatorname{im}(d_{\nu(x)}F) =d_{g(x)}\nu(T_{g(x)}X)$.
A covector in the kernel of $g^*\delta_{\nu}$ at $x$
annihilates the tangent space on the right. Its pullback
by $d_{\nu(x)}F$ is therefore zero. Thus
$\nu^*\delta_F$ annihilates $\cK$ generically. Since $\nu^*\Omega^1_{\bP^N}$ is locally free, the image $\nu^*\delta_F(\cK)$ is thus identically zero.

Once \eqref{geom:compact-absorption} is established, we can combine it with the natural map $g^*\cM \to \nu^*\Omega^1_{\bP^N}$ to get the following composite
\[
\alpha:(g^{k+1})^*\Omega_X^1\longrightarrow
g^*\cM\longrightarrow\nu^*\Omega^1_{\bP^N},\qquad
\delta_\nu\alpha=\delta_{g^{k+1}}.
\]
Note that $\alpha$ is generically injective. By taking dual of $\alpha$ and using the Euler sequence of $T_{\bP^N}$, we construct the required generically surjective map \eqref{gen_surj_map}.\\[.3mm]

\emph{Step 2: proving the sheaf inclusion.} We prove \eqref{geom:compact-absorption} in this step. If $\cM = \Omega^1_X$, then \eqref{geom:compact-absorption} trivially holds. Thus we will focus on the locus where $\cM \neq \Omega^1_X$. Since one only needs to establish \eqref{geom:compact-absorption} on some open subset, we are free to remove subschemes in $X$ of codimension $\geq 2$. That enables us to work throughout where the domain and target of some fixed iterate $g^k$ are smooth: the excluded sets, and their finitely many inverse images, have codimension at least two because $X$ is normal and $g$ is finite. 

Let $D\subset X$ be a prime divisor with generic point $\eta$.
We call $D$ \emph{nonimmersive} if
$\cM_\eta\ne(\Omega_X^1)_\eta$. There are only finitely many such
divisors, since each is a codimension-one component of the support
of the coherent torsion sheaf $\Omega_X^1/\cM$.

We need to prove \eqref{geom:compact-absorption} at $\eta$. Consider the local ring $R=\cO_{X,\eta}$, a discrete valuation ring. Let $t$ be a uniformizer, so that $R/(t)=\bC(D)$.
Take germs $a_1,\ldots,a_{r-1}\in\cO_{W,\nu(\eta)}$ whose residues
form a transcendence basis of $\bC(\nu(D))$ over $\bC$. For each $i$, let $b_i=\nu^*a_i\in R$ and let $\bar b_i\in\bC(D)$ be its residue. The finite morphism $D\to\nu(D)$ induces a finite separable extension
$\bC(\nu(D))\subset\bC(D)$. Hence the differentials
$d\bar b_1,\ldots,d\bar b_{r-1}$ form a $\bC(D)$-basis of
$(\Omega_D^1)_\eta$, the space of rational one-forms on $D$.

Normality of $X$ implies that $X$ is smooth at $\eta$, so $\Omega^1_{X,D}$ is free of rank $r$ over $R$. Consider the conormal
sequence 
\[
(t)/(t^2)\xrightarrow{\ d\ }
\Omega^1_{X,\eta}\otimes_R K
\longrightarrow(\Omega^1_{D})_\eta\longrightarrow0.
\]
The three terms have $K$-dimensions $1,r,r-1$, respectively, and the first arrow is injective. It sends the class of $t$ to the class of $dt$. The second arrow sends the class of $db_i$ to $d\bar b_i$, since $b_i$ restricts to $\bar b_i$ on $D$. Thus the classes of $db_1,\ldots,db_{r-1},dt$ form a basis of the middle term. By Nakayama's lemma these $r$ differentials generate the free $R$-module $(\Omega_X^1)_\eta$ of rank $r$, and hence form an $R$-basis.

Since $W\subset\bP^N$ is a closed subvariety, each $a_i\in\cO_{W,\nu(\eta)}$
is the restriction of a germ $\widetilde a_i\in\cO_{\bP^N,\nu(\eta)}$. Its pullback to $X$ is $b_i$, so the pullback of $d\widetilde a_i$ is $db_i$. By the definition of $\cM$, this proves that $db_i\in\cM_\eta$. 

Let $I=\{u\in R\mid u\,dt\in\cM_\eta\}$. In the basis just
constructed, an element $\sum_{i=1}^{r-1}u_i\,db_i+u\,dt$ belongs
to $\cM_\eta$ if and only if $u\in I$, since every $db_i$ belongs
to $\cM_\eta$. The ideal $I$ is nonzero because $\cM_\eta$ has
rank $r$. Thus $I=(t^c)$ for some integer $c\geq 1$, and
\[
\cM_\eta=
\left(\bigoplus_{i=1}^{r-1}R\,db_i\right)\oplus t^cR\,dt.
\]
In particular, $(\Omega_X^1)_\eta/\cM_\eta$ is generated by the
class of $dt$ and is isomorphic to $R/(t^c)$. Hence
$t^c(\Omega_X^1)_\eta\subset\cM_\eta$.

Now let $D$ be any nonimmersive prime divisor. Since $g$ is finite,
each $g^j(D)$ is a prime divisor. There are only finitely many
nonimmersive prime divisors, so some $E=g^j(D)$ is either immersive
or periodic and nonimmersive. At the generic point of $E$, the inclusion $\delta_{g^a}((g^a)^*\Omega_X^1)\subset\cM$ holds with $a=0$ if $E$ is immersive because $\cM=\Omega_X^1$ there. 

It remains to treat the case where $D$ is nonimmersive and periodic under $g$, which we assume in the sequel. Let $s$ be its period and set $H:=g^s$. Since $H$ fixes the generic point $\eta$ of $D$, it induces a local homomorphism $H^*:R\to R$. The ramification index of $H$ along $D$ is the positive integer $e=\operatorname{ord}_D(H^*t)$. Thus $H^*t=ut^e$ for some $u\in R^*$. Since $H^*$ sends units to units, for each $n\geq1$ we have $(H^n)^*t=v_nt^{e^n}$ with $v_n\in R^*$. 
If $e>1$, differentiation gives us
\[
d((H^n)^*t)\in t^{e^n-1}\Omega^1_{X,D}\subset\cM_D,
\qquad e^n-1\geq c.
\]
As cotangent pullback by $H^n$ preserves $\cM$ and
$\Omega^1_{X,D}=\cM_D+R\,dt$, the inclusion \eqref{geom:compact-absorption} holds at $\eta$ for $k=sn$ whenever $e^n-1\geq c$. 

We claim that the unramified case $e=1$ does not occur. Suppose for contradiction that $e=1$. Then $H^*t$ is a uniformizer of $R$ and $H$ is étale at $\eta$. 

Let $D^\circ$ be the dense open subset of $D$ where $X$ is smooth and $H$ is étale. Since $D$ is nonimmersive, $d\nu$ has rank less than $r$ on $D^\circ$. Since $H\vert_D$ is finite and surjective, each $H^j(D\setminus D^\circ)$, $j\geq0$, is a proper closed subset of $D$. Over $\bC$ these countably many subsets do not cover $D$. Let $x_0\in D(\bC)$ lie outside their union. Consider a sequence of points $x_n\in D$ such that $H(x_{n+1})=x_n$ for every $n\geq 0$. Since $H^n(x_n)=x_0$, every $x_n$ belongs to $D^\circ$.

For each $n$, let $w_n=\nu(x_n)$ and let $V_n$ be the image of the
analytic germ $(X,x_n)$ in $(W,w_n)$. Since $\nu$ is the normalization of $W$, $V_n$ is the irreducible analytic branch corresponding to $x_n$, and $(X,x_n)\to V_n$ is its normalization. If $V_n$ were smooth, its normalization would be an isomorphism and $d_{x_n}\nu$ would be injective. Thus $V_n$ is singular of dimension $r$. In other words its Zariski tangent space $T_{w_n}V_n$ has dimension greater than $r$. 

Let $B_n=\cO^{\mathrm{an}}_{X,x_n}$ and let $A_n$ be the image of the homomorphism $\cO^{\mathrm{an}}_{W,w_n}\to B_n$, $f\mapsto f\circ\nu$. 
Since $V_n$ is the image of $(X,x_n)$, the kernel consists of the
germs vanishing on $V_n$. Thus pullback induces an isomorphism
$\cO^{\mathrm{an}}_{V_n,w_n}\xrightarrow{\sim} A_n$. As $H$ is étale at $x_{n+1}$, the map $H^*:B_n\to B_{n+1}$ is an isomorphism. By the relation $F^s\circ\nu=\nu\circ H$, the map $F^s$ sends $V_{n+1}$ to $V_n$, and its pullback $A_n\to A_{n+1}$ is the
restriction of $H^*$. Identifying each $B_n$ with $B_0$ via $(H^n)^*$, we obtain an ascending chain $A_0\subseteq A_1\subseteq\cdots\subseteq B_0$. Each $A_n$ is then an $A_0$-submodule of $B_0$. Since $B_0$ is finite over the Noetherian ring $A_0$, for some $n$ the pullback $A_n\to A_{n+1}$ is an isomorphism. The differential $d_{w_{n+1}}F^s$ therefore restricts to an isomorphism from $T_{w_{n+1}}V_{n+1}$ onto $T_{w_n}V_n$. Its rank is at least $\dim T_{w_n}V_n>r$, contradicting the rank hypothesis on $W$. The proof by contradiction in the preceding paragraphes proves that the ramification index satisfies $e>1$, in which case the step goal was already achieved.
\\[.3mm]

\emph{Step 3: positivity on a curve.}
By Steps~1 and~2, the morphism \eqref{gen_surj_map} is defined on an
open subset $U\subset X$ whose complement has codimension at least two. We may also remove from $U$ the singular locus of $X$ and its inverse image under the fixed iterate $g^{k+1}$. 
For $r\geq2$, by taking the intersection of $X$ with $r-1$ general
hyperplanes in a fixed projective embedding, we can consider a smooth curve $C$ contained in $U$.  Let $C'=(g^{k+1})^{-1}(C)$. We can also assume that $C'$ is smooth and $C'\subset U$. 

Since \eqref{gen_surj_map} is generically surjective, the locus in $U$
where it fails to be surjective has dimension at most $r-1$.  
By restricting \eqref{gen_surj_map} to $C'$, we obtain a generically surjective morphism of sheaves
\[
(L\vert_{C'})^{\oplus(N+1)} \to (g^{k+1}\vert_{C'})^*T_X,
\]
where $L=\nu^*\mathcal O_{\mathbb P^N}(1)$ is ample. By \cite[Example~6.4.17]{Lazarsfeld2} we deduce that $(g^{k+1}\vert_{C'})^*T_X$ is ample.  Since $g^{k+1}\vert_{C'}:C'\to C$ is finite and surjective, $T_X\vert_C$ is ample. 
Applying a version of Bogomolov--McQuillan's Theorem~\cite[Corollary~5]{KebekusSolaCondeToma2007} to $T_X\vert_C$, we conclude that $X$ is rationally connected.

Finally, we recall that Theorem~\ref{thm:transverse} says that every algebraic core of a periodic point of a PCF endomorphism satisfies the hypotheses of Theorem~\ref{thm:rc-core}.
\end{proof}

\begin{cor}\label{geom:surface-rationality}
Every algebraic core of dimension at most two is rational.
In particular, every algebraic core in $\bP^N$ is rational
when $N\leq3$.
\end{cor}
\begin{proof}
This follows from the fact that projective rationally connected curves and surfaces are rational.
\end{proof}

\appendix
\section{Some lemmas in several complex variables}\label{several complex variables}

\begin{lem}\label{lem:finite-normality}
Let $\phi:X\to Y$ be a finite map of complex spaces, let
$Z$ be a complex manifold, and let $(f_j)$ be a family of holomorphic maps from $Z$ to $X$. Then $(f_j)$ is normal if $(\phi\circ f_j)$ is normal.
\end{lem}
\begin{proof}
Assume that $(\phi\circ f_j)$ form a normal family. Up to taking a subsequence we may assume that $\phi\circ f_j$ converges locally uniformly to a holomorphic map $g:Z\to Y$.

Fix $z\in Z$. We can find pairwise disjoint neighborhoods $V_i\Subset W_i$ of the points of $\phi^{-1}(g(z))$, with each $W_i$ admitting a
holomorphic embedding into a complex Euclidean space. By properness,
$\phi^{-1}(W)\subset\bigcup_i V_i$ for some neighborhood $W$ of
$g(z)$. Thus, on a sufficiently small connected neighborhood $D$ of
$z$, all sufficiently large $f_j$ take values in $\bigcup_i V_i$.
Up to passing to a subsequence we have that $f_j(D)\subset V_i$ for one fixed $i$. By Montel's theorem we obtain a locally uniform holomorphic limit. We conclude by a diagonal extraction over a countable cover of $Z$.
\end{proof}

\begin{lem}\label{lem:restriction-degree}
Let $f:\bP^N\to\bP^N$ be an endomorphism of algebraic degree
$d_f\geq 1$, where $N\geq 1$. Let $X$ be an irreducible complex
projective variety of dimension $m\geq 1$, and let $\phi:X\to\bP^N$ be a holomorphic map which is finite onto its image. Suppose that a holomorphic map $h:X\to X$ satisfies $\phi\circ h=f^r\circ \phi$ for an integer $r\geq 1$. Let $L=\phi^*\cO_{\bP^N}(1)$ be the pullback of the hyperplane line bundle.
Then $L$ is ample, $h$ is finite and surjective, and \begin{equation}\label{eq:restriction-degree}
h^*L\simeq L^{\otimes d_f^r},\qquad \deg h=d_f^{rm}.
\end{equation}
\end{lem}
\begin{proof}
The map $h$ is finite. Its image is a closed subvariety of dimension $m$, so $h(X)=X$ because $X$ is irreducible. We have
\[
h^*L \simeq(\phi\circ h)^*\cO_{\bP^N}(1)\simeq
\phi^*(f^r)^*\cO_{\bP^N}(1)\simeq L^{\otimes d_f^r}.
\]
Since $\phi:X\to \phi(X)$ is finite, the pullback $L$ is ample. In particular, $L^m$ is a positive scalar. By the projection formula we have
\[
(\deg h)L^m=(h^*L)^m=d_f^{rm}L^m.
\]
Dividing by $L^m$ we finish the proof.
\end{proof}

Let $B=\bD^k\times\bD^l$ have coordinates $(z_1,\ldots,z_k,t_1,\ldots,t_l)$. Let $D\subset B$ be the subset defined by $z_1\cdots z_k=0$, and let $T\subset$ be the subset defined by $z_1=\cdots=z_k=0$. Let $\rho:Z\to B$ be a finite surjective holomorphic map, where $Z$ is normal, every irreducible component of $Z$ dominates $B$, and $\rho$ is unramified over $B\setminus D$.

\begin{lem}\label{lem:local-power-coordinates}
The inverse image $\rho^{-1}(T)_{\rm red}$ is a finite disjoint union of complex spaces, each mapped isomorphically onto $T$ by $\rho$.
\end{lem}

\begin{proof}
Since $Z$ is normal, its irreducible components are disjoint. Each component maps surjectively onto $B$ and therefore intersects every fiber of $\rho$. It thus suffices to prove the assertion when $Z$ is irreducible. 

The restriction of $\rho$ to ${\rho^{-1}(B\setminus D)}$ is a
connected finite unramified covering of $B\setminus D$. Its corresponding subgroup of $\pi_1(B\setminus D)\simeq\bZ^k$ has finite index. Hence there is an integer $m\geq1$ such that this subgroup
contains $m\bZ^k$. 
Let $p:B\to B$ be the finite map defined by
\begin{equation}\label{eq:local-power-map}
p(z_1,\ldots,z_k,t_1,\ldots,t_l) =(z_1^m,\ldots,z_k^m,t_1,\ldots,t_l).
\end{equation}
The restriction $p\vert_{B\setminus D}$ is a covering and corresponds to the subgroup $m\bZ^k$. Therefore we have a surjective holomorphic map $h:B\setminus D\to\rho^{-1}(B\setminus D)$ satisfying $\rho\circ h=p\vert_{B\setminus D}$.

Let $b\in B$. Since $\rho$ is finite, there is a neighborhood $U$ of $b$ such that $\rho^{-1}(U)$ admits a closed holomorphic embedding into $U\times\bC^a$ over $U$, for some integer $a\geq 1$. Let $U'\Subset U$ be a neighborhood of $b$. By properness of $\rho$, the additional coordinate functions of this embedding are bounded on
$\rho^{-1}(\overline{U'})$. Their compositions with $h$ are therefore bounded holomorphic functions on $p^{-1}(U')\setminus D$. By the removable singularity theorem, they extend holomorphically to $p^{-1}(U')$. Since the embedded space is closed, the extended map still takes values in $\rho^{-1}(U)$. Thus we obtain a holomorphic extension
$h:B\to Z$ satisfying $\rho\circ h=p$.

For every compact subset $K\subset Z$, the set $h^{-1}(K)$ is a closed subset of the compact set $p^{-1}(\rho(K))$. Hence $h$ is proper and its image is closed in $Z$. Since this image contains the dense subset $\rho^{-1}(B\setminus D)$, we deduce that $h$
is surjective.

By \eqref{eq:local-power-map}, we have $p^{-1}(T)=T$ and $p\vert_T=\Id_T$. Since $\rho\circ h=p$ and $h$ is surjective, we get $h(T)=\rho^{-1}(T)$ and $\rho\circ h\vert_T=\Id_T$.
For every $x\in\rho^{-1}(T)$, there is $y\in T$ such that $h(y)=x$. We then have $\rho(x)=y$, and hence $h(\rho(x))=x$. Thus $h\vert_T$ and the restriction of $\rho$ are inverse holomorphic maps between $T$ and $\rho^{-1}(T)$.
\end{proof}

Let $\xi\in Z$ satisfy $y:=\rho(\xi)\in T$. 
Let $\theta:\widehat{\cO}_{Z,\xi}\to\bC[[s]]$ be a continuous local homomorphism of $\bC$-algebras. Assume that $\theta(\widehat\rho^*z_i)\ne0$ for $1\leq i\leq k$,
and define $a_i=\operatorname{ord}_s \theta(\widehat\rho^*z_i)>0$.
Let $\Delta\subset\bC$ be a disk centered at $0$, and let $\gamma:\Delta\to B$ be a holomorphic map such that $\gamma(0)=y$, $\gamma(\Delta\setminus\{0\})\subset B\setminus D$, and $\operatorname{ord}_0(z_i\circ\gamma)=a_i$ for every $i$.

\begin{lem}\label{lem:arc-transfer}
The map $\gamma$ has a holomorphic lift through $\rho$
on all of $\Delta$, sending $0$ to $\xi$.
\end{lem}
\begin{proof}
We continue the proof of Lemma~\ref{lem:local-power-coordinates}. Let $Z'$ be the connected component containing $\xi$. Let $m$ and $h:B\to Z'$ be as in the proof of Lemma~\ref{lem:local-power-coordinates}. Let $G\subset(\mu_m)^k$ be the deck transformation group of $h$ off the coordinate divisors, where $\mu_m$ is the group of $m$-th roots of unity. Over $B\setminus D$, the fibers of $h$ are the $G$-orbits,
and $\rho$ has monodromy group $(\mu_m)^k/G$.

Let $c=(c_1,\ldots,c_k)\in\{0,\ldots,m-1\}^k$ be such that the character $(\zeta_i)\mapsto\prod_i\zeta_i^{c_i}$ is trivial on $G$. The monomial $\prod_i w_i^{c_i}$ descends off $\rho^{-1}(D)$ and extends to a holomorphic function $\varphi_c$ on $Z'$ because it is bounded. Since $\varphi_c^m=\rho^*(\prod_i z_i^{c_i})$, we obtain, by applying $\theta$ at $\xi$, that
\begin{equation}\label{eq:arc-transfer-orders}
m\,\operatorname{ord}_s\theta(\varphi_c)
=\sum_{i=1}^k c_i a_i.
\end{equation}
Let $\epsilon=\exp(2\pi\mathrm{i}/m)$. Then \eqref{eq:arc-transfer-orders} implies that $(\epsilon^{a_1},\ldots,\epsilon^{a_k})\in G$.

The monodromy along $\gamma$ is the class of this element in $(\mu_m)^k/G$ because $\operatorname{ord}_0(z_i\circ\gamma)=a_i$.
It is therefore trivial, so $\gamma$ lifts to $Z'$ on $\Delta\setminus\{0\}$. 
As in the proof of Lemma~\ref{lem:local-power-coordinates}, the lift extends across $0$ by finiteness and the removable singularity theorem. Its value there is $\xi$, the unique point of $Z'$ over $y$.
\end{proof}

\section{Julia sets in dimension two}\label{sec:Julia}

Let $f:\bP^2\to\bP^2$ be a PCF endomorphism of degree at least two.
A periodic point is \emph{super-saddle} if one multiplier is zero and the other has modulus greater than one, and \emph{repelling} if both have modulus greater than one. If $C$ is an irreducible component of $\Crit(f)$ of period $m$ under $f$, the union
$\Gamma=\bigcup_{j=0}^{m-1}f^j(C)$ is called a \emph{critical component cycle}. A super-saddle periodic point always lies in the critical locus but it is not obvious that is contained in a periodic irreducible component. Ji Zhuchao conjectured the following and proved it under some additional assumptions (see \cite[Theorem~1.2 and \S6]{Ji}).

\begin{prop}\label{prop:ji-super-saddle}
Every super-saddle cycle of $f$ is contained in a critical component cycle.
\end{prop}
\begin{proof}
Let $p$ be a super-saddle periodic point. Its algebraic core $W$ (see Theorem~\ref{thm:algebraic-core}) is an irreducible curve in $\bP^2$. By Theorem~\ref{thm:transverse}, some iterate
$f^a$ satisfies $f^a(W)=W$ and has differential of rank one at a general point of $W$. We have $W\subset\Crit(f^a)$. Recall $\Crit(f^a)=\bigcup_{j=0}^{a-1}f^{-j}(\Crit(f))$. Since $W$ is irreducible, there is some $0\leq j<a$ for which $f^j(W)\subset\Crit(f)$. Then $C:=f^j(W)$ is an irreducible curve and hence a component of $\Crit(f)$. The critical component cycle of $C$ contains $W$ and the orbit of $p$.
\end{proof}

Let $J_1$ be the Julia set of $f$, i.e.\ the complement of the Fatou set, and let $J_2$ be the support of the unique measure of
maximal entropy of $f$. We always have $J_2\subset J_1$. The basin of a critical component cycle $\Gamma$ consists by definition of the points $x$ for which $\operatorname{dist}(f^n(x),\Gamma)\to 0$.
We define a \emph{Fatou disk} as the image of a nonconstant holomorphic map $\varphi:\bD\to\bP^2$ for which the family of maps $(f^n\circ\varphi)_{n\in \bN}$ is normal. Recall that the non-wandering set $\Omega$ of $f$ consists of the points $x\in \bP^2$ such that every neighborhood $U$ of $x$ meets $f^n(U)$ for some $n\geq 1$.

\begin{cor}\label{cor:ji-julia}
Every point of $J_1\setminus J_2$ lies in the basin of a critical
component cycle, and every Fatou disk is disjoint from $J_2$.

The set $\Omega(f)$ is the closure of the periodic points, and $J_2$ is the closure of the repelling periodic points. Moreover, $J_2$ is open in $\Omega(f)$, and $\Omega(f)\setminus J_2$ is the union of all superattracting cycles and all sets of the form $\Gamma\cap J_1$ where $\Gamma$ a critical component cycle.
\end{cor}
\begin{proof}
In the terminology of \cite{Ji}, Proposition~\ref{prop:ji-super-saddle} says that $f$ has no sporadic super-saddle cycles. The assertions about basins and Fatou disks now follow from \cite[Theorem~1.1]{Ji}, and those about the non-wandering
set and repelling periodic points follow from \cite[Theorem~1.4]{Ji}.
\end{proof}

\section*{Acknowledgement}
We thank Zhuchao Ji for communicating to us the materials in Appendix B, and Xavier Buff for a question that stimulated us to write \S\ref{subsec:Steiner}. 

Our collaboration began during the participation of the second author to the conference "birational geometry and moduli spaces" organised by the first author in January 2026 at Xiamen University, though the work presented in this paper was mostly done later. We thank Xiamen University and other organisers for hospitality and organisation. 

R.-R. Sun is supported by an NSFC grant (No. 12501059) and a grant from Xiamen University (No. X2450215). S.-Y. Zhao acknowledges partial support from the French National Research Agency under the projects GAG (ANR-24-CE40-3526-01) and DynAtrois (ANR-24-CE40-1163).

\section*{AI usage}
The work presented in this paper originates from a section of our previous joint work~\cite{SunZhao}; our idea there was to combine the main result of \cite{SunZhao} with classification results of two-dimensional contracting germs available via valuation dynamics, to obtain a proof of Theorem~\ref{mainthm:Fatou} for $\bP^3$ under an additional tautness assumption as in \cite{SunZhao}. We used ChatGPT Astra to assemble such a three-dimensional proof, which turned out to be flawed. During the process we realized the importance of a false lemma inpired by the work of Astorg and Le (their papers have no issues). We have been correcting and upgrading that lemma, sometimes through extensive AI assistance, during every stage of our later work. It finally evolved into Section~\ref{sec:local-invariant}. At one stage we decided to use scheme-theoretic language to write these materials because of the need to play with singularities and with algebraic, holomorphic, and formal categories. Then the scheme-theoretic formalism made us realize that the use of two-dimensional germ classification or valuation dynamics is an overkill because postcritical finiteness could already provide easy constraints that may be hard for general local dynamics, and that some of the arguments can be written directly in algebro-geometric language. Then through extensive interaction with ChatGPT Astra, we come up with a formulation based on symbolic powers of ideal sheaves. This step finally allows our paper to work for arbitrary dimension, and becomes Section~\ref{sec:ideal-sheaves} and its preparational Section~\ref{sec:transverse}. During the writing of the paper, the authors also realized that the use of our previous work~\cite{SunZhao} is not essential and can be replaced by more elementary arguments; this becomes now Section~\ref{sec:preparation}. The two examples in Section~\ref{sec:geometry} were found by ChatGPT Astra, under simple prompts based on some unpublished manuscripts. We also used ChatGPT Astra for the editorial work. Every single line of the present paper is human written.

\bibliographystyle{amsplain}
\bibliography{biblio}

@article{AbateRaissy2013,
  author  = {Abate, Marco and Raissy, Jasmin},
  title   = {Formal {Poincar\'e--Dulac} renormalization for holomorphic germs},
  journal = {Discrete Contin. Dyn. Syst.},
  volume  = {33},
  number  = {5},
  year    = {2013},
  pages   = {1773--1807},
  doi     = {10.3934/dcds.2013.33.1773},
}

@misc{AbboudXie,
  author        = {Abboud, Marc and Xie, Junyi},
  title         = {Dynamical degrees of twisted rational maps},
  year          = {2026},
  eprint        = {2608.09275},
  archivePrefix = {arXiv},
  primaryClass  = {math.AG},
  howpublished  = {Preprint, \href{https://arxiv.org/abs/2608.09275}{arXiv:2608.09275}},
}

@book{Arnold1988,
  author    = {Arnold, Vladimir Igorevich},
  title     = {Geometrical Methods in the Theory of Ordinary Differential Equations},
  edition   = {Second},
  series    = {Grundlehren der mathematischen Wissenschaften},
  volume    = {250},
  publisher = {Springer-Verlag},
  address   = {New York},
  year      = {1988},
  doi       = {10.1007/978-1-4612-1037-5},
}

@article{Astorg2020,
  author     = {Astorg, Matthieu},
  title      = {Dynamics of post-critically finite maps in higher dimension},
  journal    = {Ergodic Theory Dynam. Systems},
  volume     = {40},
  number     = {2},
  year       = {2020},
  pages      = {289--308},
  doi        = {10.1017/etds.2018.32},
  url        = {https://doi.org/10.1017/etds.2018.32},
  issn       = {0143-3857,1469-4417},
  fjournal   = {Ergodic Theory and Dynamical Systems},
  mrreviewer = {Jeffrey\ Diller},
}

@article{Belitskii1978,
  author  = {Belitskii, Genrich Ruvimovich},
  title   = {Equivalence and normal forms of germs of smooth mappings},
  journal = {Russian Math. Surveys},
  volume  = {33},
  number  = {1},
  year    = {1978},
  pages   = {107--177},
  doi     = {10.1070/RM1978v033n01ABEH002237},
}

@incollection{BelkKoch2010,
  author    = {Belk, James and Koch, Sarah},
  title     = {Iterated monodromy for a two-dimensional map},
  editor    = {Bonk, Mario and Gilman, Jane and Masur, Howard and Minsky, Yair and Wolf, Michael},
  booktitle = {In the Tradition of Ahlfors--Bers. V},
  series    = {Contemporary Mathematics},
  volume    = {510},
  publisher = {American Mathematical Society},
  address   = {Providence, RI},
  year      = {2010},
  pages     = {1--11},
  doi       = {10.1090/conm/510/10013},
}

@article{BertelootLoeb2001,
  author  = {Berteloot, Fran{\c{c}}ois and Loeb, Jean-Jacques},
  title   = {Une caract{\'e}risation g{\'e}om{\'e}trique des exemples de {Latt{\`e}s} de {$\mathbb{P}^k$}},
  journal = {Bull. Soc. Math. France},
  volume  = {129},
  number  = {2},
  year    = {2001},
  pages   = {175--188},
  doi     = {10.24033/bsmf.2392},
}

@incollection{BonifantDabija2002,
  author    = {Bonifant, Araceli M. and Dabija, Marius},
  title     = {Self-maps of {$\mathbb{P}^2$} with invariant elliptic curves},
  editor    = {Earle, Clifford J. and Harvey, William J. and Recillas-Pishmish, Sevin},
  booktitle = {Complex Manifolds and Hyperbolic Geometry},
  series    = {Contemporary Mathematics},
  volume    = {311},
  publisher = {American Mathematical Society},
  address   = {Providence, RI},
  year      = {2002},
  pages     = {1--25},
  doi       = {10.1090/conm/311/05444},
}

@article{bost01,
  author     = {Bost, Jean-Beno\^{i}t},
  title      = {Algebraic leaves of algebraic foliations over number fields},
  journal    = {Publ. Math. Inst. Hautes {\'E}tudes Sci.},
  volume     = {93},
  year       = {2001},
  pages      = {161--221},
  doi        = {10.1007/s10240-001-8191-3},
  url        = {https://doi.org/10.1007/s10240-001-8191-3},
  issn       = {0073-8301,1618-1913},
  fjournal   = {Publications Math\'{e}matiques. Institut de Hautes \'{E}tudes Scientifiques},
  mrreviewer = {Carlo\ Gasbarri},
}

@article{BuffEpsteinKoch2012,
  author  = {Buff, Xavier and Epstein, Adam L. and Koch, Sarah},
  title   = {B{\"o}ttcher coordinates},
  journal = {Indiana Univ. Math. J.},
  volume  = {61},
  number  = {5},
  year    = {2012},
  pages   = {1765--1799},
  doi     = {10.1512/iumj.2012.61.4981},
}

@article{Chaperon2004,
  author   = {Chaperon, Marc},
  title    = {Stable manifolds and the {Perron}-{Irwin} method},
  journal  = {Ergodic Theory Dynam. Systems},
  volume   = {24},
  number   = {5},
  year     = {2004},
  pages    = {1359--1394},
  doi      = {10.1017/S0143385703000701},
  issn     = {0143-3857},
  fjournal = {Ergodic Theory and Dynamical Systems},
  zbl      = {1087.37023},
  zbmath   = {5019134},
}

@article{ColliotTheleneOjanguren,
  author  = {Colliot-Th{\'e}l{\`e}ne, Jean-Louis and Ojanguren, Manuel},
  title   = {Vari{\'e}t{\'e}s unirationnelles non rationnelles: au-del{\`a} de l'exemple d'{Artin} et {Mumford}},
  journal = {Invent. Math.},
  volume  = {97},
  number  = {1},
  year    = {1989},
  pages   = {141--158},
  doi     = {10.1007/BF01850658},
  issn    = {0020-9910},
}

@book{ColliotTheleneSkorobogatovBook,
  author    = {Colliot-Th{\'e}l{\`e}ne, Jean-Louis and Skorobogatov, Alexei N.},
  title     = {The {Brauer}-{Grothendieck} group},
  series    = {Ergebnisse der Mathematik und ihrer Grenzgebiete. 3. Folge},
  volume    = {71},
  publisher = {Springer},
  address   = {Cham},
  year      = {2021},
  doi       = {10.1007/978-3-030-74248-5},
}

@article{Crass2007Klein,
  author  = {Crass, Scott},
  title   = {Solving the heptic by iteration in two dimensions: Geometry and dynamics under {Klein}'s group of order 168},
  journal = {J. Mod. Dyn.},
  volume  = {1},
  number  = {2},
  year    = {2007},
  pages   = {175--203},
  doi     = {10.3934/jmd.2007.1.175},
}

@article{Crass2011Valentiner,
  author  = {Crass, Scott},
  title   = {New light on solving the sextic by iteration: An algorithm using reliable dynamics},
  journal = {J. Mod. Dyn.},
  volume  = {5},
  number  = {2},
  year    = {2011},
  pages   = {397--408},
  doi     = {10.3934/jmd.2011.5.397},
}

@article{DabijaJonsson2010,
  author  = {Dabija, Marius and Jonsson, Mattias},
  title   = {Algebraic webs invariant under endomorphisms},
  journal = {Publ. Mat.},
  volume  = {54},
  number  = {1},
  year    = {2010},
  pages   = {137--148},
  doi     = {10.5565/PUBLMAT_54110_07},
}

@article{DeyGorlachKaihnsa2020,
  author  = {Dey, Papri and G{\"o}rlach, Paul and Kaihnsa, Nidhi},
  title   = {Coordinate-wise powers of algebraic varieties},
  journal = {Beitr. Algebra Geom.},
  volume  = {61},
  number  = {3},
  year    = {2020},
  pages   = {473--505},
  doi     = {10.1007/s13366-019-00481-8},
}

@misc{DFR,
  author        = {Dujardin, Romain and Favre, Charles and Ruggiero, Matteo},
  title         = {Polynomial skew products with small relative degree},
  year          = {2025},
  eprint        = {2507.09197},
  archivePrefix = {arXiv},
  primaryClass  = {math.DS},
  howpublished  = {Preprint, \href{https://arxiv.org/abs/2507.09197}{arXiv:2507.09197}},
  url           = {https://arxiv.org/abs/2507.09197},
}

@article{Dupont2003,
  author  = {Dupont, Christophe},
  title   = {Exemples de {Latt{\`e}s} et domaines faiblement sph{\'e}riques de {$\mathbb{C}^n$}},
  journal = {Manuscripta Math.},
  volume  = {111},
  number  = {3},
  year    = {2003},
  pages   = {357--378},
  doi     = {10.1007/s00229-003-0378-0},
}

@article{ELS2001,
  author  = {Ein, Lawrence and Lazarsfeld, Robert and Smith, Karen E.},
  title   = {Uniform bounds and symbolic powers on smooth varieties},
  journal = {Invent. Math.},
  volume  = {144},
  number  = {2},
  year    = {2001},
  pages   = {241--252},
  doi     = {10.1007/s002220100121},
}

@incollection{FornaessSibony1992,
  author    = {Forn{\ae}ss, John Erik and Sibony, Nessim},
  title     = {Critically finite rational maps on {$\mathbb{P}^2$}},
  editor    = {Nagel, Alexander and Stout, Edgar Lee},
  booktitle = {The Madison Symposium on Complex Analysis},
  series    = {Contemporary Mathematics},
  volume    = {137},
  publisher = {American Mathematical Society},
  address   = {Providence, RI},
  year      = {1992},
  pages     = {245--260},
  doi       = {10.1090/conm/137/1190986},
}

@incollection{FornaessSibony,
  author    = {Forn{\ae}ss, John Erik and Sibony, Nessim},
  title     = {Complex dynamics in higher dimension. {I}},
  editor    = {Camacho, C{\'e}sar and Lins Neto, Alcides and Moussu, Robert and Sad, Paulo},
  booktitle = {Complex analytic methods in dynamical systems (IMPA, January 1992)},
  series    = {Ast{\'e}risque},
  number    = {222},
  publisher = {Soci{\'e}t{\'e} Math{\'e}matique de France},
  address   = {Paris},
  year      = {1994},
  pages     = {201--231},
  doi       = {10.24033/ast.259},
}

@article{GauthierHutzKaschner2023,
  author  = {Gauthier, Thomas and Hutz, Benjamin and Kaschner, Scott},
  title   = {Symmetrization of rational maps: Arithmetic properties and families of {Latt{\`e}s} maps of {$\mathbb{P}^k$}},
  journal = {Conform. Geom. Dyn.},
  volume  = {27},
  year    = {2023},
  pages   = {98--117},
  doi     = {10.1090/ecgd/382},
}

@article{GTV,
  author  = {Gauthier, Thomas and Taflin, Johan and Vigny, Gabriel},
  title   = {Sparsity of postcritically finite maps of {$\mathbb{P}^k$} and beyond: A complex analytic approach},
  journal = {Publ. Math. Inst. Hautes {\'E}tudes Sci.},
  volume  = {143},
  year    = {2026},
  pages   = {1--96},
  doi     = {10.5802/pmihes.1},
}

@book{GrauertRemmert,
  author    = {Grauert, Hans and Remmert, Reinhold},
  title     = {Coherent Analytic Sheaves},
  series    = {Grundlehren der mathematischen Wissenschaften},
  volume    = {265},
  publisher = {Springer-Verlag},
  address   = {Berlin},
  year      = {1984},
  doi       = {10.1007/978-3-642-69582-7},
}

@article{hart68,
  author     = {Hartshorne, Robin},
  title      = {Cohomological dimension of algebraic varieties},
  journal    = {Ann. of Math. (2)},
  volume     = {88},
  number     = {3},
  year       = {1968},
  pages      = {403--450},
  doi        = {10.2307/1970720},
  url        = {https://doi.org/10.2307/1970720},
  issn       = {0003-486X},
  fjournal   = {Annals of Mathematics. Second Series},
  mrreviewer = {Y.\ Nakai},
}

@book{Hartshorne,
  author    = {Hartshorne, Robin},
  title     = {Algebraic Geometry},
  series    = {Graduate Texts in Mathematics},
  volume    = {52},
  publisher = {Springer-Verlag},
  address   = {New York},
  year      = {1977},
  doi       = {10.1007/978-1-4757-3849-0},
}

@article{HartshornePolini2015,
  author  = {Hartshorne, Robin and Polini, Claudia},
  title   = {Divisor class groups of singular surfaces},
  journal = {Trans. Amer. Math. Soc.},
  volume  = {367},
  number  = {9},
  year    = {2015},
  pages   = {6357--6385},
  doi     = {10.1090/S0002-9947-2015-06228-X},
}

@article{Hironaka,
  author  = {Hironaka, Heisuke},
  title   = {Resolution of Singularities of an Algebraic Variety over a Field of Characteristic Zero. {I, II}},
  journal = {Ann. of Math. (2)},
  volume  = {79},
  year    = {1964},
  pages   = {109--203, 205--326},
}

@misc{Ingram2020MinimallyCritical,
  author        = {Ingram, Patrick},
  title         = {Minimally critical endomorphisms of {$\mathbb{P}^N$}},
  year          = {2020},
  eprint        = {2006.12869},
  archivePrefix = {arXiv},
  primaryClass  = {math.NT},
  howpublished  = {Preprint, \href{https://arxiv.org/abs/2006.12869}{arXiv:2006.12869}},
  note          = {Version 3, 2022},
  doi           = {10.48550/arXiv.2006.12869},
  url           = {https://arxiv.org/abs/2006.12869},
}

@article{Ji,
  author  = {Ji, Zhuchao},
  title   = {Structure of {Julia} sets for post-critically finite endomorphisms on {$\mathbb{P}^2$}},
  journal = {Math. Ann.},
  volume  = {384},
  number  = {1--2},
  year    = {2022},
  pages   = {857--880},
  doi     = {10.1007/s00208-021-02302-7},
  issn    = {0025-5831},
}

@article{Jonsson1998,
  author     = {Jonsson, Mattias},
  title      = {Some properties of {$2$}-critically finite holomorphic maps of {$\mathbb{P}^2$}},
  journal    = {Ergodic Theory Dynam. Systems},
  volume     = {18},
  number     = {1},
  year       = {1998},
  pages      = {171--187},
  doi        = {10.1017/S0143385798097521},
  url        = {https://doi.org/10.1017/S0143385798097521},
  issn       = {0143-3857,1469-4417},
  fjournal   = {Ergodic Theory and Dynamical Systems},
  mrreviewer = {Dan\ Coman},
}

@article{KebekusSolaCondeToma2007,
  author  = {Kebekus, Stefan and Sol{\'a} Conde, Luis and Toma, Matei},
  title   = {Rationally connected foliations after {Bogomolov} and {McQuillan}},
  journal = {J. Algebraic Geom.},
  volume  = {16},
  number  = {1},
  year    = {2007},
  pages   = {65--81},
  doi     = {10.1090/S1056-3911-06-00435-8},
  url     = {https://arxiv.org/abs/math/0505222},
}

@article{Koch2013,
  author   = {Koch, Sarah},
  title    = {Teichm\"{u}ller theory and critically finite endomorphisms},
  journal  = {Adv. Math.},
  volume   = {248},
  year     = {2013},
  pages    = {573--617},
  doi      = {10.1016/j.aim.2013.08.019},
  url      = {https://doi.org/10.1016/j.aim.2013.08.019},
  issn     = {0001-8708,1090-2082},
  fjournal = {Advances in Mathematics},
}

@book{Lazarsfeld2,
  author    = {Lazarsfeld, Robert},
  title     = {Positivity in algebraic geometry {II}: Positivity for vector bundles, and multiplier ideals},
  series    = {Ergebnisse der Mathematik und ihrer Grenzgebiete. 3. Folge},
  volume    = {49},
  publisher = {Springer},
  address   = {Berlin},
  year      = {2004},
  doi       = {10.1007/978-3-642-18810-7},
  isbn      = {3-540-22534-X},
  issn      = {0071-1136},
  fseries   = {Ergebnisse der Mathematik und ihrer Grenzgebiete. 3. Folge},
}

@article{Le2020,
  author  = {Le, Van Tu},
  title   = {Periodic points of post-critically algebraic holomorphic endomorphisms},
  journal = {Ergodic Theory Dynam. Systems},
  volume  = {42},
  number  = {7},
  year    = {2022},
  pages   = {2382--2414},
  doi     = {10.1017/etds.2021.48},
}

@article{Le2023,
  author   = {Le, Van Tu},
  title    = {Periodic points of weakly post-critically finite all the way down maps},
  journal  = {Ann. Mat. Pura Appl. (4)},
  volume   = {202},
  number   = {3},
  year     = {2023},
  pages    = {1187--1195},
  doi      = {10.1007/s10231-022-01275-x},
  url      = {https://doi.org/10.1007/s10231-022-01275-x},
  issn     = {0373-3114,1618-1891},
  fjournal = {Annali di Matematica Pura ed Applicata. Series IV},
}

@book{Milnor2006,
  author    = {Milnor, John},
  title     = {Dynamics in One Complex Variable},
  edition   = {Third},
  series    = {Annals of Mathematics Studies},
  volume    = {160},
  publisher = {Princeton University Press},
  address   = {Princeton, NJ},
  year      = {2006},
  doi       = {10.1515/9781400835539},
  isbn      = {978-0-691-12488-9},
}

@unpublished{NieZhao,
  author = {Nie, Hongming and Zhao, Shengyuan},
  title  = {Commuting skew products and twisted rational maps},
  note   = {Preprint},
  year   = {2026},
}

@article{Rong2008,
  author  = {Rong, Feng},
  title   = {The {Fatou} set for critically finite maps},
  journal = {Proc. Amer. Math. Soc.},
  volume  = {136},
  number  = {10},
  year    = {2008},
  pages   = {3621--3625},
  doi     = {10.1090/S0002-9939-08-09358-1},
}

@article{Rong2010Lattes,
  author  = {Rong, Feng},
  title   = {{Latt{\`e}s} maps on {$\mathbb{P}^2$}},
  journal = {J. Math. Pures Appl. (9)},
  volume  = {93},
  number  = {6},
  year    = {2010},
  pages   = {636--650},
  doi     = {10.1016/j.matpur.2009.10.002},
}

@article{RudinStout1969,
  author  = {Rudin, Walter and Stout, Edgar Lee},
  title   = {Modules over polydisc algebras},
  journal = {Trans. Amer. Math. Soc.},
  volume  = {138},
  year    = {1969},
  pages   = {327--342},
  doi     = {10.1090/S0002-9947-1969-0241695-2},
}

@book{Shub1987,
  author    = {Shub, Michael},
  title     = {Global stability of dynamical systems},
  publisher = {Springer-Verlag},
  address   = {New York},
  year      = {1987},
  note      = {With the collaboration of Albert Fathi and R{\'e}mi Langevin. Translated from the French by Joseph Christy},
  doi       = {10.1007/978-1-4757-1947-5},
  zbl       = {0606.58003},
  zbmath    = {3980052},
}

@unpublished{PCF_II1/2,
  author = {Sun, Ruiran and Zhao, Shengyuan},
  title  = {Postcritically finite endomorphisms {II1/2}: an example of non rational core},
  note   = {Unpublished notes (authors' website)},
}

@misc{SunZhao,
  author        = {Sun, Ruiran and Zhao, Shengyuan},
  title         = {Postcritically finite endomorphisms {I}: {Kobayashi} hyperbolicity and tautness},
  year          = {2026},
  eprint        = {2610.11817},
  archivePrefix = {arXiv},
  primaryClass  = {math.DS},
  howpublished  = {Preprint, \href{https://arxiv.org/abs/2610.11817}{arXiv:2610.11817}},
}

@incollection{Ueda1993,
  author    = {Ueda, Tetsuo},
  title     = {Complex dynamical systems on projective spaces},
  editor    = {Ushiki, Shigehiro},
  booktitle = {Chaotic Dynamical Systems},
  series    = {Advanced Series in Dynamical Systems},
  volume    = {13},
  publisher = {World Scientific},
  address   = {Singapore},
  year      = {1993},
  pages     = {120--138},
}

@article{Ueda1998,
  author  = {Ueda, Tetsuo},
  title   = {Critical orbits of holomorphic maps on projective spaces},
  journal = {J. Geom. Anal.},
  volume  = {8},
  number  = {2},
  year    = {1998},
  pages   = {319--334},
  doi     = {10.1007/BF02921645},
}

@article{Ueno2007Symmetric,
  author  = {Ueno, Kohei},
  title   = {Dynamics of symmetric holomorphic maps on projective spaces},
  journal = {Publ. Mat.},
  volume  = {51},
  number  = {2},
  year    = {2007},
  pages   = {333--344},
  doi     = {10.5565/PUBLMAT_51207_04},
}

@article{Zariski_normal,
  author    = {Zariski, Oscar},
  title     = {Sur la normalit\'e analytique des vari\'et\'es normales},
  journal   = {Ann. Inst. Fourier (Grenoble)},
  volume    = {2},
  year      = {1950},
  pages     = {161--164},
  doi       = {10.5802/aif.27},
  address   = {Grenoble},
  publisher = {Institut Fourier},
}

@article{ZhaoCentralizer,
  author  = {Zhao, Shengyuan},
  title   = {Centralizers of elements of infinite order in plane {Cremona} groups},
  journal = {Ann. Sc. Norm. Super. Pisa Cl. Sci. (5)},
  volume  = {23},
  number  = {2},
  year    = {2022},
  pages   = {915--957},
  doi     = {10.2422/2036-2145.201908_007},
}
\end{document}